\documentclass[12pt]{article}
\usepackage{latexsym,amsfonts,amssymb}
\usepackage[dvips]{color}
\usepackage{colordvi,multicol}
\usepackage{amsmath}
\usepackage{graphicx}
\usepackage{pdflscape}
\usepackage{rotating}

\usepackage{latexsym,amsfonts,amssymb}
\usepackage[dvips]{color}
\usepackage{colordvi,multicol}
\usepackage{amsmath}
\usepackage{graphicx}
\usepackage{pdflscape}
\usepackage{rotating}
\usepackage{breqn}

\def \cal{\mathcal}

\newtheorem{thm}{Theorem}[section]

\newtheorem{pro}[thm]{Proposition}

\newtheorem{rem}[thm]{Remark}

      \def\be{\begin{equation}} 
      \def\ee{\end{equation}} 
      \def\beqn{\begin{eqnarray}} 
      \def\eeqn{\end{eqnarray}} 
      \def\beq{\begin{eqnarray*}} 
      \def\eeq{\end{eqnarray*}}
      
      \def\ga{\gamma} 
      \def\ep{\varepsilon} 
      \def\ra{\rightarrow} 

\date{}
\begin{document}

\title{\bf  Dynamical models for the two-parameter Poisson-Dirichlet distribution and the Pitman-Yor process}
\author{}

\maketitle
\date{}

 \centerline{Shui Feng} \centerline{\small Department of Mathematics and Statistics}
 \centerline{\small McMaster
 University}
\centerline{\small Hamilton, L8S 4K1, Canada} \centerline{\small
E-mail: shuifeng@mcmaster.ca} \vskip 1cm \centerline{Wei Sun}
\centerline{\small Department of Mathematics and Statistics}
 \centerline{\small Concordia University} \centerline{\small Montreal, H3G 1M8,
Canada} \centerline{\small E-mail: wei.sun@concordia.ca}

\vskip 1cm
\maketitle
\begin{abstract}
\noindent The Pitman-Yor process is a discrete labelled random measure that has been used as a prior in Bayesian nonparametrics for the study of groups of clustered data exhibiting power law behaviour. The masses of the process in descending order follow the unlabelled two-parameter Poisson-Dirichlet distribution. They are natural generalizations of their one-parameter counterparts: Ferguson's Dirichlet process and Kingman's Poisson-Dirichlet distribution. There have been well established studies of dynamical models associated with the Dirichlet process and the Poisson-Dirichlet distribution.   In this paper, we introduce and study a family of diffusion processes associated with the Pitman-Yor process and the two-parameter Poisson-Dirichlet distribution.  The diffusion coefficients indexed by a non-negative parameter $\ga$ are smaller than the corresponding one-parameter models in terms of quadratic forms or bilinear forms when $\ga $ is positive.   The well known Petrov's diffusion \cite{P} corresponds to $\ga =0$ among the unlabelled diffusions.  If $\ga$ is the same as the stable parameter $\alpha$ in the Pitman-Yor process, we obtain both labelled and unlabelled reversible diffusion processes with the Pitman-Yor process and the two-parameter Poisson-Dirichlet distribution as the corresponding reversible measures. We construct these  processes analytically through Dirichlet forms. In comparison with existing models in the literature, our models possess two fundamental new features. Firstly, our labelled model is characterized by an explicit generator, which is the first among all models studied so far. This makes it possible to establish the crucial integration-by-parts formula. Secondly, a novel foundational structure of the two-parameter distributions is the existence of a diversity index (a multiple of local time) for the positive stable parameter.  By slowing down the diffusion in our model, the essential role of the diversity index is revealed in the evolution of the population.  Additionally, we also obtain properties including ergodicity, path behaviour,  and finite dimensional approximations.
\end{abstract}

\section{Introduction}

For any $0\leq \alpha<1$,  $\theta+\alpha>0$, let $U_1, U_2, \ldots$ be a sequence of independent beta random variables with $U_n$ having the Beta$(1-\alpha,\theta+n\alpha)$ distribution. Set
\[
V_1=U_1,\ V_n=U_n\prod_{i=1}^{n-1}(1-U_i), \  n\geq 2.
\]
Let $S$ be a Polish space with metric $d$. Denote by ${\cal P}_1(S)$ the space of all probability measures on $S$ with weak topology. Let $\nu_0\in{\cal P}_1(S)$ be diffuse, i.e., $\nu_0(\{s\})=0$ for all $s$ in $S$, and $\xi_1, \xi_2, \ldots$ be i.i.d. with common distribution $\nu_0$. Assume that $(V_1, V_2, \ldots)$ and $(\xi_1, \xi_2, \ldots)$  are independent. Define the random measure
\[
\Xi_{\alpha, \theta,\nu_0}=\sum_{i=1}^\infty V_i \delta_{\xi_i}.
\]
The distribution of the descending order statistics $(P_1, P_2, \ldots)$ of $(V_1, V_2, \ldots)$ is called the {\it two-parameter Poisson-Dirichlet distribution}, denoted by PD$(\alpha,\theta)$; the law of $\Xi_{\alpha, \theta,\nu_0}$, denoted by $\Pi_{\alpha,\theta,\nu_0}$, is called the {\it two-parameter Dirichlet process} or the {\it Pitman-Yor process} with concentration parameter $\theta$, discount parameter $\alpha$, and base distribution $\nu_0$.

In addition to the above construction, the two-parameter Poisson-Dirichlet distribution and the Pitman-Yor process can also be constructed from subordinators (\cite{Kingman75,  PY}), the P\'olya urn scheme (\cite{BlaMac73, Hop84, FeHo98}), and Dubins and Pitman's Chinese restaurant process (\cite{al85}).

When $\alpha=0$, the distribution of $(V_1, V_2, \ldots)$ is called the GEM distribution in Ewens \cite{Ewen04}, named after Griffiths \cite{Gri80a}, Engen \cite{engen78}, and McCloskey \cite{mcc65} for their pioneering work on the structure. The descending order statistics $(P_1, P_2, \ldots)$ of $(V_1, V_2, \ldots)$ is  Kingman's Poisson-Dirichlet distribution with parameter $\theta$.  The law $\Pi_{0,\theta,\nu_0}$ is Ferguson's Dirichlet process (\cite{Fer73}) with concentration parameter $\theta$ and base distribution $\nu_0$ (\cite{Kingman75}).
We will call the distribution of $(V_1, V_2, \ldots)$ the two-parameter GEM distribution for positive $\alpha$. The GEM distribution can be recovered from the Poisson-Dirichlet distribution through the {\it size-biased permutation} (\cite{mcc65}).  This relation is maintained between the two-parameter GEM distribution and the two-parameter Poisson-Dirichlet distribution. In fact, the relation characterizes the two-parameter Poisson-Dirichlet distribution \cite{Pitman96}.

 The Dirichlet process, the GEM distribution, and the Poisson-Dirichlet distribution arise naturally in many fields including but not limited to Bayesian statistics, ecology, economics, finance,  machine learning, physics, and population genetics. In Bayesian statistics, the Dirichlet process serves as a fundamental prior on the space of probabilities. It plays an important role in cluster analysis, and is used in infinite mixture model  inferring the number of groups in the data (\cite{GV17}).

In population genetics, the Dirichlet process and the Poisson-Dirichlet distribution describe the equilibrium behaviour of a population that evolves under the  parent-independent  mutation and resampling.  More specifically, let
$$
{\nabla}_{\infty}=\bigg\{x=(x_1,x_2,\dots)\in \mathbb{R}^{\infty}:x_1\ge x_2\ge\cdots\ge 0\ {\rm
and}\ \sum_{i=1}^{\infty}x_i= 1\bigg\}
$$
denote the infinite dimensional simplex, and equip it with the product topology.
In \cite{EK2},  a $\nabla_{\infty}$-valued diffusion process $\bf{X}(\cdot)$,  the infinitely-many-neutral-alleles model, is constructed. The generator on appropriate domain has the form
\vspace{-1mm}
\[
\frac{1}{2}\left[\sum_{i,j=1}^{\infty}x_i(\delta_{ij}-x_j)\partial_i\partial_j -\sum_{i=1}^\infty \theta x_i \partial_i\right],
\]
and the process is reversible with reversible measure PD$(0,\theta)$. Since $\bf{X}(\cdot)$ contains only the information of the descending proportions of types in the population, it is an  {\it unlabelled model}.

When the evolution of types is included, one is led to a {\it labelled model}, the reversible Fleming-Viot process with parent-independent mutation rate $\theta$ (\cite{EK}), generated by
\[
\frac{1}{2}\int_S \int_S \frac{\delta^2 F(\mu)}{\delta
\mu(s) \delta \mu(r)}\left(\delta_s(dr)-\mu(dr)\right)\mu(ds)+ \int_S  A \Bigl(\frac{\delta F(\mu)}{\delta \mu(\cdot)}\Bigr)(s)\mu(ds),
\]
where  $F$ is in appropriate domain, $A$ is the parent-independent mutation operator, $\delta_s$ is the  Dirac measure at $s$, and
$$
\frac{\delta F(\mu)}{\delta \mu(s)} = \lim_{\ep \ra 0+ }
\frac{F(\mu + \ep \delta_s)-F(\mu)}{\ep},\ \ \ \ s\in S,\, \mu\in {\cal P}_1(S).
$$
The reversible measure of the Fleming-Viot process is the Dirichlet process $\Pi_{0,\theta,\nu_0}$.  Furthermore it is known  \cite{lsy99}  that, among
the family of all Fleming-Viot processes, the process with the parent-independent mutation is the only one with reversible measure $\Pi_{0,\theta,\nu_0}$.
The Poisson-Dirichlet distribution PD$(0,\theta)$ and the GEM distribution are also characterized as the stationary distribution for certain split-merge Markov chains in \cite{GK01, Pitman2002, Dia-04}.

Many results in the case $\alpha=0$ have natural generalizations to the case $0<\alpha<1$. These include Pitman's sampling formula, the stick-breaking representation, the Chinese restaurant process, and the explicit posterior distributions. But significant differences do exist.  For one thing, the partition property of the Dirichlet process does not hold for the Pitman-Yor process. From Kingman's subordinator representation \cite{Kingman75}, it is clear that $\theta$ and $\alpha$ correspond to the Gamma and stable components, respectively. Thus in Bayesian inferences, the case $\alpha=0$ is more appropriate for modelling data exhibiting clusters of logarithmic scale while the positive $\alpha$ model is suitable for data exhibiting clusters of power law scale.

 There have been intensive work on generalizing the infinitely-many-neutral-alleles model and the Fleming-Viot process with parent-independent mutation to the positive $\alpha$ setting. Tremendous progress has been made in the past two decades.  For unlabelled models, a reversible coagulation-fragmentation Markov process is constructed in \cite{berj08} with the two-parameter Poisson-Dirichlet distribution as the reversible measure.  The GEM  process is constructed and studied in \cite{fenwang07} where the reversible distribution is  the two-parameter GEM distribution. Using the methods developed in \cite{B-O09}, Petrov \cite{P} constructed a $\nabla_{\infty}$-valued diffusion with reversible measure PD$(\alpha,\theta)$. It is considered a natural generalization of  the infinitely-many-neutral-alleles model, and has been well studied in \cite{FS0,FSWX,Ethier14,  CBERS,FS,  Forman-etc23}.  The generator of Petrov's diffusion on appropriate domain has the form
\vspace{-1mm}
\[
\frac{1}{2}\left[\sum_{i,j=1}^{\infty}x_i(\delta_{ij}-x_j)\partial_i\partial_j -\sum_{i=1}^\infty (\alpha+\theta x_i) \partial_i\right].
\]
It is worth noting that the diffusion coefficient is the same as the infinitely-many-neutral-alleles model. The only difference is the $\alpha $ term in the drift.

The construction of the labelled models turns out to be much more complicated. The first model was constructed in \cite{fenwang07}  using the techniques of Dirichlet forms. The model is not natural in the sense that  the masses and the labels are assumed to move independently according to the GEM process and an ergodic process. In \cite{FS}, a diffusion process is constructed for $\alpha=\frac{1}{2}$ and  finite type space $S$. Recently,  using the interval-valued diffusions in \cite{Forman-etc21}, a measure-valued diffusion was constructed in \cite{Forman-etc22}. The process was obtained by normalization and random time change of a branching process with immigration. The Pitman-Yor process is the stationary distribution of the process and the projection of the process to the ordered infinite dimensional simplex is Petrov's diffusion. This makes it a valuable model for the two-parameter generalization of the Fleming-Viot process.  But the generator of the process is not explicit, and  the evolution of labels in the model is less random than the one-parameter Fleming-Viot process. In fact, the existing labels in the process do not change in time unless the corresponding masses become zero. New independent labels selected at random from the base distribution are added to the population after a branching event.  This is different from the one-parameter Fleming-Viot process, where both masses and labels evolve continuously in time as individuals in the population undergoing mutation and resampling.

In this paper, we introduce and study a family of unlabelled and labelled models.  The diffusion coefficients are indexed by a non-negative parameter $\ga$.  For unlabelled models, the diffusion matrix has  the form
\[
\left(\frac{1}{2} x_i\left[x_i^{\ga}\delta_{ij}+x_j(H_{1+\ga}(x)-x^{\ga}_i-x^{\ga}_j)\right]\right)_{1\leq i,j <\infty},
\]
where $H_{1+\ga}(x)=\sum_{i=1}^\infty x_i^{1+\ga}$.
 The case $\ga=0$ corresponds to Petrov's diffusion \cite{P}.  When $\ga=\alpha$, our new unlabelled model  has the  generator
\vspace{-1mm}
\beq
&&\frac{1}{2}\bigg\{\sum\limits_{i,j=1}^\infty x_i\bigg[x_i^{\alpha}\delta_{ij}+x_j\left(H_{1+\alpha}(x)-x^{\alpha}_i-x^{\alpha}_j\right)\bigg]\partial_i\partial_j\\
&&\ \ \ \ \ \ \ \ -\sum\limits_{i=1}^\infty x_i\bigg[(\theta+\alpha+1)\left(x_{i}^{\alpha}-H_{1+\alpha}(x)\right)+\alpha L(x)\bigg]\partial_i\bigg\},
\eeq
where $L(x)$ under  PD$(\alpha,\theta)$  is the $\alpha$-diversity and is related to the local time of certain Markov process.  In particular, under  PD$(\alpha,0)$, $L(x)$ is a multiple of the Mittag-Leffler random variable. In this model, both the diffusion coefficient and the drift depend on $\alpha$. As $\alpha$ approaches zero, we obtain the infinitely-many-neutral-alleles model.

The generator for the labelled model is notationally more involved, and the explicit form is given in Section 3.  As in the unlabelled case, both the diffusion coefficient and the drift depend on $\alpha$, and when $\alpha$ approaches zero the model becomes the Fleming-Viot process with parent-independent mutation.  Our labelled and unlabelled models are also linked through  the projection from ${\cal P}_1(S)$ to $\nabla_{\infty}$.
  The crucial steps in the construction of both the labelled and unlabelled models are the establishment of  integration-by-parts formulas.  These rely heavily on the explicit forms of the corresponding generators.

The remaining of the paper is organized as follows. In Section 2, we construct the unlabelled reversible diffusion associated with the two-parameter Poisson-Dirichlet distribution.  It provides an alternate generalization to the infinitely-many-neutral-alleles model.   The diffusion coefficient in our model is smaller than its one-parameter counterparts and Petrov's diffusion in terms of the Loewner partial order of bilinear forms.  By leveraging the slow mixing property of our model, we explicitly reveal the evolutionary role of \(\alpha \)-diversity,  thereby capturing the dynamical aspects of this fundamental feature of the two-parameter Poisson-Dirichlet distribution. To prove the main result of Section 2, we need to generalize the moments formulas  for ${\rm PD}(\alpha,\theta)$ to fractional powers including the Ewens-Pitman sampling formula. The process constructed is shown to be  ergodic, and starting from any point in  $\nabla_{\infty}$ the process will never leave $\nabla_{\infty}$. We recover the infinitely-many-neutral-alleles model by letting  $\alpha$ tend to zero. In Section 3, we introduce and construct the labelled diffusion associated with the Pitman-Yor process. Here both the masses and their locations evolve in time. The explicit form of the model's generator is derived through extensive trial and error, heavily motivated by the special case where \(\alpha=\frac{1}{2}\). With the generator identified explicitly, we are able to establish the integration-by-parts formula and the closability of the corresponding pre-Dirichlet forms.  By projecting  labelled models  to $\nabla_{\infty}$, we obtain the Dirichlet forms for the unlabelled model in Section 2. We are also able to generalize our construction for the Pitman-Yor process with symmetric selections.  In Section 4, we consider the finite dimensional approximations to the infinite dimensional model in the case $\alpha=\frac{1}{2}$. A finite dimensional Markov chain model is constructed to illustrate the population evolution at the individual level.  This is in parallel to the Wright-Fisher construction of the Fleming-Viot process.   The speciality of the case $\alpha=\frac{1}{2}$  is due to  the fact that the finite dimensional projection of $\Pi_{\alpha,\theta,\nu_0}$ has an explicit density function. Finally, in Section 5, we focus on two issues. Firstly, we investigate the family of diffusion processes with diffusion coefficient index $\ga \geq 1$. These will play a key role in establishing ergodicity of the process in Section 2. Secondly, we consider the model \cite{FS} where $\ga=0$, $\alpha=\frac{1}{2}$, and  show the existence of non-trivial diffusion processes in the Mosco convergence. This is a labelled diffusion model with reversible measure $\Pi_{\frac{1}{2}, \theta,\nu_0}$. Its projection to the ordered simplex is Petrov's diffusion.  But the generator is not known, and its  domain turns out to be much more restrictive than the model in Section 3. This may partially explain the difficulty in obtaining the explicit form of the generator.

We establish many of our proofs using the machinery of Dirichlet forms. For a comprehensive overview of the fundamental concepts, terminology, and key results in the theory of Dirichlet forms, we refer the reader to \cite{Fuku2, MR,CF,  S}.

\section{The infinite dimensional unlabelled model}\setcounter{equation}{0}

In this section, we introduce and study a class of infinite dimensional unlabelled models generalizing Ethier-Kurtz's infinitely-many-neutral-alleles model. The main results include  the construction of the process and some path properties. One important tool in the proof is the generalized Ewens-Pitman sampling formula
for fractional powers.

We fix  $\alpha\in (0,1)$ and $\theta>-\alpha$ throughout this section.  Let $\mathbb{N}$ denote the set of positive integers. The compact space
$$
\overline{\nabla}_{\infty}=\bigg\{x=(x_1,x_2,\dots)\in \mathbb{R}^{\infty}:x_1\ge x_2\ge\cdots\ge 0\ {\rm
and}\ \sum_{i=1}^{\infty}x_i\le 1\bigg\}
$$
is the closure of ${\nabla}_{\infty}$ in $\mathbb{R}^{\infty}$.
 For $x\in \overline{\nabla}_{\infty}$ and $\beta\ge 1$, define
$$
H_{\beta}(x)=\sum_{i=1}^{\infty}x_i^{\beta},
$$
\begin{eqnarray}\label{May5D1}
\left\{
\begin{array}{ll}
a_{ii}(x)=(1-2x_i)x_i^{1+\alpha}+x_i^2H_{1+\alpha}(x),&\ \  i\in\mathbb{N},\\
a_{ij}(x)=x_ix_j\left[H_{1+\alpha}(x)-x^{\alpha}_i-x^{\alpha}_j\right],&\ \ i\not=j,\ i,j\in\mathbb{N},
\end{array}
\right.
\end{eqnarray}
and
\begin{eqnarray*}
L(x)=\left\{
\begin{array}{ll}
\liminf\limits_{i\rightarrow\infty}ix_i^{\alpha},&\ \ {\rm if}\ \liminf\limits_{i\rightarrow\infty}ix_i^{\alpha}<\infty,\\
0,&\ \ {\rm otherwise}.
\end{array}
\right.
\end{eqnarray*}

For notational convenience, we will write $H_\beta$ for $H_\beta(x)$ if it is clear from the context. Similarly, we will write
\[
a=a(x)=(a_{ij}(x))_{1\leq i,j<\infty},
\]
and drop the domain variable $x$ from time to time for functions on $\overline{\nabla}_{\infty}$ when the context is clear. Let ${\cal P}$ denote the algebra generated by $1,H_2(x),H_3(x),\dots$. For $u\in {\cal P}$, define
\begin{eqnarray}\label{May5A}
{\cal G}u(x)&=&\frac{1}{2}\sum\limits_{i,j=1}^\infty a_{ij}(x)\partial_i\partial_ju(x)\nonumber\\
&&-\frac{1}{2}\sum\limits_{i=1}^\infty x_i\bigg\{ (\theta+\alpha+1)\left[x_{i}^{\alpha}-H_{1+\alpha}(x)\right]+\alpha L(x)\bigg\}\partial_iu(x),\ \ \ \ x\in {\overline{\nabla}}_{\infty}.
\end{eqnarray}
Consider the following symmetric form on $L^2({\overline{\nabla}}_{\infty};{\rm PD}(\alpha,\theta))$:
\begin{eqnarray*}
{\cal A}(u,v)&=&\frac{1}{2}\int_{{\overline{\nabla}}_{\infty}}\langle a\nabla u,\nabla v\rangle \,d{\rm PD}(\alpha,\theta)\\
&=&\frac{1}{2}\sum_{i,j=1}^{\infty}\int_{{\overline{\nabla}}_{\infty}}a_{ij}(x)\partial_i u(x)\partial_j v(x) {\rm PD}(\alpha,\theta)(dx),\ \ \ \ u,v\in{\cal P}.
\end{eqnarray*}

Now we can state the main result of this section.
\begin{thm}\label{thm2.1} {\rm (i)} For any $u,v\in{\cal P}$,
\begin{eqnarray}\label{May5B}
{\cal A}(u,v)=-\int_{{\overline{\nabla}}_{\infty}}({\cal G}u) v\,d{\rm PD}(\alpha,\theta).
\end{eqnarray}

\noindent {\rm (ii)}  The bilinear form $({\cal A}, {\cal P})$ is closable on $L^2({\overline{\nabla}}_{\infty};{\rm PD}(\alpha,\theta))$ and its closure  $({\cal A}, D({\cal A}))$ is a regular Dirichlet form.

\noindent {\rm (iii)}  There exists a time-reversible, conservative, diffusion process $(\{X_t\}_{t\ge 0},\{P_x\}_{x\in {\overline{\nabla}}_{\infty}})$ which is associated with $({\cal A}, D({\cal A}))$ and its stationary distribution is  ${\rm PD}(\alpha,\theta)$. Moreover,  there exists an ${\cal A}$-exceptional set $M\subset {\overline{\nabla}}_{\infty}$ such that for any $x\notin M$,
\begin{eqnarray}\label{May5D}
P_x\left\{\sum_{i=1}^{\infty}X_i(t)=1\ {\rm for\ all}\ t\in[0,\infty)\right\}=1.
\end{eqnarray}

\noindent {\rm (iv)}  $({\cal A}, D({\cal A}))$ is irreducible recurrent and $(\{X_t\}_{t\ge 0},\{P_x\}_{x\in {\overline{\nabla}}_{\infty}})$ satisfies the strong law of large numbers, i.e., there exists an ${\cal A}$-exceptional set $M\subset {\overline{\nabla}}_{\infty}$ such that for any $x\notin M$,
\begin{eqnarray*}
\lim_{t\rightarrow\infty}\frac{1}{t}\int_0^tf(X_s)ds=\int_{{\overline{\nabla}}_{\infty}}f\,d{\rm PD}(\alpha,\theta)\ \ P_x\text{-}a.s.,\ \ \ \  \forall f\in L^1({\overline{\nabla}}_{\infty};{\rm PD}(\alpha,\theta)).
\end{eqnarray*}
\end{thm}

As the key step in  proving Theorem \ref{thm2.1}, we first generalize the Ewens-Sampling formula to the fractional powers.

\subsection{The Ewens-Pitman sampling formula}

\begin{pro}\label{proMay14a} Let $m\in\mathbb{N}$.
For  real numbers $\kappa_1>\alpha, \ldots, \kappa_m>\alpha$,
\begin{eqnarray}\label{Mar22a}
\int_{{\overline{\nabla}}_{\infty}}\sum_{i_1, \ldots, i_m\ {\rm distinct}} \prod_{l=1}^m x_{i_l}^{\kappa_l}\,  {\rm PD}(\alpha,\theta)(dx)= \frac{\Gamma(\theta+1)\prod_{l=1}^{m-1}(\theta+l\alpha)}{\Gamma(\theta+\sum_{l=1}^m\kappa_l)}\prod_{l=1}^m\frac{\Gamma(\kappa_l-\alpha)}{\Gamma(1-\alpha)}.
\end{eqnarray}
\end{pro}

\noindent {\bf Proof.}\ \ By \cite[Corollary 3]{PY}, for $\kappa>\alpha$, we have
$$
\int_{{\overline{\nabla}}_{\infty}}\sum_{i=1}^{\infty}x^{\kappa}_i\,  {\rm PD}(\alpha,\theta)(dx)=\frac{\Gamma(\theta+1)\Gamma(\kappa-\alpha)}{\Gamma(\theta+\kappa)\Gamma(1-\alpha)}.
$$
In what follows, we consider the case that $m\ge 2$.

\noindent (i) Assume that $\theta>0$. Let $(\Omega, {\cal F},P)$ be a probability space and $\{\tau_t: t\geq 0\}$ and $\{\gamma_t: t\geq 0\}$ be two independent subordinators on $(\Omega, {\cal F},P)$ with respective L\'evy measures
\[
x^{-(1+\alpha)}e^{-x}dx\ \  \mbox{and}\ \  x^{-1}e^{-x}dx,\ \ \ \ x>0.\]
Set
\[
\sigma_{\alpha,\theta}= \tau\left(\frac{\alpha}{\Gamma(1-\alpha)} \gamma(\theta/\alpha)\right),
\]
and denote the ranked jump sizes of $\{\tau_t: t\geq 0\}$ over $[0, \frac{\alpha}{\Gamma(1-\alpha)}\gamma(\theta/\alpha)]$ as $\{J_i(\alpha,\theta): i\geq 1\}$. It is known that $\sigma_{\alpha,\theta}$ is a Gamma$(\theta,1)$ random variable independent of $\{\rho_i =\frac{J_i(\alpha,\theta)}{\sigma_{\alpha,\theta}}: i\geq 1\}$. We use $E$ to denote expectation with respect to $P$.

For $\kappa_1>\alpha, \ldots, \kappa_m>\alpha$, and $\kappa=\sum_{i=1}^m \kappa_i$, we have
\begin{eqnarray*}
&&E\left[ \sigma^\kappa_{\alpha,\theta}\right]\int_{{\overline{\nabla}}_{\infty}} \sum_{ i_1, \ldots, i_m\ {\rm distinct}} \prod_{l=1}^mx_{i_l}^{\kappa_l} \,  {\rm PD}(\alpha,\theta)(dx) \\
&=&E\bigg[ \sum_{ i_1, \ldots, i_m\ {\rm distinct}} \prod_{l=1}^mJ_{i_l}^{\kappa_l} \bigg] \\
&=&E\left[E\left.\left[ \sum_{ i_1, \ldots, i_m\ {\rm distinct}} \prod_{l=1}^mJ_{i_l}^{\kappa_l}\,\right|\,\gamma(\theta/\alpha) \right] \right]\\
&=&E\left[\prod_{i=1}^mE\left.\left[\sum_{l=1}^{\infty} J^{\kappa_i}_{l}\,\right|\,\gamma(\theta/\alpha) \right] \right]\\
&=&\bigg[\frac{\alpha}{\Gamma(1-\alpha)}\bigg]^mE\bigg[\gamma^m(\theta/\alpha)\bigg]\prod_{i=1}^m \int_0^{\infty} x^{\kappa_i-1-\alpha} e^{-x}dx,
\end{eqnarray*}
which equals
\[
\frac{\alpha^m\Gamma(\frac{\theta}{\alpha}+m)}{\Gamma(\frac{\theta}{\alpha})}\prod_{i=1}^m \frac{\Gamma(\kappa_i-\alpha)}{\Gamma(1-\alpha)}.
\]
Then,
\begin{eqnarray*}
\int_{{\overline{\nabla}}_{\infty}} \sum_{i_1, \ldots, i_m\ {\rm distinct}} \prod_{l=1}^mx_{i_l}^{\kappa_l}\,  {\rm PD}(\alpha,\theta)(dx) &=& \frac{1}{E[\sigma^\kappa_{\alpha,\theta}]}\frac{\alpha^m\Gamma(\frac{\theta}{\alpha}+m)}{\Gamma(\frac{\theta}{\alpha})}\prod_{l=1}^m \frac{\Gamma(\kappa_l-\alpha)}{\Gamma(1-\alpha)}
\\
&=&\frac{\alpha^m\Gamma(\frac{\theta}{\alpha}+m)}{\Gamma(\frac{\theta}{\alpha})}\prod_{l=1}^m \frac{\Gamma(\kappa_l-\alpha)}{\Gamma(1-\alpha)}\frac{\Gamma(\theta)}{\Gamma(\theta+\kappa)}\\
&=&\frac{\Gamma(\theta+1)\prod_{l=1}^{m-1}(\theta+l\alpha)}{\Gamma(\theta+\kappa)}\prod_{l=1}^m\frac{\Gamma(\kappa_l-\alpha)}{\Gamma(1-\alpha)}.
\end{eqnarray*}

\noindent (ii) The case $\theta=0$ can be dealt with by writing $\Gamma(\theta)/\Gamma(\frac{\theta}{\alpha})$ as $\left(\Gamma(\theta+1)/\Gamma(\frac{\theta}{\alpha} +1)\right)\alpha^{-1}$.

\noindent (iii) Assume that $\theta\in(-\alpha,0)$.  Let $\{U_n: n\geq 1\}$ be independent beta random variables on $(\Omega, {\cal F},P)$ with $U_n$ following the Beta$(1-\alpha, \theta+n\alpha)$ distribution. Define
\[
V_1= U_1,\ V_n =U_n\prod_{i=1}^{n-1}(1-U_i),\ n \geq 2.
\]

For $\kappa_1>\alpha, \ldots, \kappa_m>\alpha$, set $\kappa=\sum_{i=1}^m \kappa_i$.  By symmetry, we have
\[
\int_{{\overline{\nabla}}_{\infty}}\sum_{ i_1, \ldots, i_m\ {\rm distinct}} \prod_{l=1}^mx_{i_l}^{\kappa_l} \,  {\rm PD}(\alpha,\theta)(dx) =E\bigg[ \sum_{ i_1, \ldots, i_m\ {\rm distinct}} \prod_{l=1}^mV_{i_l}^{\kappa_l} \bigg].\]
Write
\[
\sum_{ i_1, \ldots, i_m\ {\rm distinct}} \prod_{l=1}^mV_{i_l}^{\kappa_l}=R_0+ \sum_{j=1}^m R_j \]
with
\[
R_0= \sum_{ i_1, \ldots, i_m\ {\rm distinct\ and}\ >1}\, \prod_{l=1}^m V_{i_l}^{\kappa_l} \]
and
\[
R_j=\sum_{ i_1, \ldots, i_m\ {\rm distinct},\, i_j=1}\, \prod_{l=1}^m V_{i_l}^{\kappa_l}. \]

For $\eta>\alpha$, define
\[
J(\alpha,\theta,\eta)= \frac{\Gamma(\theta+1)\prod_{l=1}^{m-1}(\theta+l\alpha)}{\Gamma(\theta+\eta)}\prod_{l=1}^m\frac{\Gamma(\kappa_l-\alpha)}{\Gamma(1-\alpha)}.
\]
By direct calculation and the sampling formula for non-negative $\theta$, we obtain
\beq
E[R_0]&=&E[(1-U_1)^{\kappa}]J(\alpha, \alpha+\theta, \kappa)\\
&=& \frac{\Gamma(\theta+1)}{\Gamma(1-\alpha)\Gamma(\theta+\alpha)}\frac{\Gamma(1-\alpha)\Gamma(\theta+\alpha+\kappa)}{\Gamma(\theta+1+\kappa)}\\
&&\cdot \frac{\Gamma(\theta+\alpha+1)\prod_{l=1}^{m-1}(\theta+\alpha+l\alpha)}{\Gamma(\theta+\alpha+\kappa)}\prod_{l=1}^m\frac{\Gamma(\kappa_l-\alpha)}{\Gamma(1-\alpha)} \\
&=& \frac{\theta+m\alpha}{\theta+\kappa}J(\alpha,\theta, \kappa).
\eeq
For each $j=1, \ldots, m$, we have
\beq
E[R_j]&=&E[U_1^{\kappa_j}(1-U_1)^{\kappa-\kappa_j}] J(\alpha,\alpha+\theta, \kappa-\kappa_j)\\
&=&\frac{\Gamma(\theta+1)}{\Gamma(1-\alpha)\Gamma(\theta+\alpha)}\frac{\Gamma(\kappa_j+1-\alpha)\Gamma(\theta+\alpha+\kappa-\kappa_j)}{\Gamma(\theta+1+\kappa)}\\
&&\cdot\frac{\Gamma(\theta+\alpha+1)\prod_{l=1}^{m-2}(\theta+\alpha+l\alpha)}{\Gamma(\theta+\alpha+\kappa-\kappa_j)}\prod_{l\neq j}\frac{\Gamma(\kappa_l-\alpha)}{\Gamma(1-\alpha)}
 \\
&=&\frac{\kappa_j-\alpha}{\theta+\kappa}J(\alpha,\theta, \kappa).
\eeq

Putting things together, we obtain
\beq
\int_{{\overline{\nabla}}_{\infty}} \sum_{ i_1, \ldots, i_m\ {\rm distinct}} \prod_{l=1}^mx_{i_l}^{\kappa_l}\,  {\rm PD}(\alpha,\theta)(dx) &=& E[R_0]+\sum_{j=1}^mE[R_j]\\
&=& J(\alpha,\theta,\kappa).\eeq
\hfill $\Box$

For $ f\in L^1({\overline{\nabla}}_{\infty};{\rm PD}(\alpha,\theta))$, denote $E_{\alpha,\theta}[f]=\int_{{\overline{\nabla}}_{\infty}} f d\,{\rm PD}(\alpha,\theta)$.

\begin{pro}\label{pro2}
Let $m\in\mathbb{N}$. For  real numbers $\kappa_1>\alpha, \ldots, \kappa_m>\alpha$,
{\small\begin{eqnarray}\label{Apr5a3}
\int_{{\overline{\nabla}}_{\infty}}L(x)\sum_{i_1, \ldots, i_m\ {\rm distinct}} \prod_{l=1}^mx_{i_l}^{\kappa_l}  \,  {\rm PD}(\alpha,\theta)(dx)  =\frac{\Gamma(\theta+1)\prod_{l=1}^{m}(\theta+l\alpha)}{\alpha\Gamma(1-\alpha)\Gamma(\theta+\sum_{l=1}^m\kappa_l+\alpha)}\prod_{l=1}^m\frac{\Gamma(\kappa_l-\alpha)}{\Gamma(1-\alpha)}.\ \
\end{eqnarray}}

\noindent In particular, for $\kappa>\alpha$,
\begin{eqnarray*}
\int_{{\overline{\nabla}}_{\infty}}L(x)\sum_{i=1}^{\infty}x^\kappa_{i}\,  {\rm PD}(\alpha,\theta)(dx) =\frac{(\frac{\theta}{\alpha}+1)\Gamma(\theta+1)\Gamma(\kappa-\alpha)}{\Gamma(\theta+\kappa+\alpha)[\Gamma(1-\alpha)]^2}.
\end{eqnarray*}
\end{pro}

\noindent {\bf Proof.}\ \ Let $g$ be a non-negative measurable function on ${\overline{\nabla}}_{\infty}$.  By \cite[Proposition 14]{PY}, we get
\begin{eqnarray*}
E_{\alpha,\theta}[Lg]&=&\frac{\Gamma(\theta+1)[\Gamma(1-\alpha)]^{\frac{\theta}{\alpha}}}{\Gamma(\frac{\theta}{\alpha}+1)}E_{\alpha,0}\left[L^{\frac{\theta}{\alpha}+1}g\right]\nonumber\\
&=&\frac{\Gamma(\theta+1)[\Gamma(1-\alpha)]^{\frac{\theta}{\alpha}}}{\Gamma(\frac{\theta}{\alpha}+1)}\frac{\Gamma(\frac{\theta+\alpha}{\alpha}+1)}{\Gamma(\theta+\alpha+1)[\Gamma(1-\alpha)]^{\frac{\theta+\alpha}{\alpha}}}E_{\alpha,\theta+\alpha}[g]\nonumber\\
&=&\frac{(\frac{\theta}{\alpha}+1)\Gamma(\theta+1)}{\Gamma(\theta+\alpha+1)\Gamma(1-\alpha)}E_{\alpha,\theta+\alpha}[g],
\end{eqnarray*}
which together with (\ref{Mar22a}) implies (\ref{Apr5a3}).\hfill $\Box$

\subsection{Proof of Theorem \ref{thm2.1}}

(i) Let $v=\prod_{r=1}^lH_{n_r},$
where $n_1,\dots,n_l\in\{2,3,\dots\}$  and $l\in\mathbb{N}$. Denote
$n=\sum_{r=1}^ln_r.
$
We have
\begin{eqnarray}\label{May9e}
{\cal G}H_{n_r}(x)
&=&\frac{n_r(n_{r}-1)}{2}\sum_{i=1}^{\infty}\left\{(1-2x_i)x_i^{n_r-1+\alpha}+H_{1+\alpha}(x)x_i^{n_r}\right\}\nonumber\\
&&-\frac{n_r}{2}\sum\limits_{i=1}^\infty x_i^{n_r}\left\{ (\theta+\alpha+1)\left[x_{i}^{\alpha}-H_{1+\alpha}(x)\right]+\alpha L(x)\right\}\nonumber\\
&=&\frac{n_r}{2}\bigg\{\left[(\theta+\alpha+n_{r})H_{1+\alpha}(x)-\alpha L(x)\right]H_{n_r}(x)+(n_{r}-1)H_{n_r-1+\alpha}(x)\nonumber\\
&&\ \ \ \ \ \ -\left(\theta+\alpha+2n_r-1\right)H_{n_r+\alpha}(x)\bigg\},
\end{eqnarray}
and
\begin{eqnarray*}
&&\langle a(x)\nabla H_{n_r}(x),\nabla H_{n_v}(x)\rangle\\
&=&\sum_{i,j=1}^{\infty}a_{ij}(x){n_r}{n_v}x_i^{n_r-1}x_j^{n_v-1}\\
&=&{n_r}{n_v}\bigg(\sum_{i=1}^{\infty}[(1-2x_i)x_i^{n_r+n_v-1+\alpha}+H_{1+\alpha}(x)x_i^{n_r+n_v}]+\sum_{i\not =j}[H_{1+\alpha}(x)-x^{\alpha}_i-x^{\alpha}_j]x_i^{n_r}x_j^{n_v}\bigg)\\
&=&{n_r}{n_v}\left[H_{n_r+n_v-1+\alpha}(x)+H_{1+\alpha}(x)H_{n_r}(x)H_{n_v}(x)-H_{n_r+\alpha}(x)H_{n_v}(x)-H_{n_r}(x)H_{n_v+\alpha}(x)\right].
\end{eqnarray*}
Then,
\begin{eqnarray}\label{May14A}
{\cal G}v
&=&\sum_{r=1}^l{\cal L}H_{n_r}\prod_{t\not=r}H_{n_t}+\sum_{r<v}\langle a\nabla H_{n_r},\nabla H_{n_v}\rangle\prod_{t\not=r,v}H_{n_t}\nonumber\\
&=&\frac{1}{2}\sum_{r=1}^ln_r\left[(\theta+\alpha+n_{r})H_{1+\alpha}-\alpha L\right]\prod_{r=1}^lH_{n_r}\nonumber\\
&&+\frac{1}{2}\sum_{r=1}^ln_r\left[(n_{r}-1)H_{n_r-1+\alpha}-\left(\theta+\alpha+2n_r-1\right)H_{n_r+\alpha}\right]\prod_{t\not=r}H_{n_t}\nonumber\\
&&+\sum_{r<v}{n_r}{n_v}\left[H_{n_r+n_v-1+\alpha}+H_{1+\alpha}(x)H_{n_r}H_{n_v}-H_{n_r+\alpha}H_{n_v}-H_{n_r}H_{n_v+\alpha}\right]\prod_{t\not=r,v}H_{n_t}\nonumber\\
&=&\frac{n}{2}\left[(\theta+\alpha+n)H_{1+\alpha}-\alpha L\right]\prod_{r=1}^lH_{n_r}\nonumber\\
&&+\frac{1}{2}\sum_{r=1}^ln_r\left[(n_{r}-1)H_{n_r-1+\alpha}-\left(\theta+\alpha+2n_r-1\right)H_{n_r+\alpha}\right]\prod_{t\not=r}H_{n_t}\nonumber\\
&&+\frac{1}{2}\sum_{r\not=v}{n_r}{n_v}H_{n_r+n_v-1+\alpha}\prod_{t\not=r,v}H_{n_t}-\sum_{r=1}^ln_rH_{n_r+\alpha}\sum_{v\not= r}n_v\prod_{t\not=r}H_{n_t}.
\end{eqnarray}

Denote by $\sigma(l,w)$  the collection of partitions $c$ of $\{1,\dots,l\}$ into $w$ non-empty sets $c_1,\dots,c_w$ such that $\min\, c_1<\cdots<\min\, c_w$ and the number of elements in $c_i$ is denoted by $|c_i|$. For $1\le j\le w$, define
$$
N_j=\sum_{i\in c_j}n_i.
$$
Write
\begin{eqnarray*}
\int_{{\overline{\nabla}}_{\infty}}{\cal G}v\,d{\rm PD}(\alpha,\theta)=\sum_{w=1}^{l}\sum_{c\in \sigma(l,w)}\Xi(c)\frac{\Gamma(\theta+1)\prod_{j=1}^{w-1}(\theta+j\alpha)}{\Gamma(\theta+n)}\prod_{j=1}^w\frac{\Gamma(N_j-\alpha)}{\Gamma(1-\alpha)},
\end{eqnarray*}
where
{\small\begin{eqnarray*}
\Xi(c)
&=&\frac{n(\theta+\alpha+n)\Gamma(\theta+n)}{2\Gamma(\theta+\alpha+n+1)}\left[\frac{\theta+w\alpha}{\Gamma(1-\alpha)}+\sum_{j=1}^w\frac{\Gamma(N_j+1)}{\Gamma(N_j-\alpha)}\right]-\frac{n(\theta+w\alpha)\Gamma(\theta+n)}{2\Gamma(1-\alpha)\Gamma(\theta+\alpha+n)}\\
&&+\frac{1}{2}\sum_{j=1}^w\sum_{r\in c_j}n_r\Bigg\{\frac{(n_{r}-1)\Gamma(\theta+n)}{\Gamma(\theta+\alpha+n-1)}\frac{\Gamma(N_j-1)}{\Gamma(N_j-\alpha)}-\frac{(\theta+\alpha+2n_r-1)\Gamma(\theta+n)}{\Gamma(\theta+\alpha+n)}\frac{\Gamma(N_j)}{\Gamma(N_j-\alpha)}\Bigg\}\\
&&+\frac{1}{2}\sum_{|c_j|\ge2}\sum_{r\not=v;\,r,v\in c_j}n_rn_v\frac{\Gamma(\theta+n)}{\Gamma(\theta+\alpha+n-1)}\frac{\Gamma(N_j-1)}{\Gamma(N_j-\alpha)}\\
&&-\sum_{j=1}^w\sum_{r\in c_j}n_r(n-n_r)\frac{\Gamma(\theta+n)}{\Gamma(\theta+\alpha+n)}\frac{\Gamma(N_j)}{\Gamma(N_j-\alpha)}
\end{eqnarray*}}

\noindent by (\ref{May14A}), Propositions \ref{proMay14a} and \ref{pro2}. Then, by direct calculation, we obtain
{\small\begin{eqnarray*}
\Xi(c)
&=&\frac{\Gamma(\theta+n)}{\Gamma(\theta+\alpha+n)}\Bigg\{\frac{n}{2}\sum_{j=1}^w\frac{\Gamma(N_j+1)}{\Gamma(N_j-\alpha)}\\
&&\ \ +\frac{\theta+\alpha+n-1}{2}\sum_{j=1}^w\frac{\Gamma(N_j-1)}{\Gamma(N_j-\alpha)}N_j(N_j-1)-\frac{\theta+\alpha+2n-1}{2}\sum_{j=1}^w\frac{\Gamma(N_j)}{\Gamma(N_j-\alpha)}N_j\Bigg\}\\
&=&0,
\end{eqnarray*}}

\noindent which implies
\begin{eqnarray}\label{May14C}
\int_{{\overline{\nabla}}_{\infty}}{\cal G}v\,d{\rm PD}(\alpha,\theta)=0.
\end{eqnarray}

Let $u=\prod_{p=1}^kH_{m_p},
$ where $m_1,\dots,m_k\in\{2,3,\dots\}$  and $k\in\mathbb{N}$. Denote
$ m=\sum_{p=1}^km_p.
$
By (\ref{May14A}), we get
\begin{eqnarray}\label{May14B}
u{\cal G}v
&=&\frac{n}{2}\prod_{p=1}^kH_{m_p}\left[(\theta+\alpha+n)H_{1+\alpha}-\alpha L\right]\prod_{r=1}^lH_{n_r}\nonumber\\
&&+\frac{1}{2}\prod_{p=1}^kH_{m_p}\sum_{r=1}^ln_r\left[(n_{r}-1)H_{n_r-1+\alpha}-\left(\theta+\alpha+2n_r-1\right)H_{n_r+\alpha}\right]\prod_{t\not=r}H_{n_t}\nonumber\\
&&+\frac{1}{2}\prod_{p=1}^kH_{m_p}\sum_{r\not=v}{n_r}{n_v}H_{n_r+n_v-1+\alpha}\prod_{t\not=r,v}H_{n_t}-\prod_{p=1}^kH_{m_p}\sum_{r=1}^l{n_r}H_{n_r+\alpha}\sum_{v\not=r}{n_v}\prod_{t\not=r}H_{n_t}.\nonumber\\
&&
\end{eqnarray}

Denote by $\sigma(k+l,w)$  the collection of partitions $c$ of $\{1,\dots,k+l\}$ into $w$ non-empty sets $c_1,\dots,c_w$ such that $\min\, c_1<\cdots<\min\, c_w$ and the number of elements in $c_i$ is denoted by $|c_i|$. Then, for every $1\le w\le k+l$ and $c$ in $\sigma(k+l,w)$ there is a unique way of finding a partition $b$ of $\{1,\dots,k\}$ and a partition $d$ of $\{1,\dots,l\}$ such that
$$
b_i\bigcup d_i\not=\emptyset,\ c_i=b_i\bigcup\{k+j:j\in d_i\},\ i=1,\dots,w.
$$
For $1\le j\le w$, define
$$
M_j=\sum_{i\in b_j}m_i,\ N_j=\sum_{i\in d_j}n_i,\ \eta_j=M_j+N_j.
$$
Write $c=(b,d)$,
\begin{eqnarray*}
\int_{{\overline{\nabla}}_{\infty}}u{\cal G}v\,d{\rm PD}(\alpha,\theta)=\sum_{w=1}^{k+l}\sum_{(b,d)\in \sigma(k+l,w)}\Lambda(b,d)\frac{\Gamma(\theta+1)\prod_{j=1}^{w-1}(\theta+j\alpha)}{\Gamma(\theta+m+n)}\prod_{j=1}^w\frac{\Gamma(\eta_j-\alpha)}{\Gamma(1-\alpha)},
\end{eqnarray*}
and
\begin{eqnarray*}
\int_{{\overline{\nabla}}_{\infty}}v{\cal G}u\,d{\rm PD}(\alpha,\theta)=\sum_{w=1}^{k+l}\sum_{(b,d)\in \sigma(k+l,w)}\Upsilon(b,d)\frac{\Gamma(\theta+1)\prod_{j=1}^{w-1}(\theta+j\alpha)}{\Gamma(\theta+m+n)}\prod_{j=1}^w\frac{\Gamma(\eta_j-\alpha)}{\Gamma(1-\alpha)}.
\end{eqnarray*}

By (\ref{May14B}), Propositions \ref{proMay14a} and \ref{pro2},  we get
{\small\begin{eqnarray*}
&&\Lambda(b,d)\nonumber\\
&=&\frac{n(\theta+\alpha+n)\Gamma(\theta+m+n)}{2\Gamma(\theta+\alpha+m+n+1)}\left[\frac{\theta+w\alpha}{\Gamma(1-\alpha)}+\sum_{j=1}^w\frac{\Gamma(\eta_j+1)}{\Gamma(\eta_j-\alpha)}\right]-\frac{n(\theta+w\alpha)\Gamma(\theta+m+n)}{2\Gamma(1-\alpha)\Gamma(\theta+\alpha+m+n)}\nonumber\\
&&+\frac{1}{2}\sum_{|d_j|\ge1}\sum_{r\in d_j}n_r\bigg\{\frac{(n_{r}-1)\Gamma(\theta+m+n)}{\Gamma(\theta+\alpha+m+n-1)}\frac{\Gamma(\eta_j-1)}{\Gamma(\eta_j-\alpha)}-\frac{(\theta+\alpha+2n_r-1)\Gamma(\theta+m+n)}{\Gamma(\theta+\alpha+m+n)}\frac{\Gamma(\eta_j)}{\Gamma(\eta_j-\alpha)}\bigg\}\nonumber\\
&&+\frac{1}{2}\sum_{|d_j|\ge2}\sum_{r\not=v;\,r,v\in d_j}n_rn_v\frac{\Gamma(\theta+m+n)}{\Gamma(\theta+\alpha+m+n-1)}\frac{\Gamma(\eta_j-1)}{\Gamma(\eta_j-\alpha)}\nonumber\\
&&-\sum_{|d_j|\ge1}\sum_{r\in d_j}n_r(n-n_r)\frac{\Gamma(\theta+m+n)}{\Gamma(\theta+\alpha+m+n)}\frac{\Gamma(\eta_j)}{\Gamma(\eta_j-\alpha)}\nonumber\\
&=&\frac{\Gamma(\theta+m+n)}{\Gamma(\theta+\alpha+m+n)}\bigg\{-\frac{mn(\theta+w\alpha)}{2(\theta+\alpha+m+n)\Gamma(1-\alpha)}+\frac{n(\theta+\alpha+n)}{2(\theta+\alpha+m+n)}\sum_{j=1}^w\frac{\Gamma(\eta_j+1)}{\Gamma(\eta_j-\alpha)}\nonumber\\
&&\ \ +\frac{\theta+\alpha+m+n-1}{2}\sum_{|d_j|\ge1}\frac{\Gamma(\eta_j-1)}{\Gamma(\eta_j-\alpha)}N_j(N_j-1)-\frac{\theta+\alpha+2n-1}{2}\sum_{|d_j|\ge1}\frac{\Gamma(\eta_j)}{\Gamma(\eta_j-\alpha)}N_j\bigg\}\nonumber\\
&=&\frac{\Gamma(\theta+m+n)}{\Gamma(\theta+\alpha+m+n)}\bigg\{-\frac{mn(\theta+w\alpha)}{2(\theta+\alpha+m+n)\Gamma(1-\alpha)}+\frac{n(\theta+\alpha+n)}{2(\theta+\alpha+m+n)}\sum_{j=1}^w\frac{\Gamma(\eta_j+1)}{\Gamma(\eta_j-\alpha)}\nonumber\\
&&\ \ +\frac{\theta+\alpha+m+n-1}{2}\sum_{j=1}^w\frac{\Gamma(\eta_j-1)}{\Gamma(\eta_j-\alpha)}N_j(N_j-1)-\frac{\theta+\alpha+2n-1}{2}\sum_{j=1}^w\frac{\Gamma(\eta_j)}{\Gamma(\eta_j-\alpha)}N_j\bigg\}.
\end{eqnarray*}}

\noindent Similarly, we get
{\small\begin{eqnarray}\label{May23b}
&&\Upsilon(b,d)
=\frac{\Gamma(\theta+m+n)}{\Gamma(\theta+\alpha+n+m)}\bigg\{-\frac{mn(\theta+w\alpha)}{2(\theta+\alpha+m+n)\Gamma(1-\alpha)}\nonumber\\
&& \hspace{7cm}+\frac{m(\theta+\alpha+m)}{2(\theta+\alpha+m+n)}\sum_{j=1}^w\frac{\Gamma(\eta_j+1)}{\Gamma(\eta_j-\alpha)}\\
&&\ \ \ \ +\frac{\theta+\alpha+m+n-1}{2}\sum_{j=1}^w\frac{\Gamma(\eta_j-1)}{\Gamma(\eta_j-\alpha)}M_j(M_j-1)
-\frac{\theta+\alpha+2m-1}{2}\sum_{j=1}^w\frac{\Gamma(\eta_j)}{\Gamma(\eta_j-\alpha)}M_j\bigg\}.\nonumber
\end{eqnarray}}

\noindent Then, by direct calculation, we obtain
{\small\begin{eqnarray*}
&&\Lambda(b,d)-\Upsilon(b,d)\\
&=&\frac{\Gamma(\theta+m+n)}{\Gamma(\theta+\alpha+m+n)}\Bigg\{\frac{n-m}{2}\sum_{j=1}^w\frac{\Gamma(\eta_j+1)}{\Gamma(\eta_j-\alpha)}+\frac{\theta+\alpha+m+n-1}{2}\sum_{j=1}^w\frac{\Gamma(\eta_j)}{\Gamma(\eta_j-\alpha)}(N_j-M_j)\\
&&\ \ -\frac{\theta+\alpha+2n-1}{2}\sum_{j=1}^w\frac{\Gamma(\eta_j)}{\Gamma(\eta_j-\alpha)}N_j+\frac{\theta+\alpha+2m-1}{2}\sum_{j=1}^w\frac{\Gamma(\eta_j)}{\Gamma(\eta_j-\alpha)}M_j\Bigg\},
\end{eqnarray*}}

\noindent which simplifies to
{\small\begin{eqnarray*}
&&\frac{\Gamma(\theta+m+n)}{\Gamma(\theta+\alpha+m+n)}\Bigg\{\frac{n-m}{2}\sum_{j=1}^w\frac{\Gamma(\eta_j)}{\Gamma(\eta_j-\alpha)}(N_j+M_j)+\frac{m+n}{2}\sum_{j=1}^w\frac{\Gamma(\eta_j)}{\Gamma(\eta_j-\alpha)}(N_j-M_j)\\
&&\ \ -n\sum_{j=1}^w\frac{\Gamma(\eta_j)}{\Gamma(\eta_j-\alpha)}N_j+m\sum_{j=1}^w\frac{\Gamma(\eta_j)}{\Gamma(\eta_j-\alpha)}M_j\Bigg\}\\
&=&0.
\end{eqnarray*}}

\noindent Hence,
\begin{eqnarray}\label{Apr30b}
\int_{{\overline{\nabla}}_{\infty}}(u{\cal G}v-v{\cal G}u)\,d{\rm PD}(\alpha,\theta)=0.
\end{eqnarray}

By (\ref{May5A}), we find that
\begin{eqnarray}\label{q2}
{\cal G}(uv)(x)={\cal G}u(x)v(x)+u(x){\cal G}v(x)+\sum_{i,j=1}^{\infty}a_{ij}(x)\partial_iu(x)\partial_j v(x),\ \ \ \ x\in {\overline{\nabla}}_{\infty}.
\end{eqnarray}
Taking integration on both sides of (\ref{q2}), by (\ref{May14C}), we obtain
\begin{eqnarray*}
0&=&\int_{{\overline{\nabla}}_{\infty}}{\cal G}u(x)v(x)
{\rm PD}(\alpha,\theta)(dx)+\int_{{\overline{\nabla}}_{\infty}}u(x){\cal G}v(x)
{\rm PD}(\alpha,\theta)(dx)\\
&&+\int_{{\overline{\nabla}}_{\infty}}\sum_{i,j=1}^{\infty}a_{ij}(x)\partial_iu(x)\partial_j v(x)
{\rm PD}(\alpha,\theta)(dx),
\end{eqnarray*}
which together with (\ref{Apr30b}) implies (\ref{May5B}).
Therefore, the proof of Part (i) is complete.

\noindent (ii) For  $u\in{\cal P}$ and $x\in {\overline{\nabla}}_{\infty}$, we have
\begin{eqnarray}\label{May6Y1}
&&\langle a(x)\nabla u(x),\nabla u(x)\rangle\nonumber\\
&=&\sum_{i=1}^{\infty}x_i^{1+\alpha}[\partial_iu(x)]^2+H_{1+\alpha}(x)\left(\sum_{i=1}^{\infty}x_i\partial _iu(x)\right)^2-2\sum_{i=1}^{\infty}x_i^{1+\alpha}\partial _iu(x)\sum_{i=1}^{\infty}x_i\partial _iu(x)\nonumber\\
&\ge&2\left(\sum_{i=1}^{\infty}x_i^{1+\alpha}[\partial_iu(x)]^2\right)^{\frac{1}{2}}\left[H_{1+\alpha}(x)\right]^{\frac{1}{2}}\left|\sum_{i=1}^{\infty}x_i\partial _iu(x)\right|-2\sum_{i=1}^{\infty}x_i^{1+\alpha}\partial _iu(x)\sum_{i=1}^{\infty}x_i\partial _iu(x)\nonumber\\
&\ge&0.
\end{eqnarray}
Then, $({\cal A}, {\cal P})$ is non-negative definite, i.e., ${\cal A}(u,u)\ge 0$ for any $u\in {\cal P}$.

By (\ref{May5B}),   $({\cal A}, {\cal P})$ is closable on $L^2({\overline{\nabla}}_{\infty};{\rm PD}(\alpha,\theta))$. This form is clearly Markovian. Since the Markovian property is preserved by closure, $({\cal A}, D({\cal A}))$ is a Dirichlet form. Note that ${\cal P}$ is uniformly dense in $C({\overline{\nabla}}_{\infty})$ and ${\rm PD}(\alpha,\theta)$ is a regular Borel measure which charges every non-empty open set. Then, $({\cal A}, D({\cal A}))$  is a regular Dirichlet form.

\noindent (iii) By Fukushima's celebrated theorem (cf. \cite[Theorem 6.2.1]{Fuku} and \cite[Theorem 1.5.1]{CF}), there exists a time-reversible Markov process $(\{X_t\}_{t\ge 0}$, $\{P_x\}_{x\in {\overline{\nabla}}_{\infty}})$ which is associated with $({\cal A}, D({\cal A}))$ and its stationary distribution is ${\rm PD}(\alpha,\theta)$. Since ${\cal L}1=0$,  $\{X_t\}_{t\ge 0}$ is conservative. Further, by  \cite[Proposition 2.3]{S2}, we find that $({\cal A}, D({\cal A}))$ satisfies the local property. Then,  $\{X_t\}_{t\ge 0}$ is a  diffusion.  Finally, by following exactly the argument in \cite[Proposition 1]{S1} (cf. also \cite[Lemma 5.1.2]{Fuku2}), we deduce that (\ref{May5D}) holds.

\noindent (iv) Let $f$ be an arbitrary function in $D({\cal A})$ satisfying ${\cal A}(f,f)=0$. To complete the proof, it suffices to show that  $f$ must be a constant (cf. \cite[Sections 1.6 and 4.7]{Fuku2}). Let $({\cal A}^{\gamma,\alpha,\theta}, {\cal P})$, $\gamma\ge0$, be the bilinear form given by (\ref{FGHJ}) below. We have  $({\cal A}, {\cal P})=({\cal A}^{\alpha,\alpha,\theta}, {\cal P})$. Then, by (\ref{jkl0}) and Theorem \ref{thm54} (ii), we find that $f\in D({\cal A}^{1,\alpha,\theta})$ and ${\cal A}^{1,\alpha,\theta}(f,f)=0$. Therefore, $f$ is a constant by Theorem \ref{thm54} (iv).
\hfill $\Box$

\subsection{Further properties}

It is well known that if $\alpha=0$ then the generator of the infinitely-many-neutral-alleles model is given by (see \cite{EK2})
\begin{eqnarray}\label{May1d}
{\cal G}u(x)=\frac{1}{2}\sum\limits_{i,j=1}^\infty a_{ij}(x)\partial_i\partial_ju(x)-\frac{\theta}{2}\sum\limits_{i=1}^\infty x_i\partial_iu(x),
\end{eqnarray}
where $a=(a_{ij})_{i,j=1}^{\infty}$ is given by  (\ref{May5D1}) with $\alpha=0$. A comparison between (\ref{May1d}) and (\ref{May5A}) indicates that \(\alpha L\rightarrow\theta\) as \(\alpha\downarrow 0\). We now show that this  is indeed the case.

\begin{pro}
$$
\lim_{\alpha\downarrow 0 }E_{\alpha,\theta}[(\alpha L-\theta)^2]=0.
$$
\end{pro}

\noindent {\bf Proof.}\ \ Denote
\[
C(\alpha,\theta)= \frac{\Gamma(\theta+1)}{\Gamma(\frac{\theta}{\alpha}+1)}[\Gamma(1-\alpha)]^{\frac{\theta}{\alpha}}.
\]
By \cite[Proposition 14]{PY}, we get
$$
{E}_{\alpha,0}\left[L^{\frac{\theta}{\alpha}}\right]=\frac{1}{C(\alpha,\theta)}.
$$
Then, for any $n\in\mathbb{N}$,
\beq
{E}_{\alpha,\theta}[L^n]&=&C(\alpha,\theta){E}_{\alpha,0}\left[L^{n+\frac{\theta}{\alpha}}\right]\\
&=&C(\alpha,\theta)\frac{1}{C(\alpha,n\alpha+\theta)}\\
&=& C(\alpha,\theta)\frac{\Gamma(n+\frac{\theta}{\alpha} +1)}{\Gamma(n\alpha+\theta+1)}\frac{1}{[\Gamma(1-\alpha)]^{n+\frac{\theta}{\alpha}}}\\
&=& \frac{\Gamma(\theta+1)}{\Gamma(n\alpha+\theta+1)}\frac{\Gamma(\frac{\theta}{\alpha}+n+1)}{\Gamma(\frac{\theta}{\alpha}+1)}\frac{1}{[\Gamma(1-\alpha)]^n}.
\eeq
This implies that
\beq
{E}_{\alpha,\theta}[(\alpha L)^2]&=&  \frac{\Gamma(\theta+1)}{\Gamma(2\alpha+\theta+1)}\frac{1}{[\Gamma(1-\alpha)]^2}(\theta+2\alpha)(\theta+\alpha),\\
{E}_{\alpha,\theta}[\alpha L]&=&  \frac{\Gamma(\theta+1)}{\Gamma(\alpha+\theta+1)}\frac{1}{\Gamma(1-\alpha)}(\theta+\alpha) \\
&\ra& \theta\ {\rm as}\  \alpha \downarrow 0,\\
E_{\alpha,\theta}\left[\left(\alpha L-{E}_{\alpha,\theta}[\alpha L]\right)^2\right]&=&\frac{\Gamma(\theta+1)}{\Gamma(2\alpha +\theta+1)[\Gamma(1-\alpha)]^2}\bigg\{ \alpha \theta+\alpha^2-\bigg[\frac{\Gamma(\theta+1)}{\Gamma(\theta+2\alpha+1)}-1\bigg](\theta+\alpha)^2\bigg\}\\
&\rightarrow&0\ {\rm as}\ \alpha \downarrow 0.
\eeq
Therefore, the proof is complete.
\hfill $\Box$
\vskip 0.3cm
Denote by ${\cal B}({\overline{\nabla}}_{\infty})$ the Borel $\sigma$-algebra of ${\overline{\nabla}}_{\infty}$ and $B_b({\overline{\nabla}}_{\infty})$ the set of all bounded measurable functions on ${\overline{\nabla}}_{\infty}$. Define the occupation time process by
$$
L_t(C)=\frac{1}{t}\int_0^t1_C(X_s)ds,\ \ \ \ C\in  {\cal B}({\overline{\nabla}}_{\infty}),\, t>0.
$$
Let ${\cal P}_1({\overline{\nabla}}_{\infty})$ be the space of all probability measures on ${\overline{\nabla}}_{\infty}$. We  equip ${\cal P}_1({\overline{\nabla}}_{\infty})$ with the $\tau$-topology, which is generated by open sets of the form
$$
U(\nu;\varepsilon,F)=\left\{\mu\in{\cal P}_1({\overline{\nabla}}_{\infty}):\left|\int_{{\overline{\nabla}}_{\infty}}F\,d(\mu-\nu)\right|<\varepsilon\right\},
$$
where $\nu\in {\cal P}_1({\overline{\nabla}}_{\infty})$, $\varepsilon>0$ and $F\in B_b({\overline{\nabla}}_{\infty})$.

By Theorem \ref{thm2.1} (iv) and \cite[Theorems 1 and 2]{Mu}, we obtain the following large deviation  result.
\begin{pro}
Let $U$ be a $\tau$-open and $K$  a $\tau$-compact subset of ${\cal P}_1({\overline{\nabla}}_{\infty})$. Then, for ${\cal A}$-q.e. $x\in {\overline{\nabla}}_{\infty}$,
$$
\liminf_{t\rightarrow\infty}\frac{1}{t}\log P_x(L_t \in U)\ge -\inf\{{\cal A}(u,u):u\in D({\cal A}), u^2{\rm PD}(\alpha,\theta)\in U\},
$$
and
\begin{eqnarray*}
&&\inf\bigg\{\sup_{x\in{\overline{\nabla}}_{\infty}\backslash N }\limsup_{t\rightarrow\infty}\frac{1}{t}\log P_x(L_t \in K): N\subset {\overline{\nabla}}_{\infty}, N\ {\rm is\ }{\cal A}{\text-exceptional}\bigg\}\\
&\le&-\inf\{{\cal A}(u,u):u\in D({\cal A}), u^2{\rm PD}(\alpha,\theta)\in K\}.
\end{eqnarray*}
\end{pro}
\vskip 0.2cm

Finally, we present a uniqueness result for the invariant measures associated with the operator \(\mathcal{G}\) defined by (\ref{May5A}). Note that invariance under the operator \(\mathcal{G}\) does not automatically equate to invariance under the Markov process \(\{X_t\}_{t\ge 0}\) from Theorem \ref{thm2.1}. In general settings,  an operator may admit an invariant measure even though it fails to generate a well-defined semigroup or diffusion process on the corresponding \(L^{2}\)-space.

Denote
$$
{\cal PD}=\{{\rm PD}(\beta,\tau):\beta\in [0,1),\tau>-\beta\}.
$$
\begin{pro}\label{proMay16a}
Let $\alpha\in(0,1)$ and $\theta>-\alpha$. Then, ${\rm PD}(\alpha,\theta)$ is the unique element in ${\cal PD}$ that  is an invariant measure of ${\cal G}$.
\end{pro}

\noindent {\bf Proof.}\ \ By Theorem \ref{thm2.1}, ${\rm PD}(\alpha,\theta)$ is an invariant measure of ${\cal G}$. Now suppose that ${\rm PD}(\beta,\tau)$ is an invariant measure of ${\cal G}$  for some $\beta\in[0,1)$ and $\tau>-\beta$. We will show that $\beta=\alpha$ and $\tau=\theta$.

By (\ref{May9e}), we get
\begin{eqnarray}\label{May9g}
{\cal G}H_{2}
&=&\left[(\theta+\alpha+2)H_{1+\alpha}-\alpha L\right]H_{2}+H_{1+\alpha}-\left(\theta+\alpha+3\right)H_{2+\alpha},\nonumber\\
{\cal G}H_{3}
&=&\frac{3}{2}\bigg\{\left[(\theta+\alpha+3)H_{1+\alpha}-\alpha L\right]H_{3}+2H_{2+\alpha}-\left(\theta+\alpha+5\right)H_{3+\alpha}\bigg\}.
\end{eqnarray}
Then, by (\ref{Mar22a}), (\ref{Apr5a3}) and (\ref{May9g}), we obtain
\begin{eqnarray}\label{PPP1}
&&\int_{{\overline{\nabla}}_{\infty}}{\cal G}H_{2}\,d{\rm PD}(\alpha,\tau)\nonumber\\
&=&(\theta+\alpha+2)E_{\alpha,\tau}\left[H_{1+\alpha}H_2\right]-\alpha E_{\alpha,\tau}\left[LH_{2}\right]+E_{\alpha,\tau}\left[H_{1+\alpha}\right]-(\theta+\alpha+3)E_{\alpha,\tau}\left[H_{2+\alpha}\right]\nonumber\\
&=&(\theta+\alpha+2)\frac{\Gamma(\tau+1)[(1-\alpha)(\tau+\alpha)+2]}{\Gamma(\tau+\alpha+3)\Gamma(1-\alpha)}-\frac{\Gamma(\tau+1)(1-\alpha)(\tau+\alpha)}{\Gamma(\tau+\alpha+2)\Gamma(1-\alpha)}\nonumber\\
&&+\frac{\Gamma(\tau+1)}{\Gamma(\tau+\alpha+1)\Gamma(1-\alpha)}-(\theta+\alpha+3)\frac{\Gamma(\tau+1)}{\Gamma(\tau+\alpha+2)\Gamma(1-\alpha)}\nonumber\\
&=&\alpha(\tau - \theta)(\tau+\alpha  )\frac{\Gamma(\tau+1)}{\Gamma(\tau+\alpha+3)\Gamma(1-\alpha)}.
\end{eqnarray}
Assume that $\beta=\alpha$. Then, by  $\int_{{\overline{\nabla}}_{\infty}}{\cal G}H_{2}\,d{\rm PD}(\alpha,\tau)=0$ and (\ref{PPP1}), we get
$$  \tau=\theta.$$

From now on till the end of the proof, we assume that $\beta\not=\alpha$. Then, $L=0$ ${\rm PD}(\beta,\tau)$ a.s.. Thus, by (\ref{Mar22a}) and (\ref{May9g}), we get
\begin{eqnarray*}
&&\int_{{\overline{\nabla}}_{\infty}}{\cal G}H_{2}\,d{\rm PD}(\beta,\tau)\nonumber\\
&=&(\theta+\alpha+2)E_{\beta,\tau}\left[H_{1+\alpha}H_2\right]+E_{\beta,\tau}\left[H_{1+\alpha}\right]-(\theta+\alpha+3)E_{\beta,\tau}\left[H_{2+\alpha}\right]\nonumber\\
&=&(\theta+\alpha+2)\left[\frac{\Gamma(\tau+1)(1-\beta)(\tau+\beta)\Gamma(1+\alpha-\beta)}{\Gamma(\tau+\alpha+3)\Gamma(1-\beta)}+\frac{\Gamma(\tau+1)\Gamma(3+\alpha-\beta)}{\Gamma(\tau+\alpha+3)\Gamma(1-\beta)}\right]\nonumber\\
&&+\frac{\Gamma(\tau+1)\Gamma(1+\alpha-\beta)}{\Gamma(\tau+\alpha+1)\Gamma(1-\beta)}-(\theta+\alpha+3)\frac{\Gamma(\tau+1)\Gamma(2+\alpha-\beta)}{\Gamma(\tau+\alpha+2)\Gamma(1-\beta)}
\end{eqnarray*}
\begin{eqnarray}\label{PPP2}
\hskip -4cm&=&-(\tau+ \beta )(\alpha^2 + \theta \alpha+ \alpha - \tau - 2)\frac{\Gamma(\tau+1)\Gamma(1+\alpha-\beta)}{\Gamma(\tau+\alpha+3)\Gamma(1-\beta)}.
\end{eqnarray}
Hence,  by $\int_{{\overline{\nabla}}_{\infty}}{\cal G}H_{2}\,d{\rm PD}(\beta,\tau)=0$, we obtain
\begin{eqnarray}\label{May9h}
\tau=\alpha(\theta+\alpha+1) - 2.
\end{eqnarray}

By (\ref{Mar22a}) and (\ref{May9g}), we get
{\small\begin{eqnarray*}
&&\frac{2}{3}\int_{{\overline{\nabla}}_{\infty}}{\cal G}H_{3}\,d{\rm PD}(\beta,\tau)\\
&=&(\theta+\alpha+3)E_{\beta,\tau}\left[H_{1+\alpha}H_3\right]+2E_{\beta,\tau}\left[H_{2+\alpha}\right]-(\theta+\alpha+5)E_{\beta,\tau}\left[H_{3+\alpha}\right]\\
&=&(\theta+\alpha+3)\left[\frac{\Gamma(\tau+1)(1-\beta)(2-\beta)(\tau+\beta)\Gamma(1+\alpha-\beta)}{\Gamma(\tau+\alpha+4)\Gamma(1-\beta)}+\frac{\Gamma(\tau+1)\Gamma(4+\alpha-\beta)}{\Gamma(\tau+\alpha+4)\Gamma(1-\beta)}\right]\\
&&+\frac{2\Gamma(\tau+1)\Gamma(2+\alpha-\beta)}{\Gamma(\tau+\alpha+2)\Gamma(1-\beta)}-(\theta+\alpha+5)\frac{\Gamma(\tau+1)\Gamma(3+\alpha-\beta)}{\Gamma(\tau+\alpha+3)\Gamma(1-\beta)}\\
&=&-(\tau+ \beta )\left( \alpha^3 - 2\alpha^2\beta + \alpha^2\theta + 4\alpha^2 - 2\alpha\beta\theta - 4\alpha\beta - 2\alpha\tau + 3\alpha\theta + \alpha + 2\beta\tau + 6\beta - 2\tau - 6 \right)\\
&&\cdot\frac{\Gamma(\tau+1)\Gamma(1+\alpha-\beta)}{\Gamma(\tau+\alpha+4)\Gamma(1-\beta)}.
\end{eqnarray*}}

\noindent Then, by $\int_{{\overline{\nabla}}_{\infty}}{\cal G}H_{3}\,d{\rm PD}(\beta,\tau)=0$, we get
\begin{eqnarray*}
\alpha^3 - 2\alpha^2\beta + \alpha^2\theta + 4\alpha^2 - 2\alpha\beta\theta - 4\alpha\beta - 2\alpha\tau + 3\alpha\theta + \alpha + 2\beta\tau + 6\beta - 2\tau - 6=0,
\end{eqnarray*}
which together with (\ref{May9h}) implies that
\begin{eqnarray}\label{May9j}
\beta=1-\frac{\alpha(\theta+\alpha+1)}{2}.
\end{eqnarray}

Since $\beta>0$, by (\ref{May9j}), we get
$$
\alpha(\theta+\alpha+1)<2,
$$
which together with (\ref{May9h}) implies that
$$
\tau+\beta=\frac{\alpha(\theta+\alpha+1)}{2}-1<0,
$$
which contradicts the fact that $\tau+\beta>0$ . Therefore, the result holds.\hfill $\Box$

\section{The infinite dimensional labelled model}\setcounter{equation}{0}

In the unlabelled model of Section 2, the interest is on the evolution of ordered masses/proportions of different types in a population. The detailed information is lost for the type evolution. In this section, we construct the labelled model where both the masses and types evolve with time. The inclusion of type evolution makes the construction much more challenging. The model presented below is the first to possess an explicitly identified generator.

Let $(S,d)$ be a Polish space and $\nu_0$ be a diffuse probability measure on its Borel $\sigma$-algebra ${\cal B}(S)$. There exists a countable collection of continuous functions $\{g_i\}_{i\ge 1}$ on $S$ so that $\sup_i|g_i(s)|\le 1$ for all $s\in S$, and
$$
\varrho(\mu,\nu)=\sup_i\int_S g_i \,d(\mu-\nu),\ \ \ \ \mu,\nu\in {\cal P}_1(S),
$$
defines a metric on ${\cal P}_1(S)$ that is compatible with the topology of weak convergence. Then, $({\cal P}_1(S),\varrho)$ is a Polish space. Denote by $B_b(S)$ the set of all bounded measurable functions on $S$. Let $\nu$ be a finite measure on $S$ and $\phi\in B_b(S)$, define $\langle \phi,\nu\rangle=\int_S\phi d\nu$.  We consider the space of cylinder functions on ${\cal P}_1(S)$:
\begin{eqnarray}\label{May6Ha}
{\cal F}=\{F(\mu)=f(\langle \phi_1,\mu\rangle,\dots,\langle
\phi_n,\mu\rangle): \phi_i\in B_b(S),1\le i\le n,f\in
C^{\infty}(\mathbb{R}^n),n\in\mathbb{N}\}.
\end{eqnarray}
For $F\in {\cal F}$, write $\nabla F(\mu)$ for the function $s\mapsto \partial F(\mu)/\delta \mu(s)$.

Let $\alpha\in(0,1)$ and $\theta>-\alpha$. For $\mu\in {\cal P}_1(S)$, denote by $\mu_a$ and $\mu_{na}$ its atomic part and non-atomic part, respectively. Suppose that $\mu_a=\sum_{i=1}^{\infty}\rho_i\delta_{\xi_i}$ with $\rho=(\rho_1,\rho_2,\dots)\in \overline{{\nabla}}_{\infty}$ and $\xi_i\in S$, $i\ge1$. Define
$$
H_{1+\alpha}(\mu)=H_{1+\alpha}(\rho),\ \ L(\mu)=L(\rho),\
\mu^{1+\alpha}=\sum_{i=1}^{\infty}\rho_i^{1+\alpha}\delta_{\xi_i},
$$
and for $\phi,\psi\in B_b(S)$,
\begin{eqnarray}\label{May6b}
[\phi,\psi]_{\mu}&=&\langle \phi\psi,\mu^{1+\alpha}\rangle+H_{1+\alpha}(\mu)\langle \phi,\mu_a\rangle\langle \psi,\mu_a\rangle-\langle \phi,\mu^{1+\alpha}\rangle\langle \psi,\mu_a\rangle-\langle \phi,\mu_a\rangle\langle \psi,\mu^{1+\alpha}\rangle\nonumber\\
&&+\langle \phi,\psi\rangle_{\mu_{na}},
\end{eqnarray}
where $\langle \phi
,\psi\rangle_{\mu_{na}}=\int_S\phi\psi\, d\mu_{na}-\int_S\phi\, d\mu_{na}\int_S\psi\, d\mu_{na}$.

For $F(\mu)=f(\langle \phi_1,\mu\rangle,\dots,\langle
\phi_n,\mu\rangle)\in{\cal F}$, define
\begin{eqnarray}\label{Apr28d}
{\cal L}F(\mu)&=&
\frac{1}{2}\sum\limits_{i,j=1}^n\partial_i\partial_jf(\langle \phi_1,\mu\rangle,\dots,\langle
\phi_n,\mu\rangle)[\phi_i,\phi_j]_{\mu}-\frac{1}{2}\sum\limits_{i=1}^n\partial_if(\langle \phi_1,\mu\rangle,\dots,\langle
\phi_n,\mu\rangle)\nonumber\\
&&\ \ \ \ \ \cdot\bigg\{(\theta+\alpha+1)\left[\langle\phi_i,\mu^{1+\alpha}\rangle-\langle\phi_i,H_{1+\alpha}(\mu)\mu\rangle\right]+\alpha L(\mu)\langle\phi_i-\nu_0(\phi_i),\mu\rangle\bigg\}.\ \ \ \
\end{eqnarray}
Consider the following  symmetric form on $L^2({\cal P}_1(S);\Pi_{\alpha,\theta,\nu_0})$:
\begin{eqnarray*}
{\cal E}(F,G)=\frac{1}{2}\int_{{\cal P}_1(S)}[ \nabla
F(\mu),\nabla G(\mu)]_{\mu}
\Pi_{\alpha,\theta,\nu_0}(d\mu),\ \ \ \ F,G\in
{\cal F}.
\end{eqnarray*}

We now present the main result of this section. We refer the reader to  \cite{MR} for the concepts and terminology of quasi-regular Dirichlet forms. We also refer the reader to Schmuland and his collaborators' works \cite{ORS, RS, S} for the beautiful method of using Dirichlet forms and quasi-continuous functions to investigate the path behaviour of Markov processes, in particular, the Fleming-Viot type diffusions.

\begin{thm}\label{thm3.1} {\rm (i)} For any $F,G\in{\cal F}$,
\begin{eqnarray}\label{May5E1}
{\cal E}(F,G)=-\int_{{\cal P}_1(S)}({\cal L}F) G\,d\Pi_{\alpha,\theta,\nu_0}.
\end{eqnarray}

\noindent {\rm (ii)}  The bilinear form $({\cal E}, {\cal F})$ is closable on $L^2({\cal P}_1(S);\Pi_{\alpha,\theta,\nu_0})$ and its closure  $({\cal E}, D({\cal E}))$ is a quasi-regular Dirichlet form.

\noindent {\rm (iii)}  There exists a time-reversible, conservative, diffusion process $(\{X_t\}_{t\ge 0},\{P_\mu\}_{\mu\in {\cal P}_1(S)})$ which is properly associated with $({\cal E}, D({\cal E}))$ and its stationary distribution is  $\Pi_{\alpha,\theta,\nu_0}$.

\noindent {\rm (iv)}  Let $A\in {\cal B}(S)$. If $\nu_0(A)=0$, then there exists an ${\cal E}$-exceptional set $N_A\subset {\cal P}_1(S)$ such that for any $\mu\notin N_A$,
\begin{eqnarray*}
P_{\mu}\left\{X_t(A)=0\ {\rm for\ all}\ t\in[0,\infty)\right\}=1.
\end{eqnarray*}
Moreover, there exists an ${\cal E}$-exceptional set $N\subset {\cal P}_1(S)$ such that for any $\mu\notin N$,
\begin{eqnarray*}
P_{\mu}\left\{X_t\ {\rm is\ purely\ atomic\ for\ all}\ t\in[0,\infty)\right\}=1,
\end{eqnarray*}
and
\begin{eqnarray*}
P_{\mu}\left\{t\mapsto X_t\ {\rm is\ continuous\ in\ variation\ norm}\right\}=1.
\end{eqnarray*}

\end{thm}
\vskip 0.3cm
\noindent {\bf Proof.}\ \ (i) Let
$$
F_{\phi_1,\dots,\phi_n}(\mu)=\langle \phi_1,\mu\rangle\cdots\langle
\phi_n,\mu\rangle,
$$
and
$$
G_{\phi_1,\dots,\phi_n}(\mu)=\langle \psi_1,\mu\rangle\cdots\langle
\psi_n,\mu\rangle,
$$
where $\phi_i,\psi_i\in B_b(S)$, $1\le i\le n$, $n\in\mathbb{N}$. By (\ref{Apr28d}), we find that
\begin{eqnarray}\label{Apr28e}
{\cal L}(FG)(\mu)={\cal L}F(\mu)G(\mu)+F(\mu){\cal L}G(\mu)+[ \nabla
F(\mu),\nabla G(\mu)]_{\mu}.
\end{eqnarray}
Suppose we can show that ${\cal L}$ is symmetric with respect to $\Pi_{\alpha,\theta,\nu_0}$, i.e., for any $F,G$ of the above form
\begin{eqnarray}\label{Apr28f}
\int_{{\cal P}_1(S)}{\cal L}F(\mu)G(\mu)
\Pi_{\alpha,\theta,\nu_0}(d\mu)=\int_{{\cal P}_1(S)}F(\mu){\cal L}G(\mu)
\Pi_{\alpha,\theta,\nu_0}(d\mu).
\end{eqnarray}
By (\ref{Apr28d}), we get $ {{\cal L}1=0}$. Since $\langle1,\mu\rangle$ can be added if necessary, letting $G=1$ in (\ref{Apr28f}), we obtain
\begin{eqnarray*}
\int_{{\cal P}_1(S)}{\cal L}F(\mu)
\Pi_{\alpha,\theta,\nu_0}(d\mu)=0.
\end{eqnarray*}
Then, by taking integration on both sides of (\ref{Apr28e}), we get
\begin{eqnarray*}
0&=&\int_{{\cal P}_1(S)}{\cal L}F(\mu)G(\mu)\Pi_{\alpha,\theta,\nu_0}(d\mu)+\int_{{\cal P}_1(S)}F(\mu){\cal L}G(\mu)\Pi_{\alpha,\theta,\nu_0}(d\mu)\\
&&+\int_{{\cal P}_1(S)}[ \nabla
F(\mu),\nabla G(\mu)]_{\mu}\Pi_{\alpha,\theta,\nu_0}(d\mu),
\end{eqnarray*}
which together with (\ref{Apr28f}) implies that
$$
\frac{1}{2}\int_{{\cal P}_1(S)}[ \nabla
F(\mu),\nabla G(\mu)]_{\mu}
\Pi_{\alpha,\theta,\nu_0}(d\mu)=-\int_{{\cal P}_1(S)}{\cal L}F(\mu)G(\mu)\Pi_{\alpha,\theta,\nu_0}(d\mu).
$$
Further, by virtue of polynomial approximation, we deduce that (\ref{May5E1}) holds for any $F,G\in{\cal F}$.

In the sequel, we will show that (\ref{Apr28f}) holds.  Suppose $A$ and $B$ are two disjoint Borel subsets of $S$. Let $F(\mu)=\mu(A)$ and $G(\mu)=\mu(B)$. Then, we have
\begin{eqnarray*}
&&\int_{{\cal P}_1(S)}{\cal L}F(\mu)G(\mu)
\Pi_{\alpha,\theta,\nu_0}(d\mu)\\
&=&-\frac{\nu_0(A)\nu_0(B)}{2}\bigg\{(\theta+\alpha+1)\int_{{\overline{\nabla}}_{\infty}}
\left[\sum_{i\not=j}\rho_i^{1+\alpha}\rho_j-\sum_{k=1}^{\infty}\rho_k^{1+\alpha}\sum_{i\not=j}\rho_i\rho_j\right]{\rm PD}(\alpha,\theta)(d\rho)\\
&&\ \ \ \ +\alpha\int_{{\overline{\nabla}}_{\infty}}
L(\rho)\left[\sum_{i\not=j}\rho_i\rho_j-1\right]{\rm PD}(\alpha,\theta)(d\rho)\bigg\}.
\end{eqnarray*}
Similarly, we can get the expression of $\int_{{\cal P}_1(S)}F(\mu){\cal L}G(\mu)
\Pi_{\alpha,\theta,\nu_0}(d\mu)$. Thus, (\ref{Apr28f}) holds when $n=1$.

Now we assume that $n\ge 2$. Write
\begin{eqnarray*}
&&F^i_{\phi_1,\dots,\phi_n}(\mu)=\prod_{l\not=i}\langle \phi_l,\mu\rangle,\\
&&F^{ij}_{\phi_1,\dots,\phi_n}(\mu)=\prod_{l\not=i,j}\langle \phi_l,\mu\rangle,\\
&&F^{i,\alpha}_{\phi_1,\dots,\phi_n}(\mu)=\langle \phi_i,\mu^{1+\alpha}\rangle\prod_{l\not=i}\langle \phi_l,\mu\rangle,\\
&&F^{j\rightarrow i}_{\phi_1,\dots,\phi_n}(\mu)=\langle \phi_i\phi_j,\mu^{1+\alpha}\rangle F^{ij}_{\phi_1,\dots,\phi_n}(\mu).
\end{eqnarray*}
Then,
\begin{eqnarray}\label{May9k}
&&{\cal L}F_{\phi_1,\dots,\phi_n}(\mu)\nonumber\\
&=&\sum_{1\le i<j\le n}F^{j\rightarrow i}_{\phi_1,\dots,\phi_n}(\mu)+\frac{n[(\theta+\alpha+n)H_{1+\alpha}(\mu)-\alpha  L(\mu)]}{2}F_{\phi_1,\dots,\phi_n}(\mu)\nonumber\\
&&-\left(\frac{\theta+\alpha+1}{2}+n-1\right)\sum_{i=1}^nF^{i,\alpha}_{\phi_1,\dots,\phi_n}(\mu)+\frac{\alpha}{2}L(\mu)\sum_{i=1}^n\nu_0(\phi_i)F^{i}_{\phi_1,\dots,\phi_n}(\mu).
\end{eqnarray}
Denote
\begin{eqnarray*}
H(\phi,\psi)&=&\sum_{1\le i<j\le n}\int_{{\cal P}_1(S)}\prod_{p=1}^n\langle \phi_p,\mu\rangle\langle \psi_i\psi_j,\mu^{1+\alpha}\rangle \prod_{l\not=i,j}\langle \psi_l,\mu\rangle\Pi_{\alpha,\theta,\nu_0}(d\mu)\\
&&-\left(\frac{\theta+\alpha+1}{2}+n-1\right)\sum_{i=1}^n\int_{{\cal P}_1(S)}\prod_{p=1}^n\langle \phi_p,\mu\rangle\langle \psi_i,\mu^{1+\alpha}\rangle\prod_{l\not=i}\langle \psi_l,\mu\rangle\Pi_{\alpha,\theta,\nu_0}(d\mu)\\
&&+\frac{\alpha}{2}\sum_{i=1}^n\nu_0(\psi_i)\int_{{\cal P}_1(S)}L(\mu)\prod_{p=1}^n\langle \phi_p,\mu\rangle\prod_{l\not=i}\langle \psi_l,\mu\rangle\Pi_{\alpha,\theta,\nu_0}(d\mu).
\end{eqnarray*}
Then, (\ref{Apr28f}) is equivalent to $H(\phi,\psi)=H(\psi,\phi)$.

Denote by $\sigma(2n,k)$  the collection of partitions $c$ of $\{1,\dots,2n\}$ into $k$ non-empty sets $c_1,\dots,c_k$ such that $\min\, c_1<\cdots<\min\, c_k$ and the number of elements in $c_i$ is denoted by $|c_i|$. Let $\zeta(n,k)$ denote the set of partitions of $\{1,\dots,n\}$ into $k$ sets. Then, for every $1\le k\le 2n$ and $c$ in $\sigma(2n,k)$ there is a unique way of finding $b$ and $d$ in $\zeta(n,k)$ such that
$$
b_i\bigcup d_i\not=\emptyset,\ \ c_i=b_i\bigcup\{n+l:l\in d_i\},\ i=1,\dots,k.
$$
By (\ref{Mar22a}), similar to \cite[Lemma 5.2.1]{F}, we can show that
\begin{eqnarray*}
&&\int_{{\cal P}_1(S)}F_{\phi_1,\dots,\phi_n}(\mu)G_{\psi_1,\dots,\psi_n}(\mu)
\Pi_{\alpha,\theta,\nu_0}(d\mu)\\
&=&\sum_{k=1}^{2n}\sum_{b,d\in\zeta(n,k)}\frac{\prod_{j=1}^{k-1}(\theta+j\alpha)}{\prod_{j=1}^{2n-1}(\theta+j)}\prod_{j=1}^k\frac{\Gamma(|b_j|+|d_j|-\alpha)}{\Gamma(1-\alpha)}\prod_{r=1}^k\nu_0\left(\prod_{j\in b_r}\phi_j\prod_{j\in d_r}\psi_j\right).
\end{eqnarray*}

Write
\begin{eqnarray*}
H(\phi,\psi)=\sum_{k=1}^{2n}\sum_{b,d\in\zeta(n,k)}h(b,d)\frac{\prod_{j=1}^{k-1}(\theta+j\alpha)}{\prod_{j=1}^{2n-1}(\theta+j)}\prod_{j=1}^k\frac{\Gamma(|b_j|+|d_j|-\alpha)}{\Gamma(1-\alpha)}\prod_{i=1}^k\nu_0\left(\prod_{j\in b_i}\phi_j\prod_{j\in d_i}\psi_j\right),
\end{eqnarray*}
and
\begin{eqnarray*}
H(\psi,\phi)=\sum_{k=1}^{2n}\sum_{b,d\in\zeta(n,k)}h(b,d)\frac{\prod_{j=1}^{k-1}(\theta+j\alpha)}{\prod_{j=1}^{2n-1}(\theta+j)}\prod_{j=1}^k\frac{\Gamma(|b_j|+|d_j|-\alpha)}{\Gamma(1-\alpha)}\prod_{i=1}^k\nu_0\left(\prod_{j\in b_i}\psi_j\prod_{j\in d_i}\phi_j\right).
\end{eqnarray*}
The coefficient  $h(b,d)$ can be calculated as follows. Noting that
\begin{eqnarray*}
&&\sum_{1\le i<j\le n}\int_{{\cal P}_1(S)}\prod_{p=1}^n\langle \phi_p,\mu\rangle\langle \psi_i\psi_j,\mu^{1+\alpha}\rangle \prod_{l\not=i,j}\langle \psi_l,\mu\rangle\Pi_{\alpha,\theta,\nu_0}(d\mu)\\
&=&\sum_{k=1}^{2n}\sum_{b,d\in\zeta(n,k)}\sum_{i:|d_i|\ge2}{|d_i|\choose 2}\frac{\prod_{j=1}^{k-1}(\theta+j\alpha)}{\frac{\Gamma(\theta+2n-1+\alpha)}{\Gamma(\theta+1)}}\prod_{r=1}^k\frac{\Gamma(|b_r|+|d_r|-\alpha)}{\Gamma(1-\alpha)}\frac{\Gamma(|b_i|+|d_i|-1)}{\Gamma(|b_i|+|d_i|-\alpha)}\\
&&\cdot \prod_{r=1}^k\nu_0\left(\prod_{j\in b_r}\phi_j\prod_{j\in d_r}\psi_j\right),
\end{eqnarray*}
it follows that  the total contributions from the term
$$
\sum_{1\le i<j\le n}\int_{{\cal P}_1(S)}\prod_{p=1}^n\langle \phi_p,\mu\rangle\langle \psi_i\psi_j,\mu^{1+\alpha}\rangle \prod_{l\not=i,j}\langle \psi_l,\mu\rangle\Pi_{\alpha,\theta,\nu_0}(d\mu)
$$
are
$$
\frac{\Gamma(\theta+2n)}{2\Gamma(\theta+2n-1+\alpha)}\sum_{i:|d_i|\ge2}\frac{|d_i|(|d_i|-1)\Gamma(|b_i|+|d_i|-1)}{\Gamma(|b_i|+|d_i|-\alpha)}.
$$

Similarly, we have
\begin{eqnarray*}
&&\sum_{i=1}^n\int_{{\cal P}_1(S)}\prod_{p=1}^n\langle \phi_p,\mu\rangle\langle \psi_i,\mu^{1+\alpha}\rangle\prod_{l\not=i}\langle \psi_l,\mu\rangle\Pi_{\alpha,\theta,\nu_0}(d\mu)\\
&=&\sum_{k=1}^{2n}\sum_{b,d\in\zeta(n,k)}\sum_{i:|d_i|\ge1}|d_i|\frac{\prod_{j=1}^{k-1}(\theta+j\alpha)}{\frac{\Gamma(\theta+2n+\alpha)}{\Gamma(\theta+1)}}\prod_{r=1}^k\frac{\Gamma(|b_r|+|d_r|-\alpha)}{\Gamma(1-\alpha)}\frac{\Gamma(|b_i|+|d_i|)}{\Gamma(|b_i|+|d_i|-\alpha)}\\\
&&\cdot \prod_{r=1}^k\nu_0\left(\prod_{j\in b_r}\phi_j\prod_{j\in d_r}\psi_j\right).
\end{eqnarray*}
Then, the total contributions from the term
$$
-\left(\frac{\theta+\alpha+1}{2}+n-1\right)\sum_{i=1}^n\int_{{\cal P}_1(S)}\prod_{p=1}^n\langle \phi_p,\mu\rangle\langle \psi_i,\mu^{1+\alpha}\rangle\prod_{l\not=i}\langle \psi_l,\mu\rangle\Pi_{\alpha,\theta,\nu_0}(d\mu)
$$
are
\begin{eqnarray*}
&&-\frac{\Gamma(\theta+2n)}{2\Gamma(\theta+2n-1+\alpha)}\sum_{i:|d_i|\ge1}\frac{|d_i|\Gamma(|b_i|+|d_i|)}{\Gamma(|b_i|+|d_i|-\alpha)}.
\end{eqnarray*}

Finally, writing
\begin{eqnarray*}
&&\sum_{i=1}^n\nu_0(\psi_i)\int_{{\cal P}_1(S)}L(\mu)\prod_{p=1}^n\langle \phi_p,\mu\rangle\prod_{l\not=i}\langle \psi_l,\mu\rangle\Pi_{\alpha,\theta,\nu_0}(d\mu)\\
&=&\sum_{k=1}^{2n}\sum_{b,d\in\zeta(n,k)}\sum_{i:|d_i|=1,|b_i|=0}\int_{{\overline{\nabla}}_{\infty}}
L(\rho)\prod_{n_1,\dots,n_{i-1},n_{i+1},\dots,n_k\ {\rm distinct}}\rho^{|b_l|+|d_l|}_{n_l}{\rm PD}(\alpha,\theta)(d\rho)\\
&&\cdot\prod_{r=1}^k\nu_0\left(\prod_{j\in b_r}\phi_j\prod_{j\in d_r}\psi_j\right),
\end{eqnarray*}
it follows that the total contributions from the term
\begin{eqnarray*}
\frac{\alpha}{2}\sum_{i=1}^n\nu_0(\psi_i)\int_{{\cal P}_1(S)}L(\mu)\prod_{p=1}^n\langle \phi_p,\mu\rangle\prod_{l\not=i}\langle \psi_l,\mu\rangle\Pi_{\alpha,\theta,\nu_0}(d\mu)
\end{eqnarray*}
are
\begin{eqnarray*}
&&\frac{\alpha}{2}\sum_{i:|d_i|=1,|b_i|=0}\frac{\int_{{\overline{\nabla}}_{\infty}}
L(\rho)\prod_{n_1,\dots,n_{i-1},n_{i+1},\dots,n_k\ {\rm distinct}}\rho^{|b_l|+|d_l|}_{n_l}{\rm PD}(\alpha,\theta)(d\rho)}{\frac{\prod_{j=1}^{k-1}(\theta+j\alpha)}{\prod_{j=1}^{2n-1}(\theta+j)}\prod_{j=1}^k\frac{\Gamma(|b_j|+|d_j|-\alpha)}{\Gamma(1-\alpha)}}.
\end{eqnarray*}
By (\ref{Apr5a3}), we get
{\small\begin{eqnarray*}
&&\int_{{\overline{\nabla}}_{\infty}}
L(\rho)\prod_{n_1,\dots,n_{i-1},n_{i+1},\dots,n_k\ {\rm distinct}}\rho^{|b_l|+|d_l|}_{n_l}{\rm PD}(\alpha,\theta)(d\rho)\\
&=&\frac{\Gamma(\theta+1)\prod_{j=1}^{k-1}(\theta+j\alpha)}{\alpha\Gamma(1-\alpha)\Gamma(\theta+2n-1+\alpha)}\prod_{j=1}^k\frac{\Gamma(|b_j|+|d_j|-\alpha)}{\Gamma(1-\alpha)}.
\end{eqnarray*}}

\noindent Thus, the total contributions from the term
\begin{eqnarray*}
\frac{\alpha}{2}\sum_{i=1}^n\nu_0(\psi_i)\int_{{\cal P}_1(S)}L(\mu)\prod_{p=1}^n\langle \phi_p,\mu\rangle\prod_{l\not=i}\langle \psi_l,\mu\rangle\Pi_{\alpha,\theta,\nu_0}(d\mu)
\end{eqnarray*}
are
$$
\frac{\Gamma(\theta+2n)}{2\Gamma(1-\alpha)\Gamma(\theta+2n-1+\alpha)}\sum_{i:|d_i|=1,|b_i|=0}1.
$$

Putting all things together, we obtain
\begin{eqnarray}\label{May23p}
h(b,d)&=&\frac{\Gamma(\theta+2n)}{2\Gamma(\theta+2n-1+\alpha)}\sum_{i:|d_i|\ge2}\frac{|d_i|(|d_i|-1)\Gamma(|b_i|+|d_i|-1)}{\Gamma(|b_i|+|d_i|-\alpha)}\nonumber\\
&&-\frac{\Gamma(\theta+2n)}{2\Gamma(\theta+2n-1+\alpha)}\sum_{i:|d_i|\ge1}\frac{|d_i|\Gamma(|b_i|+|d_i|)}{\Gamma(|b_i|+|d_i|-\alpha)}\nonumber\\
&&+\frac{\Gamma(\theta+2n)}{2\Gamma(1-\alpha)\Gamma(\theta+2n-1+\alpha)}\sum_{i:|d_i|=1,|b_i|=0}1\\
&=&-\frac{\Gamma(\theta+2n)}{2\Gamma(\theta+2n-1+\alpha)}\sum_{i:|d_i|\ge2}\frac{|d_i||b_i|\Gamma(|b_i|+|d_i|-1)}{\Gamma(|b_i|+|d_i|-\alpha)}\nonumber\\
&&-\frac{\Gamma(\theta+2n)}{2\Gamma(\theta+2n-1+\alpha)}\sum_{i:|d_i|=1}\frac{\Gamma(|b_i|+|d_i|)}{\Gamma(|b_i|+|d_i|-\alpha)}\nonumber\\
&&+\frac{\Gamma(\theta+2n)}{2\Gamma(1-\alpha)\Gamma(\theta+2n-1+\alpha)}\sum_{i:|d_i|=1,|b_i|=0}1\nonumber.
\end{eqnarray}

Then,  for any $(b,d)$,
\begin{eqnarray*}
h(b,d)-h(d,b)&=&-\frac{\Gamma(\theta+2n)}{2\Gamma(\theta+2n-1+\alpha)}\sum_{i:|d_i|\ge2,|b_i|=1}\frac{|d_i|\Gamma(|b_i|+|d_i|-1)}{\Gamma(|b_i|+|d_i|-\alpha)}\nonumber\\
&&+\frac{\Gamma(\theta+2n)}{2\Gamma(\theta+2n-1+\alpha)}\sum_{i:|b_i|\ge2,|d_i|=1}\frac{|b_i|\Gamma(|b_i|+|d_i|-1)}{\Gamma(|b_i|+|d_i|-\alpha)}\nonumber\\
&&-\frac{\Gamma(\theta+2n)}{2\Gamma(\theta+2n-1+\alpha)}\sum_{i:|d_i|=1}\frac{\Gamma(|b_i|+|d_i|)}{\Gamma(|b_i|+|d_i|-\alpha)}\nonumber\\
&&+\frac{\Gamma(\theta+2n)}{2\Gamma(\theta+2n-1+\alpha)}\sum_{i:|b_i|=1}\frac{\Gamma(|b_i|+|d_i|)}{\Gamma(|b_i|+|d_i|-\alpha)}\nonumber\\
&&+\frac{\Gamma(\theta+2n)}{2\Gamma(1-\alpha)\Gamma(\theta+2n-1+\alpha)}\sum_{i:|d_i|=1,|b_i|=0}1\nonumber\\
&&-\frac{\Gamma(\theta+2n)}{2\Gamma(1-\alpha)\Gamma(\theta+2n-1+\alpha)}\sum_{i:|b_i|=1,|d_i|=0}1
\end{eqnarray*}
\begin{eqnarray*}
&=&-\frac{\Gamma(\theta+2n)}{2\Gamma(\theta+2n-1+\alpha)}\sum_{i:|b_i|=0,|d_i|=1}\frac{\Gamma(|b_i|+|d_i|)}{\Gamma(|b_i|+|d_i|-\alpha)}\nonumber\\
&&+\frac{\Gamma(\theta+2n)}{2\Gamma(\theta+2n-1+\alpha)}\sum_{i:|d_i|=0,|b_i|=1}\frac{\Gamma(|b_i|+|d_i|)}{\Gamma(|b_i|+|d_i|-\alpha)}\nonumber\\
&&+\frac{\Gamma(\theta+2n)}{2\Gamma(1-\alpha)\Gamma(\theta+2n-1+\alpha)}\sum_{i:|d_i|=1,|b_i|=0}1\nonumber\\
&&-\frac{\Gamma(\theta+2n)}{2\Gamma(1-\alpha)\Gamma(\theta+2n-1+\alpha)}\sum_{i:|b_i|=1,|d_i|=0}1\nonumber\\
&=&0.
\end{eqnarray*}
It follows that \(H(\phi,\psi) = H(\psi,\phi)\), which yields (\ref{Apr28f}). Therefore, the proof of Part (i) is complete.

\noindent (ii) Let $\phi\in B_b(S)$ and $\mu\in{\cal P}_1(S)$. We have
\begin{eqnarray}\label{May6a}
[\phi,\phi]_{\mu}&=&\sum_{i=1}^{\infty}\rho_i^{1+\alpha}\phi^2(\xi_i)+H_{1+\alpha}(\rho)\left(\sum_{i=1}^{\infty}\rho_i\phi(\xi_i)\right)^2-2\sum_{i=1}^{\infty}\rho_i^{1+\alpha}\phi(\xi_i)\sum_{i=1}^{\infty}\rho_i\phi(\xi_i)\nonumber\\
&&+\int_S\phi^2\, d\mu_{na}-\left(\int_S\phi\, d\mu_{na}\right)^2\nonumber\\
&\ge&2\left(\sum_{i=1}^{\infty}\rho_i^{1+\alpha}\phi^2(\xi_i)\right)^{\frac{1}{2}}\left[H_{1+\alpha}(\rho)\right]^{\frac{1}{2}}\left|\sum_{i=1}^{\infty}\rho_i\phi(\xi_i)\right|-2\sum_{i=1}^{\infty}\rho_i^{1+\alpha}\phi(\xi_i)\sum_{i=1}^{\infty}\rho_i\phi(\xi_i)\nonumber\\
&\ge&0.
\end{eqnarray}
Note that for $F(\mu)=f(\langle \phi_1,\mu\rangle,\dots,\langle
\phi_n,\mu\rangle)$ with $\phi_i\in B_b(S)$, $1\le i\le n$, and $f\in
C^{\infty}(\mathbb{R}^n)$,
$$
\nabla F(\mu)=\sum_{i=1}^n\partial_if(\langle \phi_1,\mu\rangle,\dots,\langle
\phi_n,\mu\rangle)\phi_i.
$$
Then, by (\ref{May6a}), we deduce that $({\cal E}, {\cal F})$ is non-negative definite.

By (\ref{May5E1}),   $({\cal E}, {\cal F})$ is closable on $L^2({\cal P}_1(S);\Pi_{\alpha,\theta,\nu_0})$. This form is clearly Markovian. Since the Markovian property is preserved by closure, $({\cal E}, D({\cal E}))$ is a Dirichlet form.  Let $\{\mu_j\}_{j\ge 1}$ be a countable dense set in ${\cal P}_1(S)$. Define $f_{ij}\in {\cal F}$ by $f_{ij}(\mu)=\int_S g_i d(\mu-\mu_j)$ for $\mu\in{\cal P}_1(S)$. Then, we have $\nabla f_{ij}(\mu)=g_i$ and so
$$
[ \nabla
f_{ij}(\mu),\nabla f_{ij}(\mu)]_{\mu}=\langle g_i^2,\mu^{1+\alpha}\rangle+H_{1+\alpha}(\mu)\langle g_i,\mu_a\rangle^2-2\langle g_i,\mu^{1+\alpha}\rangle\langle g_i,\mu_a\rangle\le 4.
$$
Also, $\sup_{i}f_{ij}(\mu)=\varrho(\mu,\mu_j)$. By following the argument in \cite[Theorem 3.4]{RS}, we deduce that $({\cal E}, D({\cal E}))$ is quasi-regular.

\noindent (iii) According to Albeverio, Ma and R\"ockner's theory of quasi-regular Dirichlet forms  (cf. \cite{AMR} and \cite[Chapter IV, Theorem 3.5]{MR}), there exists a time-reversible Markov process $(\{X_t\}_{t\ge 0}$, $\{P_{\mu}\}_{\mu\in {\cal P}_1(S)})$ which is properly associated with $({\cal E}, D({\cal E}))$ and its stationary distribution is $\Pi_{\alpha,\theta,\nu_0}$. Since ${\cal L}1=0$,  $\{X_t\}_{t\ge 0}$ is conservative. Further, by  \cite[Proposition 2.3]{S2}, we find that $({\cal E}, D({\cal E}))$ satisfies the local property. Then,  $\{X_t\}_{t\ge 0}$ is a  diffusion.

\noindent (iv) By virtue of (\ref{May6b}) and (\ref{May6a}),  the proof can be completed exactly like the proof of  \cite[Theorems 6.4, 6.5 and 6.9]{ORS}. We omit the details here.
\hfill $\Box$
\vskip 0.3cm

We would like to point out that the labelled model considered in this section can be easily extended to include interactive selection.
\begin{pro}\label{jkll}
Let $\eta\in L^2({\cal P}_1(S);\Pi_{\alpha,\theta,\nu_0})$ satisfying $\eta^2\ge\varepsilon>0$, $\Pi_{\alpha,\theta,\nu_0}$-a.e., or $\eta\in D({\cal E})$ and $\eta>0$, $\Pi_{\alpha,\theta,\nu_0}$-a.e.. Then, the perturbed bilinear form
\begin{eqnarray*}
{\cal E}^{\eta}(F,G)=\frac{1}{2}\int_{{\cal P}_1(S)}[ \nabla
F(\mu),\nabla G(\mu)]_{\mu}
\eta^2(\mu)\Pi_{\alpha,\theta,\nu_0}(d\mu),\ \ \ \ F,G\in
{\cal F},
\end{eqnarray*}
is closable on $L^2({\cal P}_1(S);\eta^2\Pi_{\alpha,\theta,\nu_0})$ and its closure is a quasi-regular local Dirichlet form which is properly associated with a diffusion process on ${\cal P}_1(S)$.
\end{pro}

By virtue of Theorem \ref{thm3.1}, the proof of Proposition \ref{jkll} follows the same lines as \cite[\S 5.2]{ORS}. We omit the details here.
\vskip 0.3cm
We now  present a result on the uniqueness  of invariant measures for the operator ${\cal L}$ defined by (\ref{Apr28d}). Denote
$$
\Pi_{\nu_0}=\{\Pi_{\beta,\tau,\nu_0}:\beta\in [0,1),\tau>-\beta\}.
$$
\begin{pro}\label{proMay16b}
Let $\alpha\in(0,1)$ and $\theta>-\alpha$. Then, $\Pi_{\alpha,\theta,\nu_0}$ is the unique element in $\Pi_{\nu_0}$ that is an invariant measure of ${\cal L}$.
\end{pro}

\noindent {\bf Proof.}\ \ By Theorem \ref{thm3.1}, $\Pi_{\alpha,\theta,\nu_0}$ is an invariant measure of ${\cal L}$. Now suppose that $\Pi_{\beta,\tau,\nu_0}$ is an invariant measure of ${\cal L}$  for some $\beta\in[0,1)$ and $\tau>-\beta$. We will show that $\beta=\alpha$ and $\tau=\theta$.

Let $A\in{\cal B}(S)$ with $\nu_0(A)\in (0,1)$. Define $F(\mu)=[\langle 1_A,\mu\rangle]^2$. By (\ref{May9k}), we get
\begin{eqnarray}\label{May9b1}
{\cal L}F(\mu)
&=&\langle 1_A,\mu^{1+\alpha}\rangle+\left\{(\theta+\alpha+2)H_{1+\alpha}(\mu)-\alpha  L(\mu)\right\}[\langle 1_A,\mu\rangle]^2\nonumber\\
&&-(\theta+\alpha+3)\langle 1_A,\mu^{1+\alpha}\rangle\langle 1_A,\mu\rangle+\alpha \nu_0(A)L(\mu)\langle 1_A,\mu\rangle.
\end{eqnarray}
Then, by (\ref{PPP1}) and (\ref{May9b1}), we obtain
{\small\begin{eqnarray}\label{OP1}
&&\int_{{\cal P}_1(S)}{\cal L}F\,d\Pi_{\alpha,\tau,\nu_0}\nonumber\\
&=&\nu_0(A)\bigg\{E_{\alpha,\tau}\left[H_{1+\alpha}\right]+(\theta+\alpha+2)E_{\alpha,\tau}\left[H_{1+\alpha}H_2\right]+(\theta+\alpha+2)\nu_0(A)E_{\alpha,\tau}\left[H_{1+\alpha}-H_{1+\alpha}H_2\right]\nonumber\\
&&\ \ \ \ -\alpha E_{\alpha,\tau}\left[LH_{2}\right]-\alpha\nu_0(A) E_{\alpha,\tau}\left[L-LH_2\right]-(\theta+\alpha+3)E_{\alpha,\tau}\left[H_{2+\alpha}\right]\nonumber\\
&&\ \ \ \ -(\theta+\alpha+3)\nu_0(A)E_{\alpha,\tau}\left[H_{1+\alpha}-H_{2+\alpha}\right]+\alpha\nu_0(A)E_{\alpha,\tau}\left[L\right]\bigg\}\nonumber\\
&=&\nu_0(A)[1-\nu_0(A)]\bigg\{E_{\alpha,\tau}\left[H_{1+\alpha}\right]+(\theta+\alpha+2)E_{\alpha,\tau}\left[H_{1+\alpha}H_2\right]\nonumber\\
&&\ \ \ \ -\alpha E_{\alpha,\tau}\left[LH_{2}\right]-(\theta+\alpha+3)E_{\alpha,\tau}\left[H_{2+\alpha}\right]\bigg\}\nonumber\\
&=&\alpha(\tau - \theta)(\tau+\alpha  )\nu_0(A)[1-\nu_0(A)]\frac{\Gamma(\tau+1)}{\Gamma(\tau+\alpha+3)\Gamma(1-\alpha)}.
\end{eqnarray}}
Assume that $\beta=\alpha$. Then, by  $\int_{{\cal P}_1(S)}{\cal L}F\,d\Pi_{\alpha,\tau,\nu_0}=0$, we get
$$  \tau=\theta.$$

From now on till the end of the proof, we assume that $\beta\not=\alpha$. Then, $L=0$ $\Pi_{\beta,\tau,\nu_0}$ a.s.. Similar to (\ref{OP1}), by virtue of (\ref{PPP2}), we get
\begin{eqnarray*}
\int_{{\cal P}_1(S)}{\cal L}F\,d\Pi_{\beta,\tau,\nu_0}
&=&\nu_0(A)[1-\nu_0(A)]\bigg\{E_{\beta,\tau}\left[H_{1+\alpha}\right]+(\theta+\alpha+2)E_{\beta,\tau}\left[H_{1+\alpha}H_2\right]\nonumber\\
&&\ \ \ \ -(\theta+\alpha+3)E_{\beta,\tau}\left[H_{2+\alpha}\right]\bigg\}\\
&=&-(\tau+ \beta )(\alpha^2 + \theta \alpha+ \alpha - \tau - 2)\frac{\Gamma(\tau+1)\Gamma(1+\alpha-\beta)}{\Gamma(\tau+\alpha+3)\Gamma(1-\beta)}.
\end{eqnarray*}
Then, by $\int_{{\cal P}_1(S)}{\cal L}F\,d\Pi_{\beta,\tau,\nu_0}=0$, we obtain
\begin{eqnarray}\label{May9ds}
\tau=\alpha(\theta+\alpha+1) - 2.
\end{eqnarray}

Let $A,B,C\in{\cal B}(S)$ such that they are disjoint and  $\nu_0(A)=\nu_0(B)=\nu_0(C)\in (0,1)$. Define $G(\mu)=\langle 1_A,\mu\rangle\langle 1_B,\mu\rangle\langle 1_C,\mu\rangle$ for $\mu\in{\cal P}_1(S)$. By (\ref{May9k}),  we get
\begin{eqnarray*}
{\cal L}G(\mu)
&=&\frac{3(\theta+\alpha+3)}{2}H_{1+\alpha}(\mu)\langle 1_A,\mu\rangle\langle 1_B,\mu\rangle\langle 1_C,\mu\rangle\nonumber\\
&&-\frac{\theta+\alpha+5}{2}\bigg\{\langle 1_A,\mu^{1+\alpha}\rangle\langle 1_B,\mu\rangle\langle 1_C,\mu\rangle+\langle 1_A,\mu\rangle\langle 1_B,\mu^{1+\alpha}\rangle\langle 1_C,\mu\rangle\\
&&\ \ \ \ +\langle 1_A,\mu\rangle\langle 1_B,\mu\rangle\langle 1^{1+\alpha}_C,\mu\rangle\bigg\}.
\end{eqnarray*}
Then, by (\ref{Mar22a}), we obtain
{\small\begin{eqnarray*}
&&\frac{2}{3}\int_{{\cal P}_1(S)}{\cal L}G\,d\Pi_{\beta,\tau,\nu_0}\\
&=&[\nu_0(A)]^3\bigg\{(\theta+\alpha+3)\frac{(\tau+\beta)(\tau+2\beta)(\tau+3\beta)\Gamma(\tau+1)\Gamma(1+\alpha-\beta)}{\Gamma(\tau+4+\alpha)\Gamma(1-\beta)}\\
&&\ \ \ \ +3(\theta+\alpha+3)\frac{(\tau+\beta)(\tau+2\beta)\Gamma(\tau+1)\Gamma(2+\alpha-\beta)}{\Gamma(\tau+4+\alpha)\Gamma(1-\beta)}\\
&&\ \ \ \ -(\theta+\alpha+5)\frac{(\tau+\beta)(\tau+2\beta)\Gamma(\tau+1)\Gamma(1+\alpha-\beta)}{\Gamma(\tau+3+\alpha)\Gamma(1-\beta)}\bigg\}\\
&=&[\nu_0(A)]^3\frac{(\tau+\beta)(\tau+2\beta)\Gamma(\tau+1)\Gamma(1+\alpha-\beta)}{\Gamma(\tau+4+\alpha)\Gamma(1-\beta)}\bigg\{(\theta+\alpha+3)(\tau+3\beta)\\
&&\ \ \ \ +3(\theta+\alpha+3)(1+\alpha-\beta)-(\theta+\alpha+5)(\tau+3+\alpha)\bigg\}\\
&=& 2(\alpha^2 + \theta\alpha + 2\alpha - \tau - 3)[\nu_0(A)]^3\frac{(\tau+\beta)(\tau+2\beta)\Gamma(\tau+1)\Gamma(1+\alpha-\beta)}{\Gamma(\tau+4+\alpha)\Gamma(1-\beta)}.
\end{eqnarray*}}
Thus, by $\int_{{\cal P}_1(S)}{\cal L}G\,d\Pi_{\beta,\tau,\nu_0}=0$, we get
\begin{eqnarray*}
\tau=\alpha(\theta+\alpha+2) - 3,
\end{eqnarray*}
which together with (\ref{May9ds}) implies that
$$
\alpha=1.
$$
We have arrived at a contradiction. Therefore, the proof is complete.
\hfill $\Box$

\vskip 0.3cm
Finally, let us explore the relationship between the unlabelled and labelled models. Denote by $({\cal L}, D({\cal L}))$ the generator of the Dirichlet form $({\cal E}, D({\cal E}))$. Define $\Phi:{\cal P}_1(S)\mapsto{\overline{\nabla}}_{\infty}$ by letting $\Phi(\mu)=(\rho_1,\rho_2,\dots)$ if $\rho_i$ is the mass of the $i$-th largest atom of $\mu$ (or 0 if $\mu$ has fewer than $i$ atoms).
\begin{pro}
For any $u\in{\cal P}$, $u\circ\Phi\in D({\cal L})$ and
\begin{eqnarray}\label{May23a}
{\cal L}(u\circ\Phi)=({\cal G}u)\circ\Phi.
\end{eqnarray}
For any $f\in  D({\cal A})$, $f\circ\Phi\in D({\cal E})$ and
\begin{eqnarray*}
{\cal A}(f,f)={\cal E}(f\circ\Phi, f\circ\Phi).
\end{eqnarray*}
\end{pro}

\noindent {\bf Proof.}\ \ Let
$
u=\prod_{p=1}^kH_{m_p},
$
where $m_1,\dots,m_k\in\{2,3,\dots\}$  and $k\in\mathbb{N}$. Denote  $m=\sum_{p=1}^km_p$. Define
$$
F(\mu)=u\circ\Phi(\mu),\ {\cal L}^*F(\mu)=({\cal G}u)\circ\Phi(\mu),\ \mu\in{\cal P}_1(S).
$$
Let
$$
G_{\psi,\dots,\psi_n}(\mu)=\langle \psi_1,\mu\rangle\cdots\langle
\psi_n,\mu\rangle,
$$
where $\psi_i\in B_b(S)$, $1\le i\le n$, $n\in\mathbb{N}$. Then, by (\ref{May14A}) and  (\ref{May9k}), we get
{\small\begin{eqnarray}\label{May24a}
&&\int_{{\cal P}_1(S)}({\cal L}^*F)G_{\psi,\dots,\psi_n}\,d
\Pi_{\alpha,\theta,\nu_0}\nonumber\\
&=&\frac{m(\theta+\alpha+m)}{2}\int_{{\cal P}_1(S)}\langle 1,\mu^{1+\alpha}\rangle\prod_{p=1}^k\langle 1,\mu^{m_p}\rangle\prod_{i=1}^n\langle \psi_i,\mu\rangle\Pi_{\alpha,\theta,\nu_0}(d\mu)\nonumber\\
&&-\frac{m\alpha}{2}\int_{{\cal P}_1(S)}L(\mu)\prod_{p=1}^k\langle 1,\mu^{m_p}\rangle\prod_{i=1}^n\langle \psi_i,\mu\rangle\Pi_{\alpha,\theta,\nu_0}(d\mu)\nonumber\\
&&+\frac{1}{2}\sum_{p=1}^km_p(m_p-1)\int_{{\cal P}_1(S)}\langle 1,\mu^{m_p-1+\alpha}\rangle\prod_{s\not=p}\langle 1,\mu^{m_s}\rangle\prod_{i=1}^n\langle \psi_i,\mu\rangle\Pi_{\alpha,\theta,\nu_0}(d\mu)\nonumber\\
&&-\frac{1}{2}\sum_{p=1}^km_p(\theta+\alpha+2m_p-1)\int_{{\cal P}_1(S)}\langle 1,\mu^{m_p+\alpha}\rangle\prod_{s\not=p}\langle 1,\mu^{m_s}\rangle\prod_{i=1}^n\langle \psi_i,\mu\rangle\Pi_{\alpha,\theta,\nu_0}(d\mu)\nonumber\\
&&+\frac{1}{2}\sum_{p\not=q}m_pm_q\int_{{\cal P}_1(S)}\langle 1,\mu^{m_p+m_q-1+\alpha}\rangle\prod_{s\not=p,q}\langle 1,\mu^{m_s}\rangle\prod_{i=1}^n\langle \psi_i,\mu\rangle\Pi_{\alpha,\theta,\nu_0}(d\mu)\nonumber\\
&&-\sum_{p=1}^km_p\int_{{\cal P}_1(S)}\langle 1,\mu^{m_p+\alpha}\rangle\sum_{q\not=p}m_q\prod_{s\not=p}\langle 1,\mu^{m_s}\rangle\prod_{i=1}^n\langle \psi_i,\mu\rangle\Pi_{\alpha,\theta,\nu_0}(d\mu),
\end{eqnarray}}

\noindent and
{\small\begin{eqnarray}\label{May24b}
&&\int_{{\cal P}_1(S)}F{\cal L}G_{\psi,\dots,\psi_n}\,d
\Pi_{\alpha,\theta,\nu_0}\nonumber\\
&=&\sum_{1\le i<j\le n}\int_{{\cal P}_1(S)}\prod_{p=1}^k\langle 1,\mu^{m_p}\rangle\langle \psi_i\psi_j,\mu^{1+\alpha}\rangle \prod_{l\not=i,j}\langle \psi_l,\mu\rangle\Pi_{\alpha,\theta,\nu_0}(d\mu)\nonumber\\
&&+\frac{n(\theta+\alpha+n)}{2}\int_{{\cal P}_1(S)}\langle 1,\mu^{1+\alpha}\rangle\prod_{p=1}^k\langle 1,\mu^{m_p}\rangle \prod_{i=1}^n\langle \psi_i,\mu\rangle\Pi_{\alpha,\theta,\nu_0}(d\mu)\nonumber\\
&&-\frac{n\alpha}{2}\int_{{\cal P}_1(S)}L(\mu)\prod_{p=1}^k\langle 1,\mu^{m_p}\rangle\prod_{i=1}^n\langle \psi_i,\mu\rangle\Pi_{\alpha,\theta,\nu_0}(d\mu)\nonumber\\
&&-\left(\frac{\theta+\alpha+1}{2}+n-1\right)\sum_{i=1}^n\int_{{\cal P}_1(S)}\prod_{p=1}^k\langle 1,\mu^{m_p}\rangle\langle \psi_i,\mu^{1+\alpha}\rangle\prod_{l\not=i}\langle \psi_l,\mu\rangle\Pi_{\alpha,\theta,\nu_0}(d\mu)\nonumber\\
&&+\frac{\alpha}{2}\sum_{i=1}^n\nu_0(\psi_i)\int_{{\cal P}_1(S)}L(\mu)\prod_{p=1}^k\langle 1,\mu^{m_p}\rangle\prod_{l\not=i}\langle \psi_l,\mu\rangle\Pi_{\alpha,\theta,\nu_0}(d\mu).
\end{eqnarray}}

Denote by $\sigma(n+k,l)$  the collection of partitions $c$ of $\{1,\dots,n+k\}$ into $l$ non-empty sets $c_1,\dots,c_l$ such that $\min\, c_1<\cdots<\min\, c_l$ and the number of elements in $c_i$ is denoted by $|c_i|$.
Then, for every $1\le l\le n+k$ and $c$ in $\sigma(n+k,l)$ there is a unique way of finding a partition $b$ of $\{1,\dots,n\}$ and a partition $d$ of $\{1,\dots,k\}$ such that
$$
b_i\bigcup d_i\not=\emptyset,\ c_i=b_i\bigcup\{n+j:j\in d_i\},\ i=1,\dots,l.
$$
For $1\le j\le k$, define
$$
M_j=\sum_{i\in d_j}m_i.
$$
Write $c=(b,d)$,
\begin{eqnarray*}
&&\int_{{\cal P}_1(S)}G_{\psi,\dots,\psi_n}({\cal L}^*F)\,d
\Pi_{\alpha,\theta,\nu_0}\\
&=&\sum_{l=1}^{n+k}\sum_{(b,d)\in\sigma(n+k,l)}H_1(b,d)\frac{\prod_{j=1}^{l-1}(\theta+j\alpha)}{\prod_{j=1}^{n+k-1}(\theta+j)}\prod_{j=1}^l\frac{\Gamma(|b_j|+M_j-\alpha)}{\Gamma(1-\alpha)}\prod_{i=1}^l\nu_0\left(\prod_{j\in b_i}\psi_j\right),
\end{eqnarray*}
and
\begin{eqnarray*}
&&\int_{{\cal P}_1(S)}({\cal L}G_{\psi,\dots,\psi_n})F\,d
\Pi_{\alpha,\theta,\nu_0}\\
&=&\sum_{l=1}^{n+k}\sum_{(b,d)\in\sigma(n+k,l)}H_2(b,d)\frac{\prod_{j=1}^{l-1}(\theta+j\alpha)}{\prod_{j=1}^{n+k-1}(\theta+j)}\prod_{j=1}^l\frac{\Gamma(|b_j|+M_j-\alpha)}{\Gamma(1-\alpha)}\prod_{i=1}^l\nu_0\left(\prod_{j\in b_i}\psi_j\right).
\end{eqnarray*}

By (\ref{May24a}), similar to (\ref{May23b}), we get
{\small\begin{eqnarray}\label{May23a1}
H_1(b,d)&=&\frac{\Gamma(\theta+m+n)}{2\Gamma(\theta+\alpha+m+n)}\Bigg\{-\frac{mn(\theta+l\alpha)}{(\theta+\alpha+m+n)\Gamma(1-\alpha)}+\frac{m(\theta+\alpha+m)}{\theta+\alpha+m+n}\sum_{j=1}^l\frac{\Gamma(|b_j|+M_j+1)}{\Gamma(|b_j|+M_j-\alpha)}\nonumber\\
&&\ \ +(\theta+\alpha+m+n-1)\sum_{j:|d_j|\ge1}\frac{\Gamma(|b_j|+M_j-1)}{\Gamma(|b_j|+M_j-\alpha)}M_j(M_j-1)\nonumber\\
&&\ \ -(\theta+\alpha+2m-1)\sum_{j=1}^l\frac{\Gamma(|b_j|+M_j)}{\Gamma(|b_j|+M_j-\alpha)}M_j\Bigg\}.
\end{eqnarray}}

\noindent By (\ref{May24b}), similar to (\ref{May23p}), we get
{\small\begin{eqnarray}\label{May23a2}
H_2(b,d)&=&\frac{\Gamma(\theta+m+n)}{2\Gamma(\theta+m+n-1+\alpha)}\sum_{j:|b_j|\ge2}\frac{|b_j|(|b_j|-1)\Gamma(|b_j|+M_j-1)}{\Gamma(|b_j|+M_j-\alpha)}\nonumber\\
&&+\frac{n(\theta+\alpha+n)\Gamma(\theta+m+n)}{2\Gamma(\theta+m+n+1+\alpha)}\left[\frac{\theta+l\alpha}{\Gamma(1-\alpha)}+\sum_{j=1}^l\frac{\Gamma(|b_j|+M_j+1)}{\Gamma(|b_j|+M_j-\alpha)}\right]\nonumber\\
&&-\frac{n(\theta+l\alpha)\Gamma(\theta+m+n)}{2\Gamma(1-\alpha)\Gamma(\theta+m+n+\alpha)}\nonumber\\
&&-\frac{(\theta+\alpha+2n-1)\Gamma(\theta+m+n)}{2\Gamma(\theta+m+n+\alpha)}\sum_{j:|b_j|\ge1}\frac{|b_j|\Gamma(|b_j|+M_j)}{\Gamma(|b_j|+M_j-\alpha)}\nonumber\\
&&+\frac{\Gamma(\theta+m+n)}{2\Gamma(1-\alpha)\Gamma(\theta+m+n-1+\alpha)}\sum_{j:|b_j|=1,|d_j|=0}1\nonumber\\
&=&\frac{\Gamma(\theta+m+n)}{2\Gamma(\theta+m+n+\alpha)}\bigg\{(\theta+m+n-1+\alpha)\sum_{j:|b_j|\ge2}\frac{|b_j|(|b_j|-1)\Gamma(|b_j|+M_j-1)}{\Gamma(|b_j|+M_j-\alpha)}\nonumber\\
&&\ \ \ \ +\frac{n}{\theta+m+n+\alpha}\bigg[-\frac{m(\theta+l\alpha)}{\Gamma(1-\alpha)}+(\theta+\alpha+n)\sum_{j=1}^l\frac{\Gamma(|b_j|+M_j+1)}{\Gamma(|b_j|+M_j-\alpha)}\bigg]\nonumber\\
&&\ \ \ \ -(\theta+2n-1+\alpha)\sum_{j:|b_j|\ge1}\frac{|b_j|\Gamma(|b_j|+M_j)}{\Gamma(|b_j|+M_j-\alpha)}\nonumber\\
&&\ \ \ \ +\frac{\theta+m+n-1+\alpha}{\Gamma(1-\alpha)}\sum_{j:|b_j|=1,|d_j|=0}1\bigg\}.
\end{eqnarray}}

By (\ref{May23a1}) and  (\ref{May23a2}), we obtain
{\small\begin{eqnarray*}
&&H_2(b,d)-H_1(b,d)\\
&=&\frac{\Gamma(\theta+m+n)}{2\Gamma(\theta+m+n+\alpha)}\bigg\{(\theta+\alpha+m+n-1)\sum_{j:|b_j|\ge2}\frac{|b_j|(|b_j|-1)\Gamma(|b_j|+M_j-1)}{\Gamma(|b_j|+M_j-\alpha)}\nonumber\\
&&\ \ \ \ +(n-m)\sum_{j=1}^l\frac{\Gamma(|b_j|+M_j+1)}{\Gamma(|b_j|+M_j-\alpha)}-(\theta+\alpha+2n-1)\sum_{j:|b_j|\ge1}\frac{|b_j|\Gamma(|b_j|+M_j)}{\Gamma(|b_j|+M_j-\alpha)}\nonumber\\
&&\ \ \ \ +\frac{\theta+\alpha+m+n-1}{\Gamma(1-\alpha)}\sum_{j:|b_j|=1,|d_j|=0}1
\end{eqnarray*}}
{\small\begin{eqnarray*}
&&\ \ \ \ -(\theta+\alpha+m+n-1)\sum_{j:|d_j|\ge1}\frac{\Gamma(|b_j|+M_j-1)}{\Gamma(|b_j|+M_j-\alpha)}M_j(M_j-1)\\
&&\ \ \ \ +(\theta+\alpha+2m-1)\sum_{j=1}^l\frac{\Gamma(|b_j|+M_j)}{\Gamma(|b_j|+M_j-\alpha)}M_j\bigg\}\\
&=&\frac{\Gamma(\theta+m+n)}{2\Gamma(\theta+m+n+\alpha)}\bigg\{(\theta+\alpha+m+n-1)\sum_{j=1}^l\frac{(|b_j|-M_j)\Gamma(|b_j|+M_j)}{\Gamma(|b_j|+M_j-\alpha)}\nonumber\\
&&\ \ \ \ +(n-m)\sum_{j=1}^l\frac{\Gamma(|b_j|+M_j+1)}{\Gamma(|b_j|+M_j-\alpha)}-(\theta+\alpha+2n-1)\sum_{j=1}^l\frac{|b_j|\Gamma(|b_j|+M_i)}{\Gamma(|b_j|+M_j-\alpha)}\nonumber\\
&&\ \ \ \ +(\theta+\alpha+2m-1)\sum_{j=1}^l\frac{M_j\Gamma(|b_j|+M_j)}{\Gamma(|b_j|+M_j-\alpha)}\bigg\}\\
&=&0.
\end{eqnarray*}}

\noindent Then,
\begin{eqnarray*}
\int_{{\cal P}_1(S)}({\cal L}^*F)G_{\psi,\dots,\psi_n}\,d
\Pi_{\alpha,\theta,\nu_0}=\int_{{\cal P}_1(S)}F{\cal L}G_{\psi,\dots,\psi_n}\,d
\Pi_{\alpha,\theta,\nu_0},
\end{eqnarray*}
which implies that $F\in D({\cal L})$  and
$$
{\cal L}F={\cal L}^*F.
$$
Hence, (\ref{May23a}) holds and
$$
{\cal A}(u,u)=-\int_{{\overline{\nabla}}_{\infty}}({\cal G}u)u\,d{\rm PD}(\alpha,\theta)=-\int_{{\cal P}_1(S)}\left[{\cal L}(u\circ\Phi)\right](u\circ\Phi)\, d
\Pi_{\alpha,\theta,\nu_0}={\cal E}(u\circ\Phi,u\circ\Phi ).
$$
Since $u\in{\cal P}$ is arbitrary and ${\cal P}$ is dense in $D({\cal A})$, the proof is complete.
\hfill $\Box$

\section{The finite dimensional model and approximation}\setcounter{equation}{0}

Ferguson's Dirichlet process has the partition property, namely, for any finite measurable partition $(A_1, \ldots, A_m)$ of $S$ the projection $(\Xi_{0,\theta,\nu_0}(A_1), \ldots, \Xi_{0,\theta, \nu_0}(A_m))$ follows the Dirichlet distribution with parameters $(\theta\nu_0(A_1), \ldots, \theta\nu_0(A_m))$. The same property does not hold in general for the Pitman-Yor process.  But for $\alpha=\frac{1}{2}$, the Pitman-Yor process is intimately connected to the excursions of one dimensional Brownian motion. The distributions of projections on any measurable partitions of $S$
have explicit probability density functions \cite{Car}. This makes it possible to consider the finite dimensional approximations to the infinite dimensional models.  The generators we studied in Sections 2 and 3 are motivated by these approximations.

Throughout this section, we assume that $\alpha=\frac{1}{2}$ and $\theta>-\alpha$.

\subsection{The finite dimensional model}

For $d\in\mathbb{N}$, define
$$
\Delta_d=\left\{x=(x_1,\dots, x_d)\in \mathbb{R}^{d}:x_i\ge 0\ {\rm
and}\ \sum_{i=1}^dx_i\le 1\right\}
$$
with $\Delta^{\circ}_{d}$ denoting the interior of $\Delta_d$. Let $(B_1,\dots,B_{d+1})$ be a measurable partition of $S$ with
$$\nu_0(B_i)=p_i=\frac{1}{d+1},\ \ \ \ 1\le i\le d+1.
$$
Set $x_{d+1}=1-\sum_{i=1}^dx_i$ and define
\begin{eqnarray}\label{May6Q1}
\varsigma_d(x_1,\dots,x_d)=\frac{\Gamma(\theta+\frac{d+1}{2})\prod_{i=1}^{d+1}p_i}{\pi^{\frac{d}{2}}\Gamma(\theta+\frac{1}{2})}\frac{\prod_{i=1}^{d+1}x_i^{-\frac{3}{2}}}{\left(\sum_{i=1}^{d+1}\frac{p_i^2}{x_i}\right)^{\theta+\frac{d+1}{2}}},\ \ \ \ (x_1,\dots, x_d)\in \Delta^{\circ}_d.
\end{eqnarray}
By \cite[Theorem 3.1]{C}, the joint density function of the random vector $(\mu(B_1),\dots,\mu(B_d))$ under $\Pi_{\alpha,\theta,\nu_0}$ is given by $\varsigma_d$.

Define
\begin{eqnarray}\label{May7a1}
U^{(d)}(x)=\frac{3}{2}\sum_{i=1}^{d+1}\log x_i+\left(\theta+\frac{d+1}{2}\right)\log\sum_{i=1}^{d+1}\frac{1}{x_i}.
\end{eqnarray}
Then,
$$
\varsigma_d(x)=\frac{e^{-U^{(d)}(x)}}{\int_{\Delta_d}e^{-U^{(d)}(y)}dy},\ \ \ \ x\in \Delta^{\circ}_d.
$$
For $x\in\Delta_d$, define
$$
H^{(d)}_{\frac{3}{2}}(x)=\sum_{i=1}^{d+1}x_i^{\frac{3}{2}},
$$
and
\begin{eqnarray*}
\left\{
\begin{array}{ll}
a^{(d)}_{ii}(x)=(1-2x_i)x_i^{\frac{3}{2}}+x_i^2H^{(d)}_{\frac{3}{2}}(x),&\ \  1\le i\le d,\\
a^{(d)}_{ij}(x)=x_ix_j\left[H^{(d)}_{\frac{3}{2}}(x)-x_i^{\frac{1}{2}}-x_j^{\frac{1}{2}}\right],&\ \ i\not=j,\ 1\le i,j\le d.
\end{array}
\right.
\end{eqnarray*}
Let $A^{(d)}(x)$ be the $d\times d$ symmetric matrix given by
\begin{eqnarray}\label{May8a}
A^{(d)}(x)&=&{\rm diagonal}\left\{x^{\frac{3}{2}}\right\}+H^{(d)}_{\frac{3}{2}}(x)(x_1,\dots,x_d)^T(x_1,\dots,x_d)\nonumber\\
&&-(x_1^{\frac{3}{2}},\dots,x_d^{\frac{3}{2}})^T(x_1,\dots,x_d)-(x_1,\dots,x_d)^T(x_1^{\frac{3}{2}},\dots,x_d^{\frac{3}{2}}).
\end{eqnarray}
Hereafter, $(x_1,\dots,x_d)^T$ denotes the transpose of the row vector $(x_1,\dots,x_d)$.
For any $y_1,\dots,y_d\in\mathbb{R}$, we have
\begin{eqnarray}\label{May17C}
&&(y_1,\dots,y_d)A^{(d)}(x)(y_1,\dots,y_d)^T\nonumber\\
&=&\sum_{i=1}^dx_i^{\frac{3}{2}}y_i^2+H^{(d)}_{\frac{3}{2}}(x)\left(\sum_{i=1}^dx_iy_i\right)^2-2\sum_{i=1}^dx_i^{\frac{3}{2}}y_i\sum_{i=1}^dx_iy_i\nonumber\\
&\ge&2\left(\sum_{i=1}^dx_i^{\frac{3}{2}}y_i^2\right)^{\frac{1}{2}}\left[H^{(d)}_{\frac{3}{2}}(x)\right]^{\frac{1}{2}}\left|\sum_{i=1}^dx_iy_i\right|-2\sum_{i=1}^dx_i^{\frac{3}{2}}y_i\sum_{i=1}^dx_iy_i\nonumber\\
&\ge&0.
\end{eqnarray}

Define
$$
C^{\infty}({\Delta_d})=\{f|_{\Delta_d}:f\in
C^{\infty}(\mathbb{R}^d)\}.
$$
Consider the following non-negative definite symmetric form on $L^2({\Delta_d}; \varsigma_d(x)dx)$:
\begin{eqnarray}\label{May6f}
{\cal E}^{(d)}(f,g)=\frac{1}{2}\sum_{i,j=1}^{d}\int_{\Delta_d}
a^{(d)}_{ij}(x)\partial_if(x)\partial_jg(x)\varsigma_d(x)dx,\ \ \ \
f,g\in  C^{\infty}({\Delta_d}).
\end{eqnarray}
Define
$$
2b^{(d)}(x)=\nabla\cdot A^{(d)}(x)-A^{(d)}(x)\nabla U^{(d)}(x).
$$
Then, informally, the generator of $({\cal E}^{(d)}, C^{\infty}({\Delta_d}))$ is given by
$$
{\cal L}^{(d)}f=\frac{1}{2}{\rm Tr}\left(A^{(d)}\nabla^2 f\right)+b^{(d)}\cdot\nabla f,\ \ \ \ f\in  C^{\infty}({\Delta_d}).
$$
We have
\begin{eqnarray*}
\sum_{k=1}^d\frac{\partial a^{(d)}_{1k}}{\partial x_k}
&=&\frac{3}{2}x_1^{\frac{1}{2}}-5x_1^{\frac{3}{2}}+2x_1H^{(d)}_{\frac{3}{2}}(x)+\frac{3}{2}x_1^{\frac{5}{2}}-\frac{3}{2}x_1^2x_{d+1}^{\frac{1}{2}}\\
&&+x_1\sum_{k=2}^d\left[H^{(d)}_{\frac{3}{2}}(x)-x_1^{\frac{1}{2}}-\frac{3}{2}x_k^{\frac{1}{2}}+\frac{3}{2}x_k^{\frac{3}{2}}-\frac{3}{2}x_kx_{d+1}^{\frac{1}{2}}\right]\\
&=&\frac{3}{2}x_1^{\frac{1}{2}}-\left(d+\frac{5}{2}\right)x_1^{\frac{3}{2}}+(d+1)x_1H^{(d)}_{\frac{3}{2}}(x)+\frac{3}{2}x_1^{\frac{5}{2}}-\frac{3}{2}x_1x_{d+1}^{\frac{1}{2}}(1-x_{d+1})\\
&&-\frac{3}{2}x_1\sum_{k=1}^dx_k^{\frac{1}{2}}+\frac{3}{2}x_1\left[H^{(d)}_{\frac{3}{2}}(x)-x_1^{\frac{3}{2}}-x_{d+1}^{\frac{3}{2}}\right]\\
&=&\frac{3}{2}x_1^{\frac{1}{2}}-\left(d+\frac{5}{2}\right)x_1^{\frac{3}{2}}+\left(d+\frac{5}{2}\right)x_1H^{(d)}_{\frac{3}{2}}(x)-\frac{3}{2}x_1\sum_{k=1}^{d+1}x_k^{\frac{1}{2}},
\end{eqnarray*}
and
\begin{eqnarray*}
-\sum_{k=1}^da^{(d)}_{1k}(x)\frac{\partial U^{(d)}}{\partial x_k}
&=&-\left[(1-2x_1)x_1^{\frac{3}{2}}+x_1^2H^{(d)}_{\frac{3}{2}}(x)\right]\left[\frac{3}{2x_1}-\frac{3}{2x_{d+1}}-\left(\theta+\frac{d+1}{2}\right)\frac{x_1^{-2}-x^{-2}_{d+1}}{\sum_{k=1}^{d+1}{x_k}^{-1}}\right]\\
&&-x_1\sum_{k=2}^dx_k\left[H^{(d)}_{\frac{3}{2}}(x)-x_1^{\frac{1}{2}}-x_k^{\frac{1}{2}}\right]\left[\frac{3}{2x_k}-\frac{3}{2x_{d+1}}-\left(\theta+\frac{d+1}{2}\right)\frac{x_k^{-2}-x^{-2}_{d+1}}{\sum_{k=1}^{d+1}{x_k}^{-1}}\right]\\
&=&-\frac{3\{x_1^{\frac{1}{2}}+dx_1[H^{(d)}_{\frac{3}{2}}(x)-x_1^{\frac{1}{2}}]-x_1\sum_{k=1}^dx_k^{\frac{1}{2}}\}}{2}+\frac{3x_1[x_1^{\frac{1}{2}}+x_{d+1}^{\frac{1}{2}}-H^{(d)}_{\frac{3}{2}}(x)]}{2}\\
&&+\left(\theta+\frac{d+1}{2}\right)x_1\frac{(1-2x_1)x_1^{\frac{1}{2}}(x_1^{-2}-x^{-2}_{d+1})-\sum_{k=2}^dx_k(x_1^{\frac{1}{2}}+x_k^{\frac{1}{2}})(x_k^{-2}-x^{-2}_{d+1})}{\sum_{k=1}^{d+1}{x_k}^{-1}}\\
&&+\left(\theta+\frac{d+1}{2}\right)x_1H^{(d)}_{\frac{3}{2}}(x)\frac{\sum_{k=1}^dx_k(x_k^{-2}-x^{-2}_{d+1})}{\sum_{k=1}^{d+1}{x_k}^{-1}}\\
&=&-\frac{3\{x_1^{\frac{1}{2}}-x_1x_{d+1}^{\frac{1}{2}}+(d+1)x_1[H^{(d)}_{\frac{3}{2}}(x)-x_1^{\frac{1}{2}}]-x_1\sum_{k=1}^dx_k^{\frac{1}{2}}\}}{2}\\
&&+\left(\theta+\frac{d+1}{2}\right)\frac{x_1[x_1^{-\frac{3}{2}}-\sum_{k=1}^dx_k^{-\frac{1}{2}}(1+x_1^{\frac{1}{2}}x_k^{-\frac{1}{2}})]}{\sum_{k=1}^{d+1}{x_k}^{-1}}\\
&&+\left(\theta+\frac{d+1}{2}\right)\frac{x_1[H^{(d)}_{\frac{3}{2}}(x)-x_{d+1}(x_1^{\frac{1}{2}}+x_{d+1}^{\frac{1}{2}})]}{x_{d+1}^2\sum_{k=1}^{d+1}{x_k}^{-1}}\\
&&+\left(\theta+\frac{d+1}{2}\right)x_1H^{(d)}_{\frac{3}{2}}(x)-\left(\theta+\frac{d+1}{2}\right)\frac{x_1H^{(d)}_{\frac{3}{2}}(x)}{x_{d+1}^2\sum_{k=1}^{d+1}{x_k}^{-1}}\\
&=&-\frac{3\{x_1^{\frac{1}{2}}+(d+1)x_1[H^{(d)}_{\frac{3}{2}}(x)-x_1^{\frac{1}{2}}]-x_1\sum_{k=1}^{d+1}x_k^{\frac{1}{2}}\}}{2}\\
&&+\left(\theta+\frac{d+1}{2}\right)\frac{x_1[x_1^{-\frac{3}{2}}-\sum_{k=1}^{d+1}x_k^{-\frac{1}{2}}(1+x_1^{\frac{1}{2}}x_k^{-\frac{1}{2}})]}{\sum_{k=1}^{d+1}{x_k}^{-1}}+\left(\theta+\frac{d+1}{2}\right)x_1H^{(d)}_{\frac{3}{2}}(x)\\
&=&-\frac{3x_1^{\frac{1}{2}}(1-x_1^{\frac{1}{2}}\sum_{k=1}^{d+1}x_k^{\frac{1}{2}})}{2}-(d+1-\theta)x_1\left[H^{(d)}_{\frac{3}{2}}(x)-x_1^{\frac{1}{2}}\right]\\
&&+\left(\theta+\frac{d+1}{2}\right)\frac{x_1(x_1^{-\frac{3}{2}}-\sum_{k=1}^{d+1}x_k^{-\frac{1}{2}})}{\sum_{k=1}^{d+1}{x_k}^{-1}}.
\end{eqnarray*}
Then,
\begin{eqnarray}\label{RTY1}
2b^{(d)}_1(x)
=\frac{x_1}{2}\left\{(3+2\theta) \left[H^{(d)}_{\frac{3}{2}}(x)-x_{1}^{\frac{1}{2}}\right]\right\}+\left(\theta+\frac{d+1}{2}\right)\frac{x_1(x_1^{-\frac{3}{2}}-\sum_{k=1}^{d+1}x_k^{-\frac{1}{2}})}{\sum_{k=1}^{d+1}{x_k}^{-1}}.
\end{eqnarray}
Similarly, for $i=2,\dots,d$,
\begin{eqnarray}\label{RTY2}
2b^{(d)}_i(x)
=\frac{x_i}{2}\left\{(3+2\theta) \left[H^{(d)}_{\frac{3}{2}}(x)-x_{i}^{\frac{1}{2}}\right]\right\}+\left(\theta+\frac{d+1}{2}\right)\frac{x_i(x_i^{-\frac{3}{2}}-\sum_{k=1}^{d+1}x_k^{-\frac{1}{2}})}{\sum_{k=1}^{d+1}{x_k}^{-1}}.
\end{eqnarray}
Thus, for $f\in  C^{\infty}({\Delta_d})$,
\begin{eqnarray}\label{Feb28a}
{\cal L}^{(d)}f(x)&=&
\frac{1}{2}\sum\limits_{i,j=1}^da^{(d)}_{ij}(x)\partial_i\partial_jf(x)+\frac{1}{4}\sum\limits_{i=1}^dx_i\bigg\{(3+2\theta) \left[H^{(d)}_{\frac{3}{2}}(x)-x_{i}^{\frac{1}{2}}\right]\nonumber\\
&&\ \ \ \ +\left(2\theta+d+1\right)\frac{x_i^{-\frac{3}{2}}-\sum_{k=1}^{d+1}x_k^{-\frac{1}{2}}}{\sum_{k=1}^{d+1}{x_k}^{-1}}\bigg\}\partial_if(x).
\end{eqnarray}

Note that  the term
\[
\frac{x_i^{-\frac{3}{2}}-\sum_{k=1}^{d+1}x_k^{-\frac{1}{2}}}{\sum_{k=1}^{d+1}{x_k}^{-1}}\]
is first defined on $\Delta^{\circ}_d$ and then extended to $\Delta_d$ by continuation. It is easily verified that the term becomes zero on $\Delta_d\setminus \Delta^{\circ}_d$.

\subsection{Uniform boundedness of the generator}

In this subsection, we show that the operator ${\cal L}^{(d)}$ given by (\ref{Feb28a}) is the generator of the bilinear form $({\cal E}^{(d)}, C^{\infty}({\Delta_d}))$ defined by (\ref{May6f}). To this end, we first give a uniform estimate for ${\cal L}^{(d)}$, which will play an important role in the next subsection.

\begin{pro}\label{pro3.1}
For $d=1$,
{\small\begin{eqnarray}\label{May6Z7}
&&\int_{\Delta_d}\left(\frac{\sum_{k=1}^{d+1}x_k^{-\frac{1}{2}}}{\sum_{k=1}^{d+1}{x_k}^{-1}}\right)^2\varsigma_d(x_1,\dots, x_d)dx_1\cdots dx_d\nonumber\\
&\le&\frac{1}{3}+\frac{1}{2^{\theta+8}\pi^{\frac{1}{2}}\Gamma(\theta+\frac{1}{2})(\theta+1)(\theta+2)}\left(\int_0^{\infty}u^{-\frac{\theta+5}{2}}e^{-\frac{1}{8u}}du\right)^2.
\end{eqnarray}}
For $d\ge 2$,
{\small\begin{eqnarray}\label{May6Z1}
&&\int_{\Delta_d}\left(\frac{\sum_{k=1}^{d+1}x_k^{-\frac{1}{2}}}{\sum_{k=1}^{d+1}{x_k}^{-1}}\right)^2\varsigma_d(x_1,\dots, x_d)dx_1\cdots dx_d\nonumber\\
&\le&\frac{2}{(d+1)(d+2)}+\frac{d}{2^{\theta+\frac{3}{2}}\pi\Gamma(\theta+\frac{1}{2})(d+1)(\theta+\frac{d+1}{2})(\theta+\frac{d+3}{2})}\int_0^{\infty}u^{-(\theta+\frac{5}{2})}e^{-\frac{1}{18u}}du.
\end{eqnarray}}
\end{pro}

\noindent {\bf Proof.}\ \ We have
\begin{eqnarray*}
&&\int_{\Delta_d}\left(\frac{\sum_{k=1}^{d+1}x_k^{-\frac{1}{2}}}{\sum_{k=1}^{d+1}{x_k}^{-1}}\right)^2\varsigma_d(x_1,\dots, x_d)dx_1\cdots dx_d\nonumber\\
&=&\int_{\Delta_d}\left[\frac{1}{\sum_{k=1}^{d+1}{x_k}^{-1}}+\frac{ \sum_{k\not= l} {x_k^{-\frac{1}{2}}x_l^{-\frac{1}{2}} }}{(\sum_{k=1}^{d+1}{x_k}^{-1})^2}\right]\varsigma_d(x_1,\dots, x_d)dx_1\cdots dx_d
\end{eqnarray*}
\begin{eqnarray}\label{May6Z2}
&\le&\frac{1}{\sum_{k=1}^{d+1}k}+d(d+1)\int_{\Delta_d}\frac{ {x_1^{-\frac{1}{2}}x_2^{-\frac{1}{2}} }}{(\sum_{k=1}^{d+1}{x_k}^{-1})^2}\varsigma_d(x_1,\dots, x_d)dx_1\cdots dx_d\nonumber\\
&=&\frac{2}{(d+1)(d+2)}+d(d+1)\int_{\Delta_d}\frac{ {x_1^{-\frac{1}{2}}x_2^{-\frac{1}{2}} }}{(\sum_{k=1}^{d+1}{x_k}^{-1})^2}\varsigma_d(x_1,\dots, x_d)dx_1\cdots dx_d.
\end{eqnarray}

Denote $
S_{d+1}=\{(x_1,\dots,
x_{d+1})\in \mathbb{R}^{d+1}:x_i\ge 0\ {\rm and}\
\sum_{i=1}^{d+1}x_i= 1\}.
$
We will follow the argument of \cite[Proof
of Lemma 3.1]{C}. Note that
\begin{eqnarray}\label{June2a}
\int_0^{\infty}\sigma^{-a-1}e^{-\frac{b}{\sigma}}d\sigma&=&\int_0^{\infty}u^{a+1}e^{-bu}u^{-2}du\nonumber\\
&=&\int_0^{\infty}u^{a-1}e^{-bu}du\nonumber\\
&=&\Gamma(a)b^{-a}.
\end{eqnarray}
Set
\begin{eqnarray}\label{June2b}
a=\theta+\frac{d+5}{2},\ \ \ \ b=\sum_{i=1}^{d+1}\frac{p_i^2}{x_i}.
\end{eqnarray}

\noindent (i) Assume that $d\ge2$. By (\ref{June2a}) and (\ref{June2b}), we get
{\small\begin{eqnarray*}
&&\int_{\Delta_d}\frac{x_1^{-\frac{1}{2}}x_2^{-\frac{1}{2}}}{\left(\frac{p_1^2}{x_1}+\cdots+
\frac{p_d^2}{x_{d}}+\frac{p^2_{d+1}}{1-x_1-\cdots-x_{d}}\right)^2}\varsigma_d(x_1,\dots, x_d)dx_1\cdots dx_d\nonumber\\
&=&\frac{\Gamma(\theta+\frac{d+1}{2})p_1\cdots
p_{d+1}}{\pi^{\frac{d}{2}}\Gamma(\theta+\frac{1}{2})}\int_{S_{d+1}}\frac{x_1^{-2}x_2^{-2}x_3^{-\frac{3}{2}}\cdots
x_{d+1}^{-\frac{3}{2}}}{b^{a}}dx_1\cdots dx_{d+1}\nonumber\\
&=&\frac{\Gamma(\theta+\frac{d+1}{2})p_1\cdots
p_{d+1}}{\pi^{\frac{d}{2}}\Gamma(\theta+\frac{1}{2})\Gamma(a)2^{a}}\int_{S_{d+1}}x_1^{-2}x_2^{-2}x_3^{-\frac{3}{2}}\cdots
x_{d+1}^{-\frac{3}{2}}\Gamma(a)\left(\frac{b}{2}\right)^{-a}dx_1\cdots dx_{d+1}\nonumber\\
&=&\frac{\Gamma(\theta+\frac{d+1}{2})p_1\cdots
p_{d+1}}{\pi^{\frac{d}{2}}\Gamma(\theta+\frac{1}{2})\Gamma(a)2^{a}}\int_{S_{d+1}}\int_0^{\infty}x_1^{-2}x_2^{-2}x_3^{-\frac{3}{2}}\cdots
x_{d+1}^{-\frac{3}{2}}\sigma^{-a-1}e^{-\frac{b}{2\sigma}}d\sigma dx_1\cdots dx_{d+1}\nonumber\\
&=&\frac{\Gamma(\theta+\frac{d+1}{2})p_1\cdots
p_{d+1}}{\pi^{\frac{d}{2}}\Gamma(\theta+\frac{1}{2})\Gamma(a)2^{a}}\int_{0}^{\infty}\cdots \int_{0}^{\infty}(\sigma^{-1}s_1)^{-2}(\sigma^{-1}s_2)^{-2}(\sigma^{-1}s_3)^{-\frac{3}{2}}\cdots
(\sigma^{-1}s_{d+1})^{-\frac{3}{2}}\nonumber\\
&&\ \ \ \ \cdot\sigma^{-a-1}e^{-\frac{b}{2\sigma}}\sigma^{-d}ds_1\cdots ds_{d+1}\ \ \ ({\rm by\ letting}\ s_j=\sigma x_j\ {\rm with}\ \sigma=s_1+\cdots+s_{d+1})\nonumber \\
&=&\frac{\Gamma(\theta+\frac{d+1}{2})p_1\cdots
p_{d+1}}{\pi^{\frac{d}{2}}\Gamma(\theta+\frac{1}{2})\Gamma(a)2^{a}}\int_{0}^{\infty}\cdots \int_{0}^{\infty}\sigma^{-(\theta+1)}s_1^{-2}s_2^{-2}\prod_{j>2}^{d+1}s_j^{-\frac{3}{2}}\prod_{j=1}^{d+1}e^{-\frac{p_j^2}{2s_j}}ds_1\cdots ds_{d+1}\\
&=&\frac{\Gamma(\theta+\frac{d+1}{2})p_1\cdots
p_{d+1}}{\pi^{\frac{d}{2}}\Gamma(\theta+\frac{1}{2})\Gamma(a)2^{a}}\int_0^{\infty}\int_0^{\infty}\Bigg[\int_0^{\infty}\cdots\int_0^{\infty}\sigma^{-(\theta+1)}\prod_{j>2}^{d+1}\left(s_j^{-\frac{3}{2}}e^{-\frac{p_j^2}{2s_j}}\right)ds_3\cdots ds_{d+1}\Bigg]\nonumber\\
& &\ \ \ \ \cdot
s_1^{-2}s_2^{-2}e^{-\frac{p_1^2}{2s_1}}e^{-\frac{p_2^2}{2s_2}}ds_1ds_2
\end{eqnarray*}}

{\small\begin{eqnarray}\label{May6Z4}
&=&\frac{p_1p_2}{2^{\theta+3}\pi^{\frac{1}{2}}\Gamma(\theta+\frac{1}{2})(\theta+\frac{d+1}{2})(\theta+\frac{d+3}{2})}\int_0^{\infty}\int_0^{\infty}\Bigg[\int_0^{\infty}(s_1+s_2+u)^{-(\theta+1)}\nonumber\\
& &\ \ \ \ \cdot
\frac{1-p_1-p_2}{\sqrt{2\pi}}e^{-\frac{(1-p_1-p_2)^2}{2u}}u^{-\frac{3}{2}}du\Bigg]s_1^{-2}s_2^{-2}e^{-\frac{p_1^2}{2s_1}}e^{-\frac{p_2^2}{2s_2}}ds_1ds_2.
\end{eqnarray}}

\noindent In the last equality, we have used the  fact that the sum of independent stable random variables with index $\frac{1}{2}$ and scale parameters $p_3,\dots,p_{d+1}$ is a stable random variable with index $\frac{1}{2}$ and scale parameter $p_3+\dots+p_{d+1}$. Thus,
{\small\begin{eqnarray}\label{May6Z3}
&&\int_{\Delta_d}\frac{x_1^{-\frac{1}{2}}x_2^{-\frac{1}{2}}}{\left(\frac{1}{x_1}+\cdots+
\frac{1}{x_{d}}+\frac{1}{1-x_1-\cdots-x_{d}}\right)^2}\varsigma_d(x_1,\dots, x_d)dx_1\cdots dx_d\nonumber\\
&\le&\frac{1}{2^{\theta+3}\pi^{\frac{1}{2}}\Gamma(\theta+\frac{1}{2})(d+1)^6(\theta+\frac{d+1}{2})(\theta+\frac{d+3}{2})}\int_0^{\infty}\int_0^{\infty}\Bigg[\int_0^{\infty}(s_1+s_2+u)^{-(\theta+1)}\nonumber\\
& &\ \ \ \ \cdot
\frac{1}{\sqrt{2\pi}}e^{-\frac{1}{18u}}u^{-\frac{3}{2}}du\Bigg]s_1^{-2}s_2^{-2}e^{-\frac{p_1^2}{2s_1}}e^{-\frac{p_2^2}{2s_2}}ds_1ds_2.
\end{eqnarray}}

Note that
{\small\begin{eqnarray*}
&&\int_0^{\infty}\int_0^{\infty}\Bigg[\int_0^{\infty}(s_1+s_2+u)^{-(\theta+1)}\frac{1}{\sqrt{2\pi}}e^{-\frac{1}{18u}}u^{-\frac{3}{2}}du\Bigg]s_1^{-2}s_2^{-2}e^{-\frac{p_1^2}{2s_1}}e^{-\frac{p_2^2}{2s_2}}ds_1ds_2\\
&\le&\frac{1}{\sqrt{2\pi}}\int_0^{\infty}u^{-(\theta+\frac{5}{2})}e^{-\frac{1}{18u}}du\left(\int_0^{\infty}s^{-2}e^{-\frac{1}{2(d+1)^2s}}ds\right)^2\\
&=&\frac{1}{\sqrt{2\pi}}\int_0^{\infty}u^{-(\theta+\frac{5}{2})}e^{-\frac{1}{18u}}du\left(\int_0^{\infty}e^{-\frac{t}{2(d+1)^2}}dt\right)^2\\
&=&\frac{4(d+1)^4}{\sqrt{2\pi}}\int_0^{\infty}u^{-(\theta+\frac{5}{2})}e^{-\frac{1}{18u}}du.
\end{eqnarray*}}

\noindent Then, we obtain (\ref{May6Z1}) by (\ref{May6Z2}) and (\ref{May6Z3}).

\noindent (ii) Assume that $d=1$. Similar to (\ref{May6Z4}), we get
\begin{eqnarray*}
&&\int_{\Delta_1}\frac{x^{-\frac{1}{2}}(1-x)^{-\frac{1}{2}}}{\left(\frac{1}{x}+
\frac{1}{1-x}\right)^2}\varsigma_1(x)dx\nonumber\\
&=&\frac{1}{2^{\theta+9}\pi^{\frac{1}{2}}\Gamma(\theta+\frac{1}{2})(\theta+1)(\theta+2)}\int_{0}^{\infty} \int_{0}^{\infty}(s_1+s_2)^{-(\theta+1)}s_1^{-2}s_2^{-2}e^{-\frac{1}{8s_1}}e^{-\frac{1}{8s_2}}ds_1ds_2\\
&\le&\frac{1}{2^{\theta+9}\pi^{\frac{1}{2}}\Gamma(\theta+\frac{1}{2})(\theta+1)(\theta+2)}\left(\int_0^{\infty}u^{-\frac{\theta+5}{2}}e^{-\frac{1}{8u}}du\right)^2,
\end{eqnarray*}
which together with (\ref{May6Z2}) implies (\ref{May6Z7}).
\hfill $\Box$

\vskip 0.3cm

For $x=(x_1,\dots,
x_d)\in \Delta_d$ and $1\le j\le d$,  define
$$
V_j(x)=\left(\sum_{i=1}^da^{(d)}_{ij}(x)\partial_if(x)\right)\varsigma_d(x)g(x),
$$
and $$V=(V_1,\dots,V_d).$$ Denote by $\partial\Delta_d$ the
boundary of $\Delta_d$, $\mathbf{n}$ the outward pointing unit
normal field of $\partial\Delta_d$, and $d{\cal S}_d$ the induced
volume form on the surface $\partial\Delta_d$.

For the face
$\{x=(x_1,\dots, x_d)\in \Delta_d:x_j=0\}$, $1\le j\le d$, we have
{\small\begin{eqnarray*}
V\cdot \mathbf{n}&=&V_j\\
&=&\bigg(\bigg[(1-2x_j)x_j^{\frac{3}{2}}+x_j^2H^{(d)}_{\frac{3}{2}}(x)\bigg]\partial_jf(x)+\sum_{i:\,i\not=
j}x_ix_j\bigg[H^{(d)}_{\frac{3}{2}}(x)-x_i^{\frac{1}{2}}-x_j^{\frac{1}{2}}\bigg]\partial_if(x)\bigg)\varsigma_d(x)g(x)\\
&=&0,
\end{eqnarray*}}

\noindent and for the face $\{x=(x_1,\dots, x_d)\in
\Delta_d:\sum_{j=1}^dx_j=1\}$, we have
\begin{eqnarray*}
V\cdot \mathbf{n}&=&\frac{1}{\sqrt{d}}\sum_{j=1}^dV_j\\
&=&\frac{1}{\sqrt{d}}\sum_{j=1}^d\Bigg(\left[(1-2x_j)x_j^{\frac{3}{2}}+x_j^2H^{(d)}_{\frac{3}{2}}(x)\right]\partial_jf(x)\\
&&\ \ \ \ +\sum_{i:\,i\not=
j}x_ix_j\left[H^{(d)}_{\frac{3}{2}}(x)-x_i^{\frac{1}{2}}-x_j^{\frac{1}{2}}\right]\partial_if(x)\Bigg)\varsigma_d(x)g(x)\\
&=&\frac{1}{\sqrt{d}}\Bigg(\sum_{j=1}^dx_j^{\frac{3}{2}}\left[1-2x_j+x_j^{\frac{1}{2}}H^{(d)}_{\frac{3}{2}}(x)\right]\partial_jf(x)\\
&&\ \ \ \ +\sum_{i\not=
j}x_ix_j\left[H^{(d)}_{\frac{3}{2}}(x)-x_i^{\frac{1}{2}}-x_j^{\frac{1}{2}}\right]\partial_if(x)\Bigg)\varsigma_d(x)g(x)\\
&=&\frac{1}{\sqrt{d}}\Bigg(\sum_{j=1}^dx_j^{\frac{3}{2}}\left[1-2x_j+x_j^{\frac{1}{2}}H^{(d)}_{\frac{3}{2}}(x)\right]\partial_jf(x)\\
&&\ \ \ \ -\sum_{i=1}^dx_i^{\frac{3}{2}}\left[1-2x_i+x_i^{\frac{1}{2}}H^{(d)}_{\frac{3}{2}}(x)\right]\partial_if(x)\Bigg)\varsigma_d(x)g(x)\\
&=&0.
\end{eqnarray*}
Hence, $V\cdot \mathbf{n}=0$ on $\partial\Delta_d$. Then, we obtain
by the divergence theorem and Proposition \ref{pro3.1} that
\begin{eqnarray*}
&&2\left[{\cal E}^{(d)}(f,g)+\int_{\Delta_d}{\cal L}^{(d)}f(x)g(x)\varsigma_d(x) dx\right]\nonumber\\
&=&\int_{\Delta_d}\sum_{j=1}^d\partial_{j}\left[\left(\sum_{i=1}^da^{(d)}_{ij}(x)\partial_if(x_1,\dots,
x_d)\right)\varsigma_d(x_1,\dots, x_d)g(x_1,\dots,
x_d)\right]dx_1\cdots dx_d
\end{eqnarray*}
\begin{eqnarray*}
\hskip -10.5cm&=&\int_{\Delta_d}{\rm div}Vdx_1\cdots dx_d\nonumber\\
\hskip -10.5cm&=&\int_{\partial\Delta_d}V\cdot \mathbf{n}d{\cal S}_d\nonumber\\
\hskip -10.5cm&=&0.
\end{eqnarray*}
Thus,
$$
{\cal E}^{(d)}(f,g)=-({\cal L}^{(d)}f,g)_{L^2(\Delta_d;\varsigma_d(x)dx)},\ \ \ \ f,g\in  C^{\infty}({\Delta_d}).
$$
Therefore, $({\cal E}^{(d)},C^{\infty}({\Delta_d}))$ is closable on $L^2({\Delta_d};\varsigma_d(x)dx)$ and its closure $({\cal E}^{(d)},D({\cal E}^{(d)}))$ is a regular Dirichlet form. There exists a time-reversible, conservative, diffusion process $(\{X^{(d)}_t\}_{t\ge 0},\{P^{(d)}_x\}_{x\in {\Delta_d}})$ associated with $({\cal E}^{(d)},D({\cal E}^{(d)}))$, whose stationary distribution  is $\varsigma_d(x)dx$.

\subsection{Convergence of the finite dimensional model}

Let $d\in\mathbb{N}$ and $i_1,\dots,i_{n_d}$ be arbitrary $n_d$ indices  satisfying $1\le i_1<i_2<\cdots <i_{n_d}\le d$. Define
$$h_d(x)=\sum_{j=1}^{n_d}x_{i_j},\ \ \ \ x=(x_1,\dots,x_d)\in\mathbb{R}^d.$$
By (\ref{Feb28a}), we get
\begin{eqnarray*}
{\cal L}^{(d)}h_d(x)=\frac{1}{4}\sum_{j=1}^{n_d}x_{i_j}\Bigg\{(3+2\theta) \left[H^{(d)}_{\frac{3}{2}}(x)-x_{i_j}^{\frac{1}{2}}\right]+\left(2\theta+d+1\right)\frac{(x_{i_j}^{-\frac{3}{2}}-\sum_{k=1}^{d+1}x_k^{-\frac{1}{2}})}{\sum_{k=1}^{d+1}{x_k}^{-1}}\Bigg\}.
\end{eqnarray*}
Define
$$
g_d(x)=\frac{2\theta+d+1}{4}\sum_{j=1}^{n_d}\frac{x_{i_j}^{-\frac{1}{2}}-x_{i_j}\sum_{k=1}^{d+1}x_k^{-\frac{1}{2}}}{\sum_{k=1}^{d+1}{x_k}^{-1}}.
$$
Then,
$$
{\cal L}^{(d)}h_d(x)=\frac{3+2\theta}{4}\sum_{j=1}^{n_d}x_{i_j}\left[H^{(d)}_{\frac{3}{2}}(x)-x_{i_j}^{\frac{1}{2}}\right]+g_d(x).
$$

We have
\begin{eqnarray*}
|g_d(x)|\le\frac{(2\theta+d+1)\sum_{k=1}^{d+1}x_k^{-\frac{1}{2}}}{4\sum_{k=1}^{d+1}{x_k}^{-1}}.
\end{eqnarray*}
Then,
\begin{eqnarray*}
&&\int_{\Delta_d}|g_d(x)|^2\varsigma_d(x_1,\dots, x_d)dx_1\cdots dx_d\nonumber\\
&\le& \frac{(2\theta+d+1)^2}{16}\int_{\Delta_d}\left(\frac{\sum_{k=1}^{d+1}x_k^{-\frac{1}{2}}}{\sum_{k=1}^{d+1}{x_k}^{-1}}\right)^2\varsigma_d(x_1,\dots, x_d)dx_1\cdots dx_d.
\end{eqnarray*}
Thus, by Proposition \ref{pro3.1}, we obtain
$$
\sup_{d\in\mathbb{N}}\left\{\int_{\Delta_d}|g_d(x)|^2\varsigma_d(x_1,\dots, x_d)dx_1\cdots dx_d\right\}<\infty.
$$
Therefore,
\begin{eqnarray}\label{May4v}
\sup_{d\in \mathbb{N}}\left\{\int_{\Delta_d}|{\cal L}^{(d)}h_d(x)|^2\varsigma_d(x_1,\dots, x_d)dx_1\cdots dx_d\right\}<\infty.
\end{eqnarray}

Let $({\cal E}, D({\cal E}))$  be the Dirichlet form in Theorem \ref{thm3.1} and $(\{X_t\}_{t\ge 0},\{P_\mu\}_{\mu\in {\cal P}_1(S)})$ be its associated Markov process. We fix a sequence
$\{(B^k_1,\dots,B^k_{2^k})\}_{k=1}^{\infty}$ of measurable partitions of $S$
satisfying the following conditions:

\noindent (i) $\nu_0(B^k_j)=1/2^k$, $1\le j\le 2^k$, $k\in\mathbb{N}$.

\noindent (ii) $B^k_j=B^{k+1}_{2j-1}\cup B^{k+1}_{2j}$, $1\le j\le 2^k$,
$k\in\mathbb{N}$.

\noindent Define
\begin{eqnarray*}
\Gamma_k: {\cal
P}_1(S)\mapsto\Delta_{2^k-1},\ \ \ \ \mu\mapsto\Gamma_k(\mu)=(\mu(B^k_1),\dots,\mu(B^k_{2^k-1})),
\end{eqnarray*}
$$
D({\tilde {\cal E}}^{(2^k-1)})=\left\{f\circ \Gamma_k: f\in D({\cal E}^{(2^k-1)})\right\},
$$
and
\begin{eqnarray*}
{\tilde {\cal E}}^{(2^k-1)}(f\circ \Gamma_k,g\circ \Gamma_k)={\cal E}^{(2^k-1)}(f,g),\ \ \ \ f,g\in D({\cal E}^{(2^k-1)}).
\end{eqnarray*}
Then, $({\tilde {\cal E}}^{(2^k-1)}, D({\tilde {\cal E}}^{(2^k-1)}))$ is a closed symmetric form on $L^2({\cal P}_1(S);\Pi_{\alpha,\theta,\nu_0})$, which is obtained by lifting the Dirichlet form $({\cal E}^{(2^k-1)},D({\cal E}^{(2^k-1)}))$ from $L^2({\Delta_{2^k-1}};\varsigma_{2^k-1}(x)dx)$ to $L^2({\cal P}_1(S);\Pi_{\alpha,\theta,\nu_0})$. Define
$$
{\tilde X}^{(2^k-1)}_t=\sum_{j=1}^{2^k-1}\{X^{(2^k-1)}_t\}_j\delta_{b^k_j}+\left[1-\sum_{j=1}^{2^k-1}\{X^{(2^k-1)}_t\}_j\right]\delta_{b^k_{2^k}}.
$$
Then,  $(\{{\tilde X}^{(2^k-1)}_t\}_{t\ge 0}, \{P^{(2^k-1)}_{\Gamma_k(\mu)}\}_{\mu\in{\cal P}_1(S)})$ is the Markov process associated with $({\tilde {\cal E}}^{(2^k-1)}, D({\tilde {\cal E}}^{(2^k-1)}))$.

Let $f\in C^{\infty}(\mathbb{R}^{2^{k_0}-1})$ for some $k_0\in\mathbb{N}$. For  $k\ge k_0$, define $f^{(k)}\in C(\mathbb{R}^{2^{k}-1})$ by
$$
f^{(k)}(x_1,\dots,x_{2^k-1})=f\left(\sum_{j=1}^{2^{k-k_0}}x_j,\sum_{j=2^{k-k_0}+1}^{2^{k-k_0+1}}x_j,\dots,\sum_{j=2^k-2^{k-k_0+1}+1}^{2^k-2^{k-k_0}}x_j\right).
$$
Then,
$$
f\circ \Gamma_{k_0}(\mu)=f^{(k)}\circ \Gamma_k(\mu),\ \ \ \ \mu\in {\cal P}_1(S).
$$
Define
\begin{eqnarray*}
\Psi_k: {\cal
P}_1(S)\mapsto{\cal
P}_1(S),\ \ \ \ \mu\mapsto\Psi_k(\mu)=\sum_{j=1}^{2^k}\mu(B^k_j)\delta_{b^k_j}.
\end{eqnarray*}
Suppose that $f\in C^{\infty}(\mathbb{R}^{2^{k_0}-1})$ and  $g\in C^{\infty}(\mathbb{R}^{2^{k_1}-1})$ for some $k_0,k_1\in\mathbb{N}$.  Then, for any $k\ge \max\{k_0,k_1\}$, we have
$$
{\tilde {\cal E}}^{(2^k-1)}(f^{(k)}\circ \Gamma_k,g^{(k)}\circ \Gamma_k)=\frac{1}{2}\int_{{\cal P}_1(S)}[ \nabla
(f\circ \Gamma_{k_0})(\mu),\nabla (g\circ \Gamma_{k_1})(\mu)]_{\Psi_k(\mu)}
\Pi_{\alpha,\theta,\nu_0}(d\mu).
$$
Thus,
\begin{eqnarray*}
-({\cal L}(f\circ \Gamma_{k_0}),g\circ \Gamma_{k_1})_{L^2({\cal P}_1(S);\Pi_{\alpha,\theta,\nu_0})}&=&{\cal E}(f\circ \Gamma_{k_0}, g\circ \Gamma_{k_1})\\
&=&\lim_{k\rightarrow\infty}{\tilde {\cal E}}^{(2^k-1)}(f^{(k)}\circ \Gamma_k,g^{(k)}\circ \Gamma_k)\\
&=&\lim_{k\rightarrow\infty}-({\cal L}^{(2^k-1)}f^{(k)},g^{(k)})_{L^2(\Delta_{2^k-1};\varsigma_{2^k-1}(x)dx)}\\
&=&\lim_{k\rightarrow\infty}-(({\cal L}^{(2^k-1)}f^{(k)})\circ \Gamma_k,g\circ \Gamma_{k_1})_{L^2({\cal P}_1(S);\Pi_{\alpha,\theta,\nu_0})}.
\end{eqnarray*}
By the representation (\ref{Feb28a}) of the operator ${\cal L}^{(d)}$, the uniform estimate (\ref{May4v}) and noticing the arbitrariness of functions $\{h_d\}$, we obtain 
\begin{eqnarray}\label{May6L3}
\sup_{k\in\mathbb{N}}\left\{\left\|({\cal L}^{(2^k-1)}f^{(k)})\circ \Gamma_k\right\|_{L^2({\cal P}_1(S);\Pi_{\alpha,\theta,\nu_0})}\right\}<\infty.
\end{eqnarray}
Since $\{g\circ \Gamma_{k_1}:g\in C^{\infty}(\mathbb{R}^{2^{k_1}-1}),k_1\in\mathbb{N}\}$ is a dense subset of $L^2({\cal P}_1(S);\Pi_{\alpha,\theta,\nu_0})$, by (\ref{May6L3}), we get
\begin{eqnarray}\label{May4df}
({\cal L}^{(2^k-1)}f^{(k)})\circ \Gamma_k\ {\rm converges\ weakly\ to}\ {\cal L}(f\circ \Gamma_{k_0})\ {\rm in}\  L^2({\cal P}_1(S);\Pi_{\alpha,\theta,\nu_0})\ {\rm as}\ k\rightarrow\infty.
\end{eqnarray}

Suppose that $g_0\in C^{\infty}(\mathbb{R}^{2^{k_0}-1})$ for some $k_0\in\mathbb{N}$. Then, we have
\begin{eqnarray}\label{May6L1}
{\cal E}(g_0\circ \Gamma_{k_0},g_0\circ \Gamma_{k_0})=\lim_{k\rightarrow\infty}{\tilde {\cal E}}^{(2^k-1)}(g_0\circ \Gamma_{k_0},g_0\circ \Gamma_{k_0}).
\end{eqnarray}
Let  $g_k\in C^{\infty}(\mathbb{R}^{2^{k}-1})$, $k\ge 1$, satisfying $\lim_{k\rightarrow\infty}g_k\circ \Gamma_k=g_0$ in $L^2({\cal P}_1(S);\Pi_{\alpha,\theta,\nu_0})$. By (\ref{May4df}), we get
{\small\begin{eqnarray*}
{\cal E}(g_0\circ \Gamma_{k_0},g_0\circ \Gamma_{k_0})
&=&-({\cal L}(g_0\circ \Gamma_{k_0}),g_0\circ \Gamma_{k_0})_{L^2({\cal P}_1(S);\Pi_{\alpha,\theta,\nu_0})}\\
&=&\lim_{k\rightarrow\infty}-({\cal L}^{(2^k-1)}(g_0\circ \Gamma_{k_0}),g_0\circ \Gamma_{k_0})_{L^2({\cal P}_1(S);\Pi_{\alpha,\theta,\nu_0})}\\
&=&\lim_{k\rightarrow\infty}-({\cal L}^{(2^k-1)}(g_0\circ \Gamma_{k_0}),g_k\circ \Gamma_k)_{L^2({\cal P}_1(S);\Pi_{\alpha,\theta,\nu_0})}\\
&=&\lim_{k\rightarrow\infty}{\tilde {\cal E}}^{(2^k-1)}(g_0\circ \Gamma_{k_0},g_k\circ \Gamma_k)\\
&\le&\liminf_{k\rightarrow\infty}\Bigg\{\left[{\tilde {\cal E}}^{(2^k-1)}(g_0\circ \Gamma_{k_0},g_0\circ \Gamma_{k_0})\right]^{\frac{1}{2}}\left[{\tilde {\cal E}}^{(2^k-1)}(g_k\circ \Gamma_k,g_k\circ \Gamma_k)\right]^{\frac{1}{2}}\Bigg\}\\
&=&\left[{\cal E}(g_0\circ \Gamma_{k_0},g_0\circ \Gamma_{k_0})\right]^{\frac{1}{2}}\liminf_{k\rightarrow\infty}\left[{\tilde {\cal E}}^{(2^k-1)}(g_k\circ \Gamma_k,g_k\circ \Gamma_k)\right]^{\frac{1}{2}},
\end{eqnarray*}}

\noindent which implies that
\begin{eqnarray}\label{May6L2}
{\cal E}(g_0\circ \Gamma_{k_0},g_0\circ \Gamma_{k_0})\le \liminf_{k\rightarrow\infty}{\tilde {\cal E}}^{(2^k-1)}(g_k\circ \Gamma_k,g_k\circ \Gamma_k).
\end{eqnarray}
By (\ref{May6L1}) and (\ref{May6L2}), we deduce that $({\tilde {\cal E}}^{(2^k-1)}, D({\tilde {\cal E}}^{(2^k-1)}))$ converges to $({\cal E}, D({\cal E}))$ in the Mosco convergence, which is equivalent  to convergence of their semigroups by \cite[Corollary 2.6.1]{M}. Define
$$
P^{(2^k-1)}_{\Pi_{\alpha,\theta,\nu_0}}(\cdot)=\int_{{\cal P}_1(S)}P^{(2^k-1)}_{\Gamma_k(\mu)}(\cdot)\Pi_{\alpha,\theta,\nu_0}(d\mu),\ \ \ \ k\in\mathbb{N},
$$
and
$$
P_{\Pi_{\alpha,\theta,\nu_0}}(\cdot)=\int_{{\cal P}_1(S)}P_{\mu}(\cdot)\Pi_{\alpha,\theta,\nu_0}(d\mu).
$$
Therefore, we have proved the following result.
\begin{pro}
Any finite dimensional distribution of the Markov process $\{{\tilde X}^{(2^k-1)}_t\}_{t\ge 0}$ under the probability measure $P^{(2^k-1)}_{\Pi_{\alpha,\theta,\nu_0}}(\cdot)$ converges to that of $\{X_t\}_{t\ge 0}$ under the probability measure $P_{\Pi_{\alpha,\theta,\nu_0}}(\cdot)$ as $k\rightarrow\infty$.
\end{pro}

\subsection{Markov chain model and diffusion approximation}

Denote by
\[
h(y; \beta)= \frac{\beta}{\sqrt{2\pi}}y^{-\frac{3}{2}}e^{\frac{-\beta^2}{2y}},\ \ \ \ y>0
\]
the density function of the inverse Gaussian random variable with parameter $\beta$. For $d\in\mathbb{N}$, let $Y_1, \ldots, Y_{d+1}$ be independent inverse Gaussian random variables with respective parameters $p_1, \ldots, p_{d+1}$ satisfying $p_i>0$, $p_1+\cdots+p_{d+1}=1$. It is known that (cf. \cite{Car})
\[
(X_1, \ldots, X_{d+1})=\frac{1}{\sum_{i=1}^{d+1}Y_i}(Y_1, \ldots, Y_{d+1})
\]
has the two-parameter Dirichlet distribution with parameters $\alpha=\frac{1}{2}$ and $\nu=(p_1, \ldots, p_{d+1})$. Here we have $\theta=0$.

Consider the following independent  stochastic differential equations:
\[
d\,Y_i(t)=\frac{p_i^2}{\sqrt{Y_i(t)}}d\,t + Y^{\frac{3}{4}}_i(t) dB_i(t),\ \ \ \ 1\le i\le d+1,
\]
where $(B_1(t),\dots,B_{d+1}(t))$ is a standard $(d+1)$-dimensional Brownian motion.
Then, the diffusion coefficient in the above model corresponds to the diffusion
\[
(X_1(t), \ldots, X_{d+1}(t))=\frac{1}{\sum_{i=1}^{d+1}Y_i(t)}(Y_1(t), \ldots, Y_{d+1}(t))\]
conditioned on $\sum_{i=1}^{d+1}Y_i(t)=1$. This observation motivates us to consider a Markov chain and diffusion approximation for the model introduced in Section 3 when $\alpha=\frac{1}{2}$ and $\theta>-\alpha$.

Let $d,N\in\mathbb{N}$. We consider a population of $2N$ haploid individuals with some type space $\{A_1,\dots,A_{d+1}\}$. The time index is $\{0,1,2,\dots\}$. There will be no overlap between generations.
Let $X(0)=(X_1(0),\dots,X_{d+1}(0))$ with $X_i(0)$ denoting the number of type $A_i$ individuals at time zero, $1\le i\le d+1$.  Each individual will evolve under the influence of mutation and random
sampling. For $z=(z_1,\dots,z_{d+1})\in \mathbb{Z}^{d+1}$ satisfying $z_i\ge0$ and $\sum_{i=1}^{d+1}z_i=2N$, mutation is described by the following state-dependent matrix $(u_{ij}(z))_{1\le i,j\le d+1}$:
$$
u_{ij}(z)=\frac{1}{4}\left\{(3+2\theta)\left[\frac{z_j}{2N}-\left(\frac{z_j}{2N}\right)^{\frac{3}{2}}\right]+(2\theta+d+1)\frac{(\frac{z_j}{2N})^{-\frac{1}{2}}}{\sum_{k=1}^{d+1}(\frac{z_k}{2N})^{-1}}\right\},\ \ \ \ i\not=j.
$$
Let $p_i=\frac{X_i(0)}{2N}$, $1\le i\le d+1$. Then, for $1\le i\le d+1$, the proportion of type $A_i$ individuals after the mutation is
{\small\begin{eqnarray*}
p'_i&=&\sum_{j:j\not=i}^{d+1}p_ju_{ji}(X(0))-p_i\sum_{j:j\not=i}u_{ij}(X(0))\\
&=&\frac{1}{4}\Bigg\{(3+2\theta) \left[\frac{X_i(0)}{2N}-\left(\frac{X_i(0)}{2N}\right)^{\frac{3}{2}}\right]+\left(2\theta+d+1\right)\frac{(\frac{X_i(0)}{2N})^{-\frac{1}{2}}}{\sum_{k=1}^{d+1}(\frac{X_k(0)}{2N})^{-1}}\Bigg\}\\
&&-\frac{1}{4}\frac{X_i(0)}{2N}\Bigg\{(3+2\theta) \left[1-\sum_{k=1}^{d+1}\left(\frac{X_k(0)}{2N}\right)^{\frac{3}{2}}\right]+\left(2\theta+d+1\right)\frac{\sum_{k=1}^{d+1}(\frac{X_k(0)}{2N})^{-\frac{1}{2}}}{\sum_{k=1}^{d+1}(\frac{X_k(0)}{2N})^{-1}}\Bigg\},
\end{eqnarray*}}

\noindent which is equal to $b^{(d)}_i(\frac{X_1(0)}{2N},\dots,\frac{X_d(0)}{2N})$ for $1\le i\le d$ (cf. ({\ref{RTY1}) and (\ref{RTY2})).

The next generation $X(1)=(X_1(1),\dots,X_{d+1}(1))$ is obtained through the sampling mechanism that we will now describe.  Let the matrix  $A^{(d)}(x)$ be given by (\ref{May8a}), $x\in \Delta_d$, and let \(e_1, \dots, e_d\) be the standard basis row vectors and \(e_{d+1} =0\) (the origin). Note that
$$
A^{(d)}(x)=\sum_{i=1}^{d+1}x_i^{\frac{3}{2}}(e_i-x)^T(e_i-x).
$$
Then, we have
$$
{\rm diagonal}\left\{x\right\}-x^Tx-A^{(d)}(x)=\sum_{i=1}^{d+1}\left(x_i-x_i^{\frac{3}{2}}\right)(e_i-x)^T(e_i-x),
$$
which is a non-negative definite matrix. Thus, there exists a random vector $W=(W_1,\dots,W_d)$ supported on $\Delta_d$ such that its mean is $x=(\frac{X_1(0)}{2N},\dots,\frac{X_d(0)}{2N})$ and covariance matrix is $A^{(d)}(x)$ (cf. \cite{B}). Hence, we can generate the next generation $X(1)=(X_1(1),\dots,X_{d+1}(1))$ by letting
$$
X_i(1)=[2NW_1],\dots,X_i(d)=[2NW_d],1\le i\le d;\ \ X_{d+1}(1)=1-\sum_{i=1}^dX_i(1).
$$
Hereafter, $[w]$  denotes the integer part of  $w\in\mathbb{R}$. The general $n$-th generation can be obtained accordingly. The process $\{X(n)\}$ is a Markov chain with state space $\{A_1,\dots,A_{d+1}\}$.

Next we construct diffusion approximation of $\{X(n)\}$. For distinct $i,j$, change $(u_{ij})$ to $(\frac{u_{ij}}{2N})$. Set
\begin{eqnarray*}
\left\{
\begin{array}{ll}
Y^N(t)=(Y^N_1(t),\dots,Y^N_{d+1}(t)),\ \ \ \ Y^N_i(t)=X_i([2Nt]),\\
p^N(t)=(p^N_1(t),\dots,p^N_{d+1}(t))=\frac{Y^N(t)}{2N}.
\end{array}
\right.
\end{eqnarray*}
Define
$$
\Delta p^N(t)=p^N\left(t+\frac{1}{2N}\right)-p^N(t).
$$
Then, following the argument in \cite[Section 1.2]{F}, we can show that when the mutation rate is
scaled by a factor of $\frac{1}{2N}$ and time is counted in units of $2N$ generations the Markov chain model is approximated by the following diffusion process:
$$
dp_i(t)=b^{(d)}_i(p(t))dt+\sum_{j=1}^d\sigma^{(d)}_{ij}(p(t))dB_j(t),\ \ \ \ 1\le i\le d,
$$
with
$$
\sum_{j=1}^d\sigma^{(d)}_{il}(p(t))\sigma^{(d)}_{jl}(p(t))=a^{(d)}_{ij}(p(t)).
$$
The generator of the diffusion is ${\cal L}^{(d)}$ given by (\ref{Feb28a}).

\section{Dynamical models with general diffusion coefficient indices}\setcounter{equation}{0}

The infinite dimensional diffusion matrix \(a=(a_{ij})_{i,j=1}^{\infty}\) with index \(\alpha \) (see (\ref{May5D1})) is central to constructing dynamical models for \({\rm PD}(\alpha,\theta)\) and \(\Pi _{\alpha ,\theta ,\nu _{0}}\). Extending this parameter from \(\alpha\in(0,1)\) to a general index \(\gamma\in(0,\infty)\) reveals a much clearer picture of the underlying dynamics.

Let $\gamma\ge0$.  For $x\in \overline{\nabla}_{\infty}$, define
\begin{eqnarray*}
\left\{
\begin{array}{ll}
a^{(\gamma)}_{ii}(x)=(1-2x_i)x_i^{1+\gamma}+x_i^2H_{1+\gamma}(x),&\ \  i\in\mathbb{N},\\
a^{(\gamma)}_{ij}(x)=x_ix_j\left[H_{1+\gamma}(x)-x^{\gamma}_i-x^{\gamma}_j\right],&\ \ i\not=j,\ i,j\in\mathbb{N}.
\end{array}
\right.
\end{eqnarray*}
Similar to (\ref{May17C}), we can show that $a^{(\gamma)}=(a^{(\gamma)}_{ij})_{i,j=1}^{\infty}$ is non-negative definite. For $\alpha\in[0,1)$ and $\theta>-\alpha$, consider the following symmetric form on $L^2({\overline{\nabla}}_{\infty};{\rm PD}(\alpha,\theta))$:
\begin{eqnarray}\label{FGHJ}
{\cal A}^{\gamma,\alpha,\theta}(u,v)=\frac{1}{2}\sum_{i,j=1}^{\infty}\int_{{\overline{\nabla}}_{\infty}}a^{(\gamma)}_{ij}(x)\partial_i u(x)\partial_j v(x) {\rm PD}(\alpha,\theta)(dx),\ \ \ \ u,v\in{\cal P}.
\end{eqnarray}
We have
\begin{eqnarray*}
{\cal A}^{\gamma,\alpha,\theta}(u,u)=\frac{1}{2}\int_{{\overline{\nabla}}_{\infty}}\sum_{k=1}^{\infty}x_k^{1+\gamma}\left|\sum_{i=1}^{\infty}(\delta_{ki}-x_i)\partial_i u(x)\right|^2 {\rm PD}(\alpha,\theta)(dx),\ \ \ \ u\in{\cal P}.
\end{eqnarray*}
Then, for any $\alpha\in[0,1)$, $\theta>-\alpha$ and $u\in{\cal P}$,
\begin{eqnarray}\label{jkl0}
{\cal A}^{\gamma,\alpha,\theta}(u,u)\ {\rm is\ a\ decreasing\ function\ of}\ \gamma\in[0,\infty).
\end{eqnarray}

In Section 2, we showed that if \(\gamma=\alpha\), then \(({\cal A}^{\gamma,\alpha,\theta}, {\cal P})\) is closable, thereby establishing the corresponding dynamical model. Below, we demonstrate that \(({\cal A}^{\gamma,\alpha,\theta}, {\cal P})\) remains closable for \(\gamma \ge 1\), and we explicitly characterize its generator.

\subsection{The infinite dimensional unlabelled model}

Assume that $\gamma\ge1$. For $u\in {\cal P}$, define
\begin{eqnarray}\label{May20a}
{\cal G}^{\gamma,\alpha,\theta}u(x)&=&\frac{1}{2}\sum\limits_{i,j=1}^\infty a^{(\gamma)}_{ij}(x)\partial_i\partial_ju(x)-\frac{1}{2}\sum\limits_{i=1}^\infty x_i\bigg\{ (\theta+\gamma+1)\left[x_{i}^{\gamma}-H_{1+\gamma}(x)\right]\nonumber\\
&&\ \ \ \ -(\gamma-\alpha) \left[x_{i}^{\gamma-1}-H_{\gamma}(x)\right]\bigg\}\partial_iu(x),\ \ \ \ x\in {\overline{\nabla}}_{\infty}.
\end{eqnarray}

\begin{thm}\label{thm54} Let $\gamma\ge1$, $\alpha\in[0,1)$ and $\theta>-\alpha$.

\noindent  {\rm (i)} For any $u,v\in{\cal P}$,
\begin{eqnarray*}
{\cal A}^{\gamma,\alpha,\theta}(u,v)=-\int_{{\overline{\nabla}}_{\infty}}({\cal G}^{\gamma,\alpha,\theta}u) v\,d{\rm PD}(\alpha,\theta).
\end{eqnarray*}

\noindent {\rm (ii)}  The bilinear form $({\cal A}^{\gamma,\alpha,\theta}, {\cal P})$ is closable on $L^2({\overline{\nabla}}_{\infty};{\rm PD}(\alpha,\theta))$ and its closure  $({\cal A}^{\gamma,\alpha,\theta}, D({\cal A}^{\gamma,\alpha,\theta}))$ is a regular Dirichlet form.

\noindent {\rm (iii)}  There exists a time-reversible, conservative, diffusion process $(\{X^{\gamma,\alpha,\theta}_t\}_{t\ge 0},\{P^{\gamma,\alpha,\theta}_x\}_{x\in {\overline{\nabla}}_{\infty}})$ which is associated with $({\cal A}^{\gamma,\alpha,\theta}, D({\cal A}^{\gamma,\alpha,\theta}))$ and its stationary distribution is  ${\rm PD}(\alpha,\theta)$. Moreover,  there exists an ${\cal A}^{\gamma,\alpha,\theta}$-exceptional set $M\subset {\overline{\nabla}}_{\infty}$ such that for any $x\notin M$,
\begin{eqnarray*}
P^{\gamma,\alpha,\theta}_x\left\{\sum_{i=1}^{\infty}X^{\gamma,\alpha,\theta}_i(t)=1\ {\rm for\ all}\ t\in[0,\infty)\right\}=1.
\end{eqnarray*}

\noindent {\rm (iv)}  $({\cal A}^{1,\alpha,\theta}, D({\cal A}^{1,\alpha,\theta}))$ is irreducible recurrent and $(\{X^{1,\alpha,\theta}_t\}_{t\ge 0},\{P^{1,\alpha,\theta}_x\}_{x\in {\overline{\nabla}}_{\infty}})$ satisfies the strong law of large numbers, i.e., there exists an ${\cal A}^{1,\alpha,\theta}$-exceptional set $M\subset {\overline{\nabla}}_{\infty}$ such that for any $x\notin M$,
\begin{eqnarray*}
\lim_{t\rightarrow\infty}\frac{1}{t}\int_0^tf(X^{1,\alpha,\theta}_s)ds=\int_{{\overline{\nabla}}_{\infty}}f\,d{\rm PD}(\alpha,\theta)\ \ P^{1,\alpha,\theta}_x\text{-}a.s.,\ \ \ \ \forall f\in L^1({\overline{\nabla}}_{\infty};{\rm PD}(\alpha,\theta)).
\end{eqnarray*}
\end{thm}
\vskip 0.3cm
\noindent {\bf Proof.}\ \ The proof of Parts (i)--(iii) is similar to that of Theorem \ref{thm2.1}. We omit the details here. In what follows, we only prove (iv). Unlike the one-parameter setting, proving the ergodicity of \(({\cal A}^{1,\alpha,\theta}, D({\cal A}^{1,\alpha,\theta}))\) via standard induction on lower-degree monomials in \(\mathcal{P}\) fails. Addressing this requires leveraging the entire collection of operators \({\cal G}^{\gamma,\alpha,\theta}\) across \(\gamma\in\mathbb{N}\).

Let $f$ be an arbitrary function in $D({\cal A}^{1,\alpha,\theta})$ satisfying ${\cal A}^{1,\alpha,\theta}(f,f)=0$. We will show that $f$ must be a constant.  Once this is proved, the assertions follow from \cite[Sections 1.6 and 4.7]{Fuku2}. Denote $f_n=\max\{\min\{f,n\},-n\}$. Note that
$$
{\cal A}^{1,\alpha,\theta}(f_n,f_n)\le {\cal A}^{1,\alpha,\theta}(f,f)
$$
and ${\cal A}^{1,\alpha,\theta}(1,1)=0$, hence we can assume without loss of generality that $f$ is bounded, non-negative and $f\not=0$.

By (\ref{jkl0}), we get
$$
f\in D({\cal A}^{\gamma,\alpha,\theta}),\ {\cal A}^{\gamma,\alpha,\theta}(f,f)=0,\ \forall \gamma\in\mathbb{N}.
$$
Then,
$$
\int_{{\overline{\nabla}}_{\infty}}({\cal G}^{\gamma,\alpha,\theta}u)f\,
d{\rm PD}(\alpha,\theta)=0,\ \ \ \  u\in{\cal P}, \gamma\in\mathbb{N}.
$$
Let
\begin{eqnarray}\label{May20g}
g\in\{1,f\}.
\end{eqnarray}
Then, we have
\begin{eqnarray}\label{May20b}
\int_{{\overline{\nabla}}_{\infty}}({\cal G}^{\gamma,\alpha,\theta}u)g\,
d{\rm PD}(\alpha,\theta)=0,\ \ \ \  u\in{\cal P}, \gamma\in\mathbb{N}.
\end{eqnarray}

For $\gamma\ge1$ and $n\ge2$, by (\ref{May20a}), we get
\begin{eqnarray}\label{A}
{\cal G}^{\gamma,\alpha,\theta}H_{n}
&=&\frac{n}{2}\bigg\{\left[(\theta+\gamma+n)H_{1+\gamma}-(\gamma-\alpha)H_{\gamma}\right]H_{n}\nonumber\\
&&\ \ \ \ \ \ +(n-1+\gamma-\alpha)H_{n-1+\gamma}-\left(\theta+\gamma+2n-1\right)H_{n+\gamma}\bigg\}.
\end{eqnarray}
Then,
\begin{eqnarray*}
{\cal G}^{1,\alpha,\theta}H_{n}
=\frac{n}{2}\bigg\{(\theta+1+n)H_{2}H_{n}+(n-1)H_{n}-\left(\theta+2n\right)H_{n+1}\bigg\},
\end{eqnarray*}
which together with (\ref{May20b}) implies that
\begin{eqnarray}\label{F}
\int_{{\overline{\nabla}}_{\infty}}H_{2}H_{n}g\,
d{\rm PD}(\alpha,\theta)=\frac{1}{\theta+1+n}\int_{{\overline{\nabla}}_{\infty}}\left[\left(\theta+2n\right)H_{n+1}-(n-1)H_{n}\right]g\,
d{\rm PD}(\alpha,\theta).
\end{eqnarray}
Thus,
\begin{eqnarray}\label{May20d}
\int_{{\overline{\nabla}}_{\infty}}H_{2}H_{n+1}g\,
d{\rm PD}(\alpha,\theta)=\frac{1}{\theta+2+n}\int_{{\overline{\nabla}}_{\infty}}\left[\left(\theta+2n+2\right)H_{n+2}-nH_{n+1}\right]g\,
d{\rm PD}(\alpha,\theta).
\end{eqnarray}

By (\ref{May20b}) and (\ref{A}), we get
\begin{eqnarray*}
&&\int_{{\overline{\nabla}}_{\infty}} H_{1+\gamma}H_{n}g\,
d{\rm PD}(\alpha,\theta)\\
&=&\frac{1}{\theta+\gamma+n}\bigg\{(\gamma-\alpha)\int_{{\overline{\nabla}}_{\infty}} H_{\gamma}H_{n}g\,
d{\rm PD}(\alpha,\theta)\nonumber\\
&&\ \ \ \ +\int_{{\overline{\nabla}}_{\infty}} \left[\left(\theta+\gamma+2n-1\right)H_{n+\gamma}-(n-1+\gamma-\alpha)H_{n-1+\gamma}\right]g\,
d{\rm PD}(\alpha,\theta)\bigg\}.
\end{eqnarray*}
Then,
\begin{eqnarray}\label{PQ2}
&&\int_{{\overline{\nabla}}_{\infty}} H_{n+1}H_{2}g\,
d{\rm PD}(\alpha,\theta)\nonumber\\
&=&\frac{1}{\theta+2+n}\bigg\{(n-\alpha)\int_{{\overline{\nabla}}_{\infty}} H_{n}H_{2}g\,
d{\rm PD}(\alpha,\theta)\nonumber\\
&&\ \ \ \ +\int_{{\overline{\nabla}}_{\infty}} \left[\left(\theta+n+3\right)H_{n+2}-(n+1-\alpha)H_{n+1}\right]g\,
d{\rm PD}(\alpha,\theta)\bigg\}.
\end{eqnarray}
Thus, by (\ref{May20d}) and (\ref{PQ2}), we obtain
\begin{eqnarray*}
\left(n-1\right)\int_{{\overline{\nabla}}_{\infty}} H_{n+2}g\,
d{\rm PD}(\alpha,\theta)&=&(n-\alpha)\int_{{\overline{\nabla}}_{\infty}}H_{2}H_{n}g\,
d{\rm PD}(\alpha,\theta)\nonumber\\
&&+(\alpha-1)\int_{{\overline{\nabla}}_{\infty}} H_{n+1}g\,
d{\rm PD}(\alpha,\theta),
\end{eqnarray*}
which together with  (\ref{F}) implies that
\begin{eqnarray*}
\int_{{\overline{\nabla}}_{\infty}} H_{n+2}g\,
d{\rm PD}(\alpha,\theta)&=&\frac{1}{\theta+1+n}\int_{{\overline{\nabla}}_{\infty}}\left[\frac{(\theta+2n)(n-\alpha)}{n-1}H_{n+1}-(n-\alpha)H_{n}\right]g\,
d{\rm PD}(\alpha,\theta)\\
&&+\frac{\alpha-1}{n-1}\int_{{\overline{\nabla}}_{\infty}} H_{n+1}g\,
d{\rm PD}(\alpha,\theta)\\
&=&\frac{2n+1+\theta-\alpha}{\theta+1+n}\int_{{\overline{\nabla}}_{\infty}} H_{n+1}g\,
d{\rm PD}(\alpha,\theta)-\frac{n-\alpha}{\theta+1+n}\int_{{\overline{\nabla}}_{\infty}} H_{n}g\,
d{\rm PD}(\alpha,\theta).
\end{eqnarray*}
Further, based on this recursion, we obtain 
\begin{eqnarray}\label{May20f}
&&\int_{{\overline{\nabla}}_{\infty}}  H_{n}g\,
d{\rm PD}(\alpha,\theta)\nonumber\\
&=&\int_{{\overline{\nabla}}_{\infty}}  H_{2}g\,
d{\rm PD}(\alpha,\theta)\nonumber\\
&& + \frac{\Gamma( \theta +3)}{\Gamma(2 - \alpha)} \sum_{k=2}^{n-1} \frac{\Gamma(k - \alpha)}{\Gamma(k + \theta + 1)}\int_{{\overline{\nabla}}_{\infty}}  (H_{3}-H_2)g\,
d{\rm PD}(\alpha,\theta),\ \ \ \ \forall n\ge3.
\end{eqnarray}

Note that
\begin{eqnarray}\label{iop}
\frac{\Gamma(k - \alpha)}{\Gamma(k + \theta + 1)}=\frac{1}{\theta+\alpha}\left[\frac{\Gamma(k-\alpha)}{\Gamma(k+\theta)}-\frac{\Gamma(k+1-\alpha)}{\Gamma(k+\theta+1)}\right],
\end{eqnarray}
and
$$
\frac{\Gamma(k - \alpha)}{\Gamma(k + \theta + 1)}\sim\frac{1}{k^{\theta+\alpha+1}}.
$$
Then, by (\ref{May20f}), we get
\begin{eqnarray*}
&&\lim_{n\rightarrow\infty}\int_{{\overline{\nabla}}_{\infty}}  H_{n}g\,
d{\rm PD}(\alpha,\theta)\\
&=&\int_{{\overline{\nabla}}_{\infty}}  H_{2}g\,
d{\rm PD}(\alpha,\theta) + \frac{\Gamma( \theta +3)}{\Gamma(2 - \alpha)}\sum_{k=2}^{\infty} \frac{\Gamma(k - \alpha)}{\Gamma(k + \theta + 1)}\int_{{\overline{\nabla}}_{\infty}}  (H_{3}-H_2)g\,
d{\rm PD}(\alpha,\theta)\\
&=&\int_{{\overline{\nabla}}_{\infty}}  H_{2}g\,
d{\rm PD}(\alpha,\theta) + \frac{\Gamma( \theta +3)}{\Gamma(2 - \alpha)}\frac{\Gamma(2 - \alpha)}{(\theta+\alpha)\Gamma( \theta +2 )}\int_{{\overline{\nabla}}_{\infty}}  (H_{3}-H_2)g\,
d{\rm PD}(\alpha,\theta)\\
&=&\int_{{\overline{\nabla}}_{\infty}}  H_{2}g\,
d{\rm PD}(\alpha,\theta) + \frac{ \theta+2 }{\theta+\alpha}\int_{{\overline{\nabla}}_{\infty}}  (H_{3}-H_2)g\,
d{\rm PD}(\alpha,\theta).
\end{eqnarray*}
Since $\lim_{n\rightarrow\infty}\int_{{\overline{\nabla}}_{\infty}}  H_{n}g\,
d{\rm PD}(\alpha,\theta)=0$, we obtain
$$
\int_{{\overline{\nabla}}_{\infty}} H_{3}g\,
d{\rm PD}(\alpha,\theta)=\frac{2-\alpha}{\theta+2}\int_{{\overline{\nabla}}_{\infty}} H_{2}g\,
d{\rm PD}(\alpha,\theta),
$$
which together with (\ref{May20f}) and (\ref{iop}) implies that
\begin{eqnarray}\label{May19a}
\int_{{\overline{\nabla}}_{\infty}} H_{n}g\,
d{\rm PD}(\alpha,\theta)=\frac{\Gamma(\theta+2  )\Gamma(n - \alpha)}{\Gamma(2 - \alpha)\Gamma(n + \theta)} \int_{{\overline{\nabla}}_{\infty}} H_{2}g\,
d{\rm PD}(\alpha,\theta),\ \ \ \ \forall n\ge2.
\end{eqnarray}
Define
\begin{eqnarray}\label{May20w}
h=\frac{\int_{{\overline{\nabla}}_{\infty}} H_{2}\,
d{\rm PD}(\alpha,\theta)}{\int_{{\overline{\nabla}}_{\infty}} H_{2}f\,
d{\rm PD}(\alpha,\theta)}f.
\end{eqnarray}
Then, by (\ref{May20g}) and (\ref{May19a}), we obtain 
\begin{eqnarray}\label{May20i}
\int_{{\overline{\nabla}}_{\infty}} H_{n}h\,
d{\rm PD}(\alpha,\theta)&=&\frac{\int_{{\overline{\nabla}}_{\infty}} H_{2}\,
d{\rm PD}(\alpha,\theta)}{\int_{{\overline{\nabla}}_{\infty}} H_{2}f\,
d{\rm PD}(\alpha,\theta)}\frac{\Gamma( \theta+2 )\Gamma(n - \alpha)}{\Gamma(2 - \alpha)\Gamma(n + \theta)} \int_{{\overline{\nabla}}_{\infty}} H_{2}f\,
d{\rm PD}(\alpha,\theta)\nonumber\\
&=&\frac{\Gamma(\theta +2 )\Gamma(n - \alpha)}{\Gamma(2 - \alpha)\Gamma(n + \theta)} \int_{{\overline{\nabla}}_{\infty}} H_{2}\,
d{\rm PD}(\alpha,\theta)\nonumber\\
&=&\int_{{\overline{\nabla}}_{\infty}} H_{n}\,
d{\rm PD}(\alpha,\theta),\ \ \ \ \forall n\ge2.
\end{eqnarray}

Let $$
v=\prod_{r=1}^lH_{n_r},\ \ \ \ n=\sum_{r=1}^ln_r,
$$
where $n_1,\dots,n_l\in\{2,3,\dots\}$  and $l\in\mathbb{N}$. By (\ref{May20a}), we get
\begin{eqnarray}\label{May20k}
{\cal G}^{\gamma,\alpha,\theta}v
&=&\frac{n}{2}\left[(\theta+\gamma+n)H_{1+\gamma}-(\gamma-\alpha)H_{\gamma}\right]\prod_{r=1}^lH_{n_r}\nonumber\\
&&+\frac{1}{2}\sum_{r=1}^ln_r\left[(n_{r}-1+\gamma-\alpha)H_{n_r-1+\gamma}-\left(\theta+\gamma+2n_r-1\right)H_{n_r+\gamma}\right]\prod_{t\not=r}H_{n_t}\nonumber\\
&&+\frac{1}{2}\sum_{r\not=v}{n_r}{n_v}H_{n_r+n_v-1+\gamma}\prod_{t\not=r,v}H_{n_t}-\sum_{r=1}^l{n_r}H_{n_r+\gamma}\sum_{v\not=r}{n_v}\prod_{t\not=r}H_{n_t}.
\end{eqnarray}
Letting $\gamma=1$ and noticing that $H_1=1$, by (\ref{May20i}) and  (\ref{May20k}), we get
\begin{eqnarray*}
\int_{{\overline{\nabla}}_{\infty}}H_{2}H_{n}h\,
d{\rm PD}(\alpha,\theta)=\int_{{\overline{\nabla}}_{\infty}} H_2H_{n}\,
d{\rm PD}(\alpha,\theta),\ \ \ \ \forall n\ge2.
\end{eqnarray*}
Further, by virtue of $\int_{{\overline{\nabla}}_{\infty}}({\cal G}^{\gamma,\alpha,\theta}v)h\,
d{\rm PD}(\alpha,\theta)=\int_{{\overline{\nabla}}_{\infty}}{\cal G}^{\gamma,\alpha,\theta}v\,
d{\rm PD}(\alpha,\theta)=0$, (\ref{May20k}) and using induction, we deduce that
\begin{eqnarray}\label{May20j}
\int_{{\overline{\nabla}}_{\infty}}vh\,
d{\rm PD}(\alpha,\theta)=\int_{{\overline{\nabla}}_{\infty}} v\,
d{\rm PD}(\alpha,\theta),\ \ \ \ \forall v\in{\cal P}^*,
\end{eqnarray}
where
$$
{\cal P}^*=\left\{\prod_{r=1}^lH_{n_r}:n_1,\dots,n_l\in\{2,3,\dots\},l\in\mathbb{N}\right\}.
$$

As pointed out in \cite{{EK2},{S1}}, the following convergence takes place pointwise on $\nabla_{\infty}$:
$$
[H_{n}(x)]^{\frac{1}{n}}=\left(\sum_{i=1}^{\infty}x_i^n\right)^{\frac{1}{n}}\rightarrow x_1\ {\rm as}\ n\rightarrow\infty,
$$
and for $k\in\mathbb{N}$,
$$
\left[H_{n}(x)-\sum_{i=1}^kx_i^n\right]^{\frac{1}{n}}=\left(\sum_{i=k+1}^{\infty}x_i^n\right)^{\frac{1}{n}}\rightarrow x_{k+1}\ {\rm as}\ n\rightarrow\infty.
$$
Then, ${\cal P}^*$ is a measure determining set.
Thus, by (\ref{May20j}), we obtain that
\begin{eqnarray*}
h\,
d{\rm PD}(\alpha,\theta)=d{\rm PD}(\alpha,\theta).
\end{eqnarray*}
Hence, $h=1$. Therefore,  $f$ is a constant by (\ref{May20w}).
\hfill $\Box$

\begin{rem}
It is worth noting that the case $\gamma=1$ is very special. We have
\begin{eqnarray*}
{\cal G}^{1,\alpha,\theta}u(x)=\frac{1}{2}\sum\limits_{i,j=1}^\infty a^{(1)}_{ij}(x)\partial_i\partial_ju(x)-\frac{\theta+2}{2}\sum\limits_{i=1}^\infty x_i\left[x_{i}-H_{2}(x)\right]\partial_iu(x),\ \ \ \ x\in {\overline{\nabla}}_{\infty},
\end{eqnarray*}
which is independent of $\alpha$. Then, by Theorem \ref{thm54}, we find that  each
${\rm PD}(\alpha,\theta)$ is an invariant measure of the generator ${\cal G}^{1,\alpha,\theta}$. However, this does not contradict the ergodicity of the diffusion process \(\{X_t^{1,\alpha,\theta}\}_{t\ge0}\), because \({\rm PD}(\alpha,\theta)\) and \({\rm PD}(\beta,\theta)\) are mutually singular for distinct \(\alpha,\beta\in(0,1)\).
\end{rem}

\subsection{The infinite dimensional labelled model}

For $\mu\in {\cal P}_1(S)$ with $\mu_a=\sum_{i=1}^{\infty}\rho_i\delta_{\xi_i}$, where $\rho=(\rho_1,\rho_2,\dots)\in \overline{{\nabla}}_{\infty}$  and $\xi_i\in S$, $i\ge1$. Define
$$
H_{1+\gamma}(\mu)=H_{1+\gamma}(\rho),\ \ \ \ \mu^{1+\gamma}=\sum_{i=1}^{\infty}\rho_i^{1+\gamma}\delta_{\xi_i},
$$
and for $\phi,\psi\in B_b(S)$,
\begin{eqnarray*}
[\phi,\psi]^{(\gamma)}_{\mu}&=&\langle \phi\psi,\mu^{1+\gamma}\rangle+H_{1+\gamma}(\mu)\langle \phi,\mu_a\rangle\langle \psi,\mu_a\rangle-\langle \phi,\mu^{1+\gamma}\rangle\langle \psi,\mu_a\rangle-\langle \phi,\mu_a\rangle\langle \psi,\mu^{1+\gamma}\rangle\\
&&+\langle \phi,\psi\rangle_{\mu_{na}}.
\end{eqnarray*}
Let  ${\cal F}$ be the space of all cylinder functions, which is given by (\ref{May6Ha}).  Consider the following  symmetric form on $L^2({\cal P}_1(S);\Pi_{\alpha,\theta,\nu_0})$:
\begin{eqnarray*}
{\cal E}^{\gamma,\alpha,\theta}(F,G)=\frac{1}{2}\int_{{\cal P}_1(S)}[ \nabla
F(\mu),\nabla G(\mu)]^{(\gamma)}_{\mu}
\Pi_{\alpha,\theta,\nu_0}(d\mu),\ \ \ \ F,G\in
{\cal F}.
\end{eqnarray*}
Similar to (\ref{May6a}), we can show that $({\cal E}^{\gamma,\alpha,\theta}, {\cal F})$ is non-negative definite.

Assume that $\gamma\ge 1$. For $F(\mu)=f(\langle \phi_1,\mu\rangle,\dots,$ $\langle
\phi_n,\mu\rangle)\in{\cal F}$, define
{\small\begin{eqnarray*}
{\cal L}^{\gamma,\alpha,\theta}F(\mu)&=&
\frac{1}{2}\sum\limits_{i,j=1}^n\partial_i\partial_jf(\langle \phi_1,\mu\rangle,\dots,\langle
\phi_n,\mu\rangle)[\phi_i,\phi_j]^{(\gamma)}_{\mu}-\frac{1}{2}\sum_{i=1}^n\partial_if(\langle \phi_1,\mu\rangle,\dots,\langle
\phi_n,\mu\rangle)\nonumber\\
&&\ \ \cdot\bigg\{(\theta+\gamma+1)\left[\langle\phi_i,\mu^{1+\gamma}\rangle-\langle\phi_i,H_{1+\gamma}(\mu)\mu\rangle\right]-(\gamma-\alpha) \left[\langle\phi_i,\mu^{\gamma}\rangle-\langle\phi_i,H_{\gamma}(\mu)\mu\rangle\right]\bigg\}.
\end{eqnarray*}}

\begin{thm}\label{thm5.2}  Let $\gamma\ge1$, $\alpha\in[0,1)$ and $\theta>-\alpha$.

\noindent {\rm (i)} For any $F,G\in{\cal F}$,
\begin{eqnarray*}
{\cal E}^{\gamma,\alpha,\theta}(F,G)=-\int_{{\cal P}_1(S)}({\cal L}^{\gamma,\alpha,\theta}F) G\,d\Pi_{\alpha,\theta,\nu_0}.
\end{eqnarray*}

\noindent {\rm (ii)}  The bilinear form $({\cal E}^{\gamma,\alpha,\theta}, {\cal F})$ is closable on $L^2({\cal P}_1(S);\Pi_{\alpha,\theta,\nu_0})$ and its closure  $({\cal E}^{\gamma,\alpha,\theta}, D({\cal E}^{\gamma,\alpha,\theta}))$ is a quasi-regular Dirichlet form.

\noindent {\rm (iii)}  There exists a time-reversible, conservative, diffusion process $(\{X^{\gamma,\alpha,\theta}_t\}_{t\ge 0},\{P^{\gamma,\alpha,\theta}_\mu\}_{\mu\in {\cal P}_1(S)})$ which is properly associated with $({\cal E}^{\gamma,\alpha,\theta}, D({\cal E}^{\gamma,\alpha,\theta}))$ and its stationary distribution is  $\Pi_{\alpha,\theta,\nu_0}$.

\noindent {\rm (iv)}  Let $A\in {\cal B}(S)$. If $\nu_0(A)=0$, then there exists an ${\cal E}^{\gamma,\alpha,\theta}$-exceptional set $N_A\subset {\cal P}_1(S)$ such that for any $\mu\notin N_A$,
\begin{eqnarray*}
P^{\gamma,\alpha,\theta}_{\mu}\left\{X^{\gamma,\alpha,\theta}_t(A)=0\ {\rm for\ all}\ t\in[0,\infty)\right\}=1.
\end{eqnarray*}
Moreover, there exists an ${\cal E}^{\gamma,\alpha,\theta}$-exceptional set $N\subset {\cal P}_1(S)$ such that for any $\mu\notin N$,
\begin{eqnarray*}
P^{\gamma,\alpha,\theta}_{\mu}\left\{X^{\gamma,\alpha,\theta}_t\ {\rm is\ purely\ atomic\ for\ all}\ t\in[0,\infty)\right\}=1,
\end{eqnarray*}
and
\begin{eqnarray*}
P^{\gamma,\alpha,\theta}_{\mu}\left\{t\mapsto X^{\gamma,\alpha,\theta}_t\ {\rm is\ continuous\ in\ variation\ norm}\right\}=1.
\end{eqnarray*}
\end{thm}

The proof of Theorem \ref{thm5.2} is similar to that of Theorem \ref{thm3.1}. We omit the details here.

\subsection{The finite dimensional model and approximation}

\subsubsection{The case $\alpha=\frac{1}{2}$}

Assume that $\gamma\ge1$, $\alpha=\frac{1}{2}$ and $\theta>-\frac{1}{2}$. For $x\in\Delta_d$, define
$$
H^{(d)}_{1+\beta}(x)=\sum_{i=1}^{d+1}x_i^{1+\beta},\ \ \ \ \beta\ge0,
$$
and
\begin{eqnarray*}
\left\{
\begin{array}{ll}
a^{(\gamma,d)}_{ii}(x)=(1-2x_i)x_i^{1+\gamma}+x_i^2H^{(d)}_{1+\gamma}(x),&\ \  1\le i\le d,\\
a^{(\gamma,d)}_{ij}(x)=x_ix_j\left[H^{(d)}_{1+\gamma}(x)-x_i^{\gamma}-x_j^{\gamma}\right],&\ \ i\not=j,\ 1\le i,j\le d.
\end{array}
\right.
\end{eqnarray*}
Let $\varsigma_d$ be defined by (\ref{May6Q1}). Consider the following non-negative definite symmetric form on $L^2({\Delta_d}; \varsigma_d(x)dx)$:
\begin{eqnarray*}
{\cal E}^{(\gamma,\frac{1}{2},\theta,d)}(f,g)=\frac{1}{2}\sum_{i,j=1}^{d}\int_{\Delta_d}
a^{(\gamma,d)}_{ij}(x)\partial_if(x)\partial_jg(x)\varsigma_d(x)dx,\ \ \ \
f,g\in  C^{\infty}({\Delta_d}).
\end{eqnarray*}
By direct calculation, similar to \S 4.1, we can show that the informal generator of $({\cal E}^{(\gamma,\frac{1}{2},\theta,d)}, C^{\infty}({\Delta_d}))$ is given by
\begin{eqnarray}\label{May17F}
{\cal L}^{(\gamma,\frac{1}{2},\theta,d)}f(x)&=&
\frac{1}{2}\sum\limits_{i,j=1}^da^{(\gamma,d)}_{ij}(x)\partial_i\partial_jf(x)+\frac{1}{2}\sum\limits_{i=1}^dx_i\bigg\{\left(\gamma+1+\theta\right)\left[H^{(d)}_{1+\gamma}(x)-x_i^{\gamma}\right]\nonumber\\
&&\ \ \ \ +\left(\gamma-\frac{1}{2}\right)\left[x_i^{\gamma-1}-H^{(d)}_{\gamma}(x)\right]+\left(\theta+\frac{d+1}{2}\right)\frac{x_i^{\gamma-2}-\sum_{k=1}^{d+1}x_k^{\gamma-1}}{\sum_{k=1}^{d+1}{x_k}^{-1}}\bigg\}\partial_if(x).\nonumber\\
&&
\end{eqnarray}

Note that
\begin{eqnarray*}
(d+1)^2\sum\limits_{i=1}^d\frac{x_i|\partial_if(x)|}{\sum_{k=1}^{d+1}{x_k}^{-1}}\le (d+1)^2\sup_{1\le i\le d}\|\partial_if\|_{\infty}\frac{\sum_{i=1}^dx_i}{\sum_{k=1}^{d+1}k}\le \frac{2(d+1)}{d+2}\sup_{1\le i\le d}\|\partial_if\|_{\infty}.
\end{eqnarray*}
Then, similar to Section 4, we find that for $\gamma\ge 1$,
$$
{\cal E}^{(\gamma,\frac{1}{2},\theta,d)}(f,g)=-({\cal L}^{(\gamma,\frac{1}{2},\theta,d)}f,g)_{L^2(\Delta_d;\varsigma_d(x)dx)},\ \ \ \ f,g\in  C^{\infty}({\Delta_d}).
$$
Hence, $({\cal E}^{(\gamma,\frac{1}{2},\theta,d)},C^{\infty}({\Delta_d}))$ is closable on $L^2({\Delta_d};\varsigma_d(x)dx)$ and its closure $({\cal E}^{(\gamma,\frac{1}{2},\theta,d)},D({\cal E}^{(\gamma,\frac{1}{2},\theta,d)}))$ is a regular Dirichlet form. There exists a time-reversible, conservative, diffusion process $(\{X^{(\gamma,\frac{1}{2},\theta,d)}_t\}_{t\ge 0}$, $\{P^{(\gamma,\frac{1}{2},\theta,d)}_x\}_{x\in {\Delta_d}})$ which is associated with $({\cal E}^{(\gamma,\frac{1}{2},\theta,d)},D({\cal E}^{(\gamma,\frac{1}{2},\theta,d)}))$ and its stationary distribution  is $\varsigma_d(x)dx$. Let $\{X^{(\gamma,\frac{1}{2},\theta)}_t\}_{t\ge 0}$ be the Markov process in Theorem \ref{thm5.2}. Then, similar to Section 4, we can show that any finite dimensional distribution of $\{X^{(\gamma,\frac{1}{2},\theta,d)}_t\}_{t\ge 0}$  converges to that of $\{X^{(\gamma,\frac{1}{2},\theta)}_t\}_{t\ge 0}$ as $d\rightarrow\infty$ for any $\gamma\ge1$ and $\theta>-\frac{1}{2}$.

\begin{rem}

It follows from (\ref{May17F}) that the convergence of \(\mathcal{L}^{(\gamma,\frac{1}{2},\theta,d)}\) as \(d\rightarrow\infty\) is delicate when \(\gamma\in[0,\frac{1}{2})\cup(\frac{1}{2},1)\). The existence of the infinite dimensional limits and their explicit characterization remain open questions for further study.
\end{rem}

\subsubsection{The case $\alpha=0$}

Assume that $\gamma\ge0$, $\alpha=0$ and $\theta>0$.  Let $d\in\mathbb{N}$ and $(B_1,\dots,B_{d+1})$ be a measurable partition of $S$ with $\nu_0(B_i)=p_i=\frac{1}{d+1}$, $1\le i\le d+1$.  Set $x_{d+1}=1-\sum_{i=1}^dx_i$ and define
\begin{eqnarray*}
\vartheta_d(x_1,\dots,x_d)=\frac{\Gamma(\theta)}{\prod_{i=1}^{d+1}\Gamma(p_i\theta)}\prod_{i=1}^{d+1}x_i^{p_i\theta-1},\ \ \ \ (x_1,\dots, x_d)\in \Delta^{\circ}_d.
\end{eqnarray*}
Then, the joint density function of the random vector $(\mu(B_1),\dots,\mu(B_d))$ under $\Pi_{\alpha,\theta,\nu_0}$ is given by $\vartheta_d$.
Define
\begin{eqnarray*}
V^{(d)}(x)=\left(1-\frac{\theta}{d+1}\right)\sum_{i=1}^{d+1}\log x_i.
\end{eqnarray*}
Then,
$$
\vartheta_d(x)=\frac{e^{-V^{(d)}(x)}}{\int_{\Delta_d}e^{-V^{(d)}(y)}dy},\ \ \ \ x\in \Delta^{\circ}_d.
$$

Consider the following non-negative definite symmetric form on $L^2({\Delta_d}; \vartheta_d(x)dx)$:
\begin{eqnarray*}
{\cal E}^{(\gamma,0,\theta,d)}(f,g)=\frac{1}{2}\sum_{i,j=1}^{d}\int_{\Delta_d}
a^{(\gamma,d)}_{ij}(x)\partial_if(x)\partial_jg(x)\vartheta_d(x)dx,\ \ \ \
f,g\in  C^{\infty}({\Delta_d}).
\end{eqnarray*}
Define
\begin{eqnarray*}
{\cal L}^{(\gamma,0,\theta,d)}f(x)&=&
\frac{1}{2}\sum\limits_{i,j=1}^da^{(\gamma,d)}_{ij}(x)\partial_i\partial_jf(x)+\frac{1}{2}\sum\limits_{i=1}^dx_i\bigg\{\left[H^{(d)}_{1+\gamma}(x)-x_i^{\gamma}\right]\nonumber\\
&&\ \ \ \ +\left(\gamma+\frac{\theta}{d+1}\right)\left[x_i^{\gamma-1}-H^{(d)}_{\gamma}(x)\right]\bigg\}\partial_if(x),\ \ \ \ f\in  C^{\infty}({\Delta_d}).
\end{eqnarray*}
By direct calculation, similar to \S 4.1, we can show that
for $\gamma\ge 1$,
$$
{\cal E}^{(\gamma,0,\theta,d)}(f,g)=-({\cal L}^{(\gamma,0,\theta,d)}f,g)_{L^2(\Delta_d;\vartheta_d(x)dx)},\ \ \ \ f,g\in  C^{\infty}({\Delta_d}).
$$
Then, $({\cal E}^{(\gamma,0,\theta,d)},C^{\infty}({\Delta_d}))$ is closable on $L^2({\Delta_d};\vartheta_d(x)dx)$ and its closure $({\cal E}^{(\gamma,0,\theta,d)},D({\cal E}^{(\gamma,0,\theta,d)}))$ is a regular Dirichlet form. There exists a time-reversible, conservative, diffusion process $(\{X^{(\gamma,0,\theta,d)}_t\}_{t\ge 0}$, $\{P^{(\gamma,0,\theta,d)}_x\}_{x\in {\Delta_d}})$ which is associated with $({\cal E}^{(\gamma,0,\theta,d)},D({\cal E}^{(\gamma,0,\theta,d)}))$ and its  stationary distribution  is $\vartheta_d(x)dx$. Let $\{X^{(\gamma,0,\theta)}_t\}_{t\ge 0}$ be the Markov process  in Theorem \ref{thm5.2}. Then, similar to Section 4, we can show that any finite dimensional distribution of $\{X^{(\gamma,0,\theta,d)}_t\}_{t\ge 0}$  converges to that of $\{X^{(\gamma,0,\theta)}_t\}_{t\ge 0}$ as $d\rightarrow\infty$ for any $\gamma\ge1$ and  $\theta>0$.

\subsection{A remark on the labelled model with $\gamma=0$ and $\alpha=\frac{1}{2}$}

Let $(S,d)$ be a Polish space and $\nu_0$ be a diffuse  probability measure on its Borel $\sigma$-algebra ${\cal B}(S)$. Consider the bilinear form
\begin{eqnarray}\label{May27a}
\left\{
\begin{array}{l}
{\cal E}(F,G)=\frac{1}{2}\int_{{\cal P}_1(S)}\langle \nabla
F(\mu),\nabla G(\mu)\rangle_{\mu}
\Pi_{\alpha,\theta,\nu_0}(d\mu),\ \ \ \ F,G\in
{\cal F},\\
{\cal F}=\{F(\mu)=f(\langle \phi_1,\mu\rangle,\dots,\langle
\phi_n,\mu\rangle): \phi_i\in B_b(S),1\le i\le n,f\in
C^{\infty}(\mathbb{R}^n),n\in\mathbb{N}\},
\end{array}
\right.
\end{eqnarray}
where $\langle \phi
,\psi\rangle_{\mu}=\int_S\phi\psi\, d\mu-\int_S\phi\, d\mu\int_S\psi\, d\mu$ for $\phi,\psi\in B_b(S)$.
Suppose that $\alpha\in(0,1)$. In contrast to the one-parameter case (\(\alpha=0\)) and the regimes  where \(\gamma\in\{\alpha\}\cup[1,\infty)\), the closability of the symmetric form (\ref{May27a}) on \(L^2({\cal P}_1(S);\Pi_{\alpha,\theta,\nu_0})\) remains an unsolved problem. A main difficulty is that, in general,  ${\cal F}$ is not contained in the domain of the generator.

In \cite{FS}, we investigated the
relaxation $(\Xi, D(\Xi))$ of $({\cal E},{\cal F})$, i.e., the greatest lower semi-continuous bilinear form on $L^2({\cal
P}_1(S);\Pi_{\alpha,\theta,\nu_0})$ which is a minorant of $({\cal
E},{\cal F})$.
We have that ${\cal F}\subset D(\Xi)$ and $\Xi(F,F)\le {\cal
E}(F,F)$ for any $F\in {\cal F}$, and for every $F\in D(\Xi)$,
\begin{eqnarray*}
\Xi(F,F)&=&\min\left\{\liminf_{n\rightarrow\infty}{\cal
E}(F_n,F_n):
F_n\in {\cal F}\ {\rm for}\ n\in \mathbb{N},\right.\nonumber\\
& &\ \ \ \ \ \ \ \ \ \ \ \ \ \ \ \left.{\rm and}\
\lim_{n\rightarrow\infty}F_n=F\ {\rm in}\ L^2({\cal
P}_1(S);\Pi_{\alpha,\theta,\nu_0})\right\}.
\end{eqnarray*}
Note that if $({\cal E},{\cal F})$ is closable, then $(\Xi,
D(\Xi))$ is just the closure of $({\cal E},{\cal F})$ on
$L^2({\cal P}_1(S);\Pi_{\alpha,\theta,\nu_0})$.
$(\Xi, D(\Xi))$ is a quasi-regular Dirichlet form
on $L^2({\cal P}_1(S);\Pi_{\alpha,\theta,\nu_0})$ (cf. \cite[Corollary 2.8.2]{M}) and hence is associated with a Markov process $X=\{X_t\}_{t\ge0}$ on ${\cal P}_1(S)$ which is time-reversible with the
stationary distribution $\Pi_{\alpha,\theta,\nu_0}$.

Let $\Phi$ be the map from ${\cal P}_1(S)$ to ${\overline{\nabla}}_{\infty}$ such that $\Phi(\mu)$ is the ordered masses of $\mu\in {\cal P}_1(S)$. We will show that the projection $\Phi(X)$ of $X$ on ${\overline{\nabla}}_{\infty}$ is Petrov's diffusion when $\alpha=\frac{1}{2}$. For $p_i>0$, $1\le i\le d+1$, satisfying $\sum_{i=1}^{d+1}p_{i}=1$, let $\varsigma_d$ be defined by (\ref{May6Q1}). Define
\begin{eqnarray*}
{\cal L}^{(d)}f(x)&=&
\frac{1}{2}\sum\limits_{i=1}^dx_i\partial^2_if(x)-\frac{1}{2}\sum\limits_{i,j=1}^dx_ix_j\partial_i\partial_jf(x)\\
&&+\frac{1}{2}\sum\limits_{i=1}^d\left[-\frac{1}{2}-\theta
x_i+\frac{(\theta+\frac{d+1}{2})\frac{p_i^2}{x_i}}{\sum_{i=1}^d\frac{p_i^2}{x_i}+\frac{p_{d+1}^2}{1-\sum_{i=1}^dx_i}}\right]\partial_if(x),\ \ \ \ f\in C^{\infty}({\Delta_d}),
\end{eqnarray*}
and
$$
{\cal E}^{(d)}(f,g)=\frac{1}{2}\sum_{i,j=1}^d\int_{\Delta_d}
x_i(\delta_{ij}-x_j)\partial_if(x)\partial_jg(x)\varsigma_d(x) dx,\ \
f,g\in C^{\infty}({\Delta_d}).
$$
It is shown in \cite{FS} that
\begin{eqnarray}\label{clo}
{\cal E}^{(d)}(f,g)=-\int_{\Delta_d}{\cal L}^{(d)}f(x)g(x)\varsigma_d(x) dx.
\end{eqnarray}

We fix a sequence
$\{(B^k_1,\dots,B^k_{2^k})\}_{k=1}^{\infty}$ of measurable partitions of $S$ as in \S 4.3 and use the same definitions of the maps $\Gamma_k$ and $\Psi_k$.
For $k\in
\mathbb{N}$, define
$$
{\cal F}_{2^k-1}=\{F(\mu)=f(\langle
1_{B^k_{1}},\mu\rangle,\dots,\langle
1_{B^k_{2^k-1}},\mu\rangle):f\in
C^{\infty}(\mathbb{R}^{2^k-1})\}.
$$
Let $F(\mu)=f(\langle 1_{B^k_{1}},\mu\rangle,\dots,\langle
1_{B^k_{2^k-1}},\mu\rangle)$ with $f\in
C^{\infty}(\mathbb{R}^{2^k-1})$ and $G(\mu)=g(\langle
1_{B^k_{1}},\mu\rangle,\dots,\langle 1_{B^k_{2^k-1}},\mu\rangle)$
with $g\in C^{\infty}(\mathbb{R}^{2^k-1})$. By
(\ref{clo}), we get
\begin{eqnarray*}
{\cal E}(F,G)&=&{\cal
E}^{(2^k-1)}(f,g)\\
&=&-\int_{\Delta_{2^k-1}}{\cal L}^{(2^k-1)}f(x)g(x)\varsigma_{2^k-1}(x) dx\\
&=&\left(-({\cal L}^{(2^k-1)}f)\circ\Gamma_k,G\right)_{L^2({\cal
P}_1(S);\Pi_{\alpha,\theta,\nu_0})}.
\end{eqnarray*}
Then,  $({\cal E},{\cal F}_{2^k-1})$ is closable on $L^2({\cal
P}_1(S);\Pi_{\alpha,\theta,\nu_0})$. By \cite[Corollary 3.3]{FS}, $({\cal E},{\cal F}_{2^k-1})$ converges
to $(\Xi, D(\Xi))$ in the  Mosco convergence.

Denote by $({\cal E}^{(2^k-1)},D({\cal E}^{(2^k-1)}))$ the closure of
$({\cal E}^{(2^k-1)},C^{\infty}({\Delta_{2^k-1}}))$ on
$L^2({\Delta_{2^k-1}},\Pi_{\alpha,\theta,\nu_0}\circ\Gamma_{k}^{-1})$.
Let $(T^{(2^k-1)}_t)_{t\ge 0}$ be the semigroup associated with $({\cal
E}^{(2^k-1)},D({\cal E}^{(2^k-1)}))$  on
$L^2({\Delta_{2^k-1}},\Pi_{\alpha,\theta,\nu_0}\circ\Gamma_{k}^{-1})$ and
${Q}_k$ be the orthogonal projection of $L^2({\cal
P}_1(S);\Pi_{\alpha,\theta,\nu_0})$ on the closure of ${\cal
F}_{2^k-1}$. Define
$$
{\tilde T}^{2^k-1}_tF=(T^{(2^k-1)}_t(({Q}_kF)\circ\Gamma_{k}^{-1}))\circ\Gamma_{k},\ \ F\in L^2({\cal P}_1(S);\Pi_{\alpha,\theta,\nu_0}),\, t\ge0.
$$
Let $(T_t)_{t\ge 0}$ be the strongly continuous contraction
semigroup associated with $(\Xi, D(\Xi))$. By \cite[Theorem 3.4]{FS}, ${\tilde T}^{2^k-1}_t$ converges to $T_t$ in the strong operator topology for each $t\ge0$.

Denote by  $(P_t)_{t\ge 0}$ the semigroup associated with Petrov's diffusion on $L^2({\overline{\nabla}}_{\infty};{\rm PD}(\alpha,\theta))$. Define
$$
{\cal H}=\{H\in C({\overline{\nabla}}_{\infty}):H(x_1,x_2,\dots)=f(x_1,\dots,x_d)\ {\rm for\ some}\ f\in C^{\infty}(\mathbb{R}^d)\ {\rm and}\ d\in\mathbb{N}\}.
$$
In what follows, we will show that
\begin{eqnarray}\label{22}
\lim_{k\rightarrow\infty}{\tilde T}^{2^k-1}_t(H\circ(\Phi\circ\Psi_k))=(P_tH)\circ \Phi\ \ {\rm in}\ L^2({\cal P}_1(S);\Pi_{\alpha,\theta,\nu_0}),\ \ H\in {\cal H},\,t\ge0.
\end{eqnarray}
Once the proof of (\ref{22}) is complete, by the  convergence of $H\circ(\Phi\circ\Psi_k)$ to $H\circ\Phi$ in $L^2({\cal P}_1(S);\Pi_{\alpha,\theta,\nu_0})$,  the contraction property of the  operator ${\tilde T}^{2^k-1}_t$  and the  convergence of ${\tilde T}^{2^k-1}_t$ to $T_t$, we deduce that
$$
T_t(H\circ\Phi)=(P_tH)\circ \Phi\ \ \ \ \Pi_{\alpha,\theta,\nu_0} {\text -a.e.},\ \  H\in {\cal H},\,t\ge0,
$$
which implies that the projection $\Phi(X)$ of $X$ on ${\overline{\nabla}}_{\infty}$ is consistent with Petrov's diffusion.

Define
\begin{eqnarray*}
A=\frac{1}{2}\left\{\sum_{i,j=1}^{\infty}x_i(\delta_{ij}-x_j)\frac{\partial^2}{\partial
x_i\partial x_j} -\sum_{i=1}^{\infty}(\alpha+\theta x_i) \frac{\partial
}{\partial x_i}\right\}.
\end{eqnarray*}
It is known that $A$ is the generator of $(P_t)_{t\ge 0}$ (cf. \cite{P} and \cite{FS0}). For $k\in
\mathbb{N}$, define
\begin{eqnarray*}
A^{({2^k-1})}=\frac{1}{2}\left\{\sum_{i,j=1}^{2^k-1}x_i(\delta_{ij}-x_j)\frac{\partial^2}{\partial
x_i\partial x_j} -\sum_{i=1}^{2^k-1}\left[\alpha+\theta x_i-\frac{(\theta+2^{k-1})\frac{1}{x_i}}{\sum_{i=1}^{2^k-1}\frac{1}{x_i}+\frac{1}{1-\sum_{i=1}^{2^k-1}{x_i}}}\right] \frac{\partial
}{\partial x_i}\right\}.
\end{eqnarray*}
Let
$$
\nabla_{d}=\left\{x=(x_1,\dots, x_d)\in \mathbb{R}^{d}:x_1\ge\cdots \ge x_d\ge0\ {\rm
and}\ \sum_{i=1}^dx_i\le 1\right\},\ \ d\in\mathbb{N}.
$$
Then, $A^{({2^k-1})}$ is the projection of  the operator ${\cal L}^{(2^k-1)}$ on $\nabla_{2^k-1}$.
Each $\nabla_{2^k-1}$ is naturally embedded into ${\overline{\nabla}}_{\infty}$.

For $m\ge2$ and $k\in\mathbb{N}$, define
$$
u_k(x)=\sum_{i=1}^{k}x_i^m,\ \ \ \ x\in \nabla_{{2^k-1}}.
$$
Note that
$$
\frac{\theta+2^{k-1}}{\sum_{i=1}^{2^k-1}\frac{1}{x_i}+\frac{1}{1-\sum_{i=1}^{2^k-1}{x_i}}}\le \frac{\theta+2^{k-1}}{\sum_{i=1}^{2^k-1}{i}}\le\frac{\theta+2^{k-1}}{2^{2k-2}}.
$$
Then, for any $x\in \nabla_{{2^k-1}}$,
\begin{eqnarray*}
&&\left|A^{({2^k-1})}u_k(x)-AH_{m}(x)\right|\\
&=&\frac{m}{2}\left|\sum_{i=1}^{k}\frac{(\theta+2^{k-1})x_i^{m-2}}{\sum_{i=1}^{2^k-1}\frac{1}{x_i}+\frac{1}{1-\sum_{i=1}^{2^k-1}{x_i}}}-\sum_{i=k+1}^{{2^k-1}}\left[(m-1)(1-x_i)-(\alpha+\theta x_i)\right]x_i^{m-1}\right|\\
&\le&\frac{m}{2}\left[\frac{k(\theta+2^{k-1})}{2^{2k-2}}+\sum_{i=k+1}^{\infty}\left(m-1+\alpha+|\theta|\right)x_i^{m-1}\right].
\end{eqnarray*}
Thus,
\begin{eqnarray}\label{Jan20a}
\lim_{k\rightarrow\infty}\max_{x\in \nabla_{{2^k-1}}}\left|A^{({2^k-1})}u_k(x)-AH_{m}(x)\right|=0,\ \ \ \ \forall m\ge 2.
\end{eqnarray}

Note that ${\cal P}$  is a core of $A$ and for any $u,v\in{\cal P}$,
$$
A(uv)=Au\cdot v+Av\cdot u+\langle a(x)\nabla u,\nabla v\rangle,
$$
and
$$
A^{({2^k-1})}(uv)=A^{({2^k-1})}u\cdot v+A^{({2^k-1})}v\cdot u+\langle a(x)\nabla u,\nabla v\rangle,\ \ k\in\mathbb{N},
$$
where $a(x)$ is the infinite matrix whose $(i, j)$-th entry is $x_i(\delta_{ij}-x_j)$.
Then, by (\ref{Jan20a}), we find that for any $u\in{\cal P}$ there exist $u_k\in C^{\infty}(\mathbb{R}^{2^k-1})$, $k\in\mathbb{N}$, such that
$$
\lim_{k\rightarrow\infty}\max_{x\in \nabla_{{2^k-1}}}\left|u_k(x)-H_{m}(x)\right|=0,
$$
and
$$
\lim_{k\rightarrow\infty}\max_{x\in \nabla_{{2^k-1}}}\left|A^{({2^k-1})}u_k(x)-AH_{m}(x)\right|=0.
$$
Let $(S^{2^k-1}_t)_{t\ge0}$ be the semigroup associated with  $A^{({2^k-1})}$ on $C(\nabla_{2^k-1})$. Hence, by \cite[Theorem 1.6.1]{EK}, we deduce that
$$
\lim_{k\rightarrow\infty}\max_{x\in \nabla_{{2^k-1}}}\left|S^{2^k-1}_t(H|_{\nabla_{2^k-1}})(x)-P_tH(x)\right|=0,\ \ \ \  H\in C({\overline{\nabla}}_{\infty}),\,t\ge0.
$$
Therefore, (\ref{22}) holds and the proof is complete.
\hfill $\Box$


\begin{thebibliography}{100}

\bibitem{AMR} S. Albeverio, Z.M. Ma and M. R\"ockner. Quasi-regular Dirichlet forms and Markov processes. J. Funct. Anal. {\bf 111}, 118--154 (1993).

      \bibitem{al85}
      D.J. Aldous. { Exchangeability and Related Topics}. Ecole d'\'Et\'e de Probabilit\'es de Saint Flour, Lecture Notes in Math., Vol. 1117, 1--198. Springer-Verlag (1985).

\bibitem{berj08}
J. Bertoin. Two-parameter Poisson-Dirichlet measures and
reversible exchangeable fragmentation-coalescence processes.  {
Combin. Probab. Comput.}  {\bf 17}, 329--337 (2008).

\bibitem{B}  D. Bertsimas and I. Popescu (2005).  Optimal inequalities in probability theory: a convex optimization approach. Oper. Res. {\bf 53}, 780--804.

\bibitem{BlaMac73}
      D. Blackwell and J.B. MacQueen. Ferguson distribution via P\'olya urn scheme. {
      Ann. Statist.} {\bf 1}, 353--355 (1973).

    \bibitem{B-O09}
A. Borodin and G. Olshanski. Infinite-dimensional diffusions as limits of random walks on
partitions. Probab. Theory Related Fields.  {\bf 144}, 281–318 (2009).

\bibitem{Car} M.A. Carlton.  Applications of the Two-parameter Poisson-Dirichlet Distributions. PhD Thesis (1999).

\bibitem{C} M.A. Carlton. A family of densities derived from the three-parameter Dirichlet process. J. Appl. Probab. {\bf 39},
764--774 (2002).

\bibitem{CF} Z.Q. Chen and M. Fukushima. Symmetric Markov Processes, Time Change, and Boundary Theory. Princeton University Press (2011).

\bibitem{CBERS}
C. Costantini, P. De Blasi, S. N. Ethier, M. Ruggiero and D. Spanò. Wright–Fisher construction
of the two-parameter Poisson–Dirichlet diffusion. Ann. Appl. Probab. {\bf  27}, 1923–1950 (2017).

 \bibitem{Dia-04}
     P. Diaconis, E. Mayer-Wolf, O. Zeitouni and  P.W. Zerner.
       The Poisson-Dirichlet law is the unique invariant distribution for the uniform split-merge transformations.
      Ann. Probab. {\bf 32},   915--938 (2004).

 \bibitem{engen78}
      S. Engen. {  Abundance Models with Emphasis on Biological Communities and Species Diversity}. Chapman-Hall (1978).


\bibitem{Ethier14}
S.N. Ethier. A property of Petrov’s diffusion. Electron. Commun. Probab. {\bf 19}, 1–4 (2014).


\bibitem{EK2} S.N. Ethier and T.G. Kurtz. The infinitely-many-neutral-alleles diffusion model. Adv. Appl. Probab. {\bf 13}, 429--452 (1981).


\bibitem{EK} S.N. Ethier and T.G. Kurtz. Markov Processes: Characterization and Convergence.
John Wiley \& Sons (1986).

\bibitem{Ewen04}
W.J. Ewens. {  Mathematical Population Genetics, Vol. I}. Springer-Verlag (2004).


\bibitem{F} S. Feng. The Poisson-Dirichlet Distribution and Related Topics. Springer (2010).



      \bibitem{FeHo98}
      S. Feng  and F.M. Hoppe. Large deviation principles for
      some random combinatorial structures in population genetics
      and Brownian motion.
      \newblock Ann. Appl. Probab.  {\bf 8},  975--994 (1998).

\bibitem{FS0} S. Feng and W. Sun. Some diffusion processes associated with two parameter
Poisson-Dirichlet distribution and Dirichlet process. Probab. Theory Relat. Fields. {\bf 148},
501--525 (2010).

\bibitem{FS} S. Feng and W. Sun. A dynamic model for the two-parameter Dirichlet process.  Potential Anal. {\bf 51}, 147--164 (2019).

\bibitem{FSWX}
S. Feng, W. Sun,  F.-Y. Wang  and F. Xu. Functional inequalities for the two-parameter
extension of the infinitely-many-neutral-alleles diffusion. J. Funct. Anal. {\bf 260}, 399–413 (2011).




\bibitem{fenwang07}
S. Feng and F.-Y. Wang. A class of infinite-dimensional
diffusion processes with connection to population genetics. {
J. Appl. Probab.} {\bf 44}, 938--949 (2007).

\bibitem{Fer73}
T.S. Ferguson. A Bayesian analysis of some nonparametric
problems. { Ann. Stat.} {\bf 1}, 209--230 (1973).


\bibitem{Forman-etc21}
N. Forman, S. Pal, D. Rizzolo and M. Winkel. Diffusions on a space of interval partitions:
Poisson–Dirichlet stationary distributions. Ann. Probab. {\bf 49}, 793–831 (2021).

\bibitem{Forman-etc23}
 N. Forman, S. Pal, D. Rizzolo and M. Winkel. Ranked masses in two-parameter Fleming-Viot
diffusions. Trans. Amer. Math. Soc. {\bf 376}, 1089–1111 (2023).

 \bibitem{Forman-etc22}
  N. Forman, D. Rizzolo, Q. Shi and M. Winkel.  A two-parameter family of measure-valued
diffusions with Poisson–Dirichlet stationary distributions. Ann. Appl. Probab. {\bf 32},
2211–2253 (2022).

\bibitem{Fuku} M. Fukushima. Dirichlet Forms and Markov Processes. North-Holland (1980).

\bibitem{Fuku2}  M. Fukushima, Y. Oshima and M. Takeda. Dirichlet Forms and Symmetric Markov Processes (2nd rev. and ext. ed.). De Gruyter (2011).

\bibitem{GV17}
S. Ghosal and A. van der Vaart.
{Fundamentals of Nonparametric Bayesian Inference},
Cambridge Series in Statistical and Probabilistic Mathematics 44. Cambridge University Press (2017).

\bibitem{GK01}
A. Gnedin and S. Kerov.  A characterization of GEM distributions. Combin. Probab. Comput. {\bf 10},  213--217 (2001).




\bibitem{Gri80a}
      R.C. Griffiths.
       Lines of descent in the diffusion approximation of neutral Wright-Fisher models.
 Theor. Pop. Biol. {\bf 17}, 35--50 (1980).

    \bibitem{Hop84}
      F.M. Hoppe. Polya-like urns and the Ewens sampling formula.  J.
      Math. Biol. {\bf 20}, 91--94 (1984).


 \bibitem{Kingman75}
      J.C.F. Kingman. Random discrete distributions. J. Roy.
      Statist. Soc. B. {\bf 37}, 1-22 (1975).


      \bibitem{lsy99}
      Z. Li, T. Shiga and L. Yao.
       A reversible problem for Fleming-Viot processes.
      \newblock Elect. Comm. Probab. {\bf 4}, 71--82 (1999).



\bibitem{MR} Z.M. Ma and M. R\"ockner. Introduction to the Theory of (Non-Symmetric) Dirichlet Forms. Springer (1992).

  \bibitem{mcc65}
      J.W. McCloskey. A Model for the Distribution of Individuals by Species in an Environment.
      PhD Thesis (1965).

\bibitem{M} U. Mosco. Composite media and asymptotic Dirichlet forms. J. Funct. Anal. {\bf 123}, 368--421 (1994).

\bibitem{Mu} S. M\"uck. Large deviations w.r.t. quasi-every starting point for symmetric right processes on general state spaces. Probab. Theory Relat. Fields. {\bf 99},
527--548 (1994).

\bibitem{ORS} L. Overbeck, M. R\"ockner and B. Schmuland (1995). An analytic approach to Fleming-Viot processes with interactive selection.  Ann. Probab. {\bf 23}, 1--36.

\bibitem{P}  L.A. Petrov. Two-parameter family of infinite-dimensional diffusions on the Kingman
simplex. Funct. Anal. Appl. {\bf 43}, 279--296 (2009).

\bibitem{Pitman96}
J. Pitman. Random discrete distributions invariant under size-biased permutation.
Adv.  Appl. Probab. {\bf 28}, 525–539 (1996).

  \bibitem{Pitman2002}
     J. Pitman. Poisson-Dirichlet and GEM invariant distributions for split-merge transformations of an interval partition.
       Combin. Probab. Comput. {\bf 11},  501--514 (2002).

\bibitem {PY} J. Pitman and M. Yor. The two-parameter Poisson-Dirichlet distribution derived from a stable subordinator. Ann. Probab. {\bf 25}, 855--900 (1997).


\bibitem{RS} M. R\"ockner and B. Schmuland. Quasi-regular Dirichlet forms: examples and counterexamples. Canadian J. Math. {\bf 47}, 165--200 (1995).

\bibitem{S1} B. Schmuland. A result on  the infinitely many neutral alleles diffusion model. Adv. Appl. Probab. {\bf 28}, 253--267 (1991).

\bibitem{S2} B. Schmuland. On the local property for positivity preserving coercive forms.  In Z.M. Ma, M. R\"ockner, \& J.A. Yan (Eds.), Dirichlet Forms and Stochastic Processes (pp. 345–354). Walter de Gruyter (1995).

\bibitem{S} B. Schmuland. Lecture Notes on Dirichlet Forms. University of Alberta (1995).





\end{thebibliography}
\end{document}